\documentclass[11pt,  reqno]{amsart}

\usepackage{amsmath,amssymb,amscd,amsthm,amsxtra,esint}

\usepackage[implicit=true]{hyperref}
\usepackage[all]{xy}

\usepackage[left=32mm, right=32mm, 
bottom=27mm]{geometry}

\allowdisplaybreaks[2]

\usepackage{url}
\newtheorem{theorem}{Theorem} [section]

\newtheorem{lemma}[theorem]{Lemma}

\newtheorem{remark}[theorem]{Remark}

\newtheorem{definition}[theorem]{Definition}

\usepackage{color}

\definecolor{gr}{rgb}   {0.,   0.69,   0.23 }
\definecolor{bl}{rgb}   {0.,   0.5,   1. }
\definecolor{mg}{rgb}   {0.85,  0.,    0.85}
\definecolor{yl}{rgb}   {0.8,  0.7,   0.}
\definecolor{or}{rgb}  {0.7,0.2,0.2}

\DeclareMathOperator*{\intt}{\int}

\newcommand{\noi}{\noindent}
\newcommand{\Z}{\mathbb{Z}}
\newcommand{\R}{\mathbb{R}}

\newcommand{\T}{\mathbb{T}}

\let\P= \undefined
\newcommand{\P}{\mathbf{P}}

\newcommand{\I}{\mathcal{I}}

\newcommand{\TT}{\mathcal{T}}
\newcommand{\TF}{\mathfrak{T}}
\newcommand{\BT}{{\bf T}}
\newcommand{\SF}{\mathfrak{S}}

\newcommand{\F}{\mathcal{F}}

\newcommand{\al}{ \alpha}

\newcommand{\dl}{\delta}

\newcommand{\eps}{\varepsilon}

\newcommand{\G}{\Gamma}

\newcommand{\s}{\sigma}

\newcommand{\ft}{\widehat}

\newcommand{\wt}{\widetilde}
\newcommand{\cj}{\overline}
\newcommand{\dx}{\partial_x}
\newcommand{\dt}{\partial_t}
\newcommand{\dd}{\partial}

\newcommand{\LRA}{\Longrightarrow}

\renewcommand{\l}{\ell}

\newcommand{\les}{\lesssim}
\newcommand{\ges}{\gtrsim}

\newcommand{\jb}[1]
{\langle #1 \rangle}

\newcommand{\ind}{\mathbf 1}

\renewcommand{\S}{\mathcal{S}}

\newcommand{\M}{\mathcal{M}}

\newtheorem*{ackno}{Acknowledgements}

\numberwithin{equation}{section}
\numberwithin{theorem}{section}

\newcommand{\N}{\mathcal{N}}
\newcommand{\NB}{\mathbb{N}}

\def\NN{{\mathbb{N}}}

\usepackage{tikz}

\makeatletter
\def\DeclareSymbol#1#2#3{\expandafter\gdef\csname MH@symb@#1\endcsname{\tikz[baseline=#2,scale=0.15]{#3}}}
\def\<#1>{\csname MH@symb@#1\endcsname}
\makeatother

\DeclareSymbol{X}{-2.4}{\node[dot] {};}
\DeclareSymbol{1}{0}{\draw[white] (-.4,0) -- (.4,0); \draw (0,0)  -- (0,1.2) node[dot] {};}
\DeclareSymbol{2}{0}{\draw (-0.5,1.2) node[dot] {} -- (0,0) -- (0.5,1.2) node[dot] {};}

\usetikzlibrary{shapes.misc}
\usetikzlibrary{shapes.symbols}
\usetikzlibrary{shapes.geometric}
\tikzset{
	dot/.style={circle,fill=black,draw=black,inner sep=1pt,minimum size=0.5mm},
	>=stealth,
	}
\tikzset{
	ddot/.style={circle,fill=white,draw=black,inner sep=2pt,minimum size=0.8mm},
	>=stealth,
	}

\tikzset{
	dddot/.style={circle,fill=white,draw=black,inner sep=1.2pt,minimum size=0.5mm},
	>=stealth,
	}

\tikzset{dia/.style
={star,fill=white,draw=black,inner sep=1.2pt,minimum size=1mm},
	>=stealth,
	}

\tikzset{dia2/.style
={star,fill=white,draw=black,inner sep=1.8pt,minimum size=1mm},
	>=stealth,
	}

\tikzset{decision/.style={ 
        draw,
        diamond,
        aspect=1.5
    }}

\tikzset{dia3/.style
={diamond,fill=white,draw=black,inner sep=2pt,minimum size=1mm},
	>=stealth,
	}

\tikzset{dia4/.style
={diamond,fill=white,draw=black,inner sep=1pt,minimum size=1mm},
	>=stealth,
	}

\DeclareSymbol{1}{-2.7}
 {
  \draw (0,0) node[dot]{};}

\DeclareSymbol{1'}{-2.7}
 {\draw (0,0) node[dddot]{};}

\DeclareSymbol{1''}{-2.7}
 {\draw (0,0) node[dia]{};}

\DeclareSymbol{1a}{-2.7}
 {\draw (0,0) node[dia4]{};}

\DeclareSymbol{3}{0}
 {\draw (0,0) node[dot]{} -- (0,1.2) node[ddot] {}; 
 \draw (-1.2,0) node[dot] {} -- (0,1.2)node[ddot] {} -- (1.2,0) node[dot] {};}

\DeclareSymbol{3'}{0}
 {\draw (0,0) node[dia]{} -- (0,1.2) node[ddot] {}; 
 \draw (-1.2,0) node[dia] {} -- (0,1.2)node[ddot] {} -- (1.2,0) node[dia] {};}

\DeclareSymbol{31}{-3}{
\draw (0,-1)node[dot] {} -- (0,0) node[ddot] {}-- (0.9, -1) node[dot] {}; 
\draw (0,-1)node[dot] {} -- (0,0) node[ddot] {}-- (-0.9, -1) node[dot] {}; 
\draw (0,0)node[ddot] {} -- (1.3,1) node[ddot] {}-- (1.3, 0) node[dot] {}; 
\draw  (1.3,1) node[ddot] {}-- (2.6, 0) node[dot] {}; 
}

\DeclareSymbol{32}{-3}{
\draw (0,-1)node[dot] {} -- (0,0) node[ddot] {}-- (0.9, -1) node[dot] {}; 
\draw (0,-1)node[dot] {} -- (0,0) node[ddot] {}-- (-0.9, -1) node[dot] {}; 
\draw (0,0)node[ddot] {} -- (0,1) node[ddot] {}-- (0.9, 0) node[dot] {}; 
\draw (0,0)node[ddot] {} -- (0,1) node[ddot] {}-- (-0.9, 0) node[dot] {}; 
}

\DeclareSymbol{33}{-3}{
\draw (2.6,-1)node[dot] {} -- (2.6,0) node[ddot] {}-- (3.5, -1) node[dot] {}; 
\draw (2.6,-1)node[dot] {} -- (2.6,0) node[ddot] {}-- (1.7, -1) node[dot] {}; 
\draw (0,0)node[dot] {} -- (1.3,1) node[ddot] {}-- (1.3, 0) node[dot] {}; 
\draw  (1.3,1) node[ddot] {}-- (2.6, 0) node[ddot] {}; 
}

\DeclareSymbol{3''}{0}
 {\draw (0,0) node[dia]{} -- (0,1.2) node[ddot] {}; 
 \draw (-1.2,0) node[dot] {} -- (0,1.2)node[ddot] {} -- (1.2,0) node[dia] {};}

\DeclareSymbol{3'''}{0}
 {\draw (0,0) node[dot]{} -- (0,1.2) node[ddot] {}; 
 \draw (-1.2,0) node[dot] {} -- (0,1.2)node[ddot] {} -- (1.2,0) node[dia] {};}

\DeclareSymbol{31'}{-3}{
\draw (0,-1.2)node[dia] {} -- (0,0) node[ddot] {}-- (0.9, -1.2) node[dia] {}; 
\draw (0,-1.2)node[dia] {} -- (0,0) node[ddot] {}-- (-0.9, -1.2) node[dia] {}; 
\draw (0,0)node[ddot] {} -- (1.3,1.2) node[ddot] {}-- (1.3, 0) node[dia] {}; 
\draw  (1.3,1.2) node[ddot] {}-- (2.6, 0) node[dia] {}; 
}

\DeclareSymbol{32'}{-3}{
\draw (0,-1.2)node[dia] {} -- (0,0) node[ddot] {}-- (0.9, -1.2) node[dia] {}; 
\draw (0,-1.2)node[dia] {} -- (0,0) node[ddot] {}-- (-0.9, -1.2) node[dia] {}; 
\draw (0,0)node[ddot] {} -- (0,1.2) node[ddot] {}-- (0.9, 0) node[dia] {}; 
\draw (0,0)node[ddot] {} -- (0,1.2) node[ddot] {}-- (-0.9, 0) node[dot] {}; 
}

\DeclareSymbol{33'}{-3}{
\draw (2.6,-1.2)node[dia] {} -- (2.6,0) node[ddot] {}-- (3.5, -1.2) node[dia] {}; 
\draw (2.6,-1.2)node[dia] {} -- (2.6,0) node[ddot] {}-- (1.7, -1.2) node[dia] {}; 
\draw (0,0)node[dot] {} -- (1.3,1.2) node[ddot] {}-- (1.3, 0) node[dot] {}; 
\draw  (1.3,1.2) node[ddot] {}-- (2.6, 0) node[ddot] {}; 
}

\DeclareSymbol{T0}{-2.7}
 { \draw (0,0) node[ddot]{};}

\DeclareSymbol{T1}{0}
 {  
 \draw (0,0) node[dot]{} -- (0,4) node[ddot] {}; 
 \draw (-3,0) node[dot] {} -- (0,4)node[ddot] {} -- (3,0) node[dot] {};
\draw[line width=1pt] (0, 4)node[ddot, label=above:$r_1$]{} 
.. controls(3,6) and (5,6) ..
(8, 4)node[ddot, label=above:$r_2$]{}
(4, 12) node[label ] {$\underline{j = 1}$};
 }

\DeclareSymbol{T2}{0}
 {\draw (0,0) node[dot]{} -- (0,4) node[ddot] {}; 
 \draw (-3,0) node[dot] {} -- (0,4)node[ddot] {} -- (3,0) node[dot] {};
 \draw (0,-4) node[dot]{} -- (0,0) node[dot] {}; 
 \draw (-3,-4) node[dot] {} -- (0,0)node[dot] {} -- (3,-4) node[dot] {};
\draw[line width=1pt] (0, 4)node[ddot, label=above:$r_1$]{} 
.. controls(3,6) and (5,6) ..
(8, 4)node[ddot, label=above:$r_2$]{}
(4, 12) node[label ] {$\underline{j = 2}$};
 }

\DeclareSymbol{T3}{0}
 {\draw (0,0) node[dot]{} -- (0,4) node[ddot] {}; 
 \draw (-3,0) node[dot] {} -- (0,4)node[ddot] {} -- (3,0) node[dot] {};
 \draw (0,-4) node[dot]{} -- (0,0) node[dot] {}; 
 \draw (-3,-4) node[dot] {} -- (0,0)node[dot] {} -- (3,-4) node[dot] {};
\draw (10,0) node[dot]{} -- (10,4) node[ddot] {}; 
 \draw (7,0) node[dot] {} -- (10,4)node[ddot] {} -- (13,0) node[dot] {};
\draw[line width=1pt] (0, 4)node[ddot, label=above:$r_1$]{} 
.. controls(4,6) and (6,6) ..
(10, 4)node[ddot, label=above:$r_2$]{}
(5, 12) node[label ] {$\underline{j = 3}$};
 }

\tikzstyle{dot1} = [ draw=  gray!00, 
 rectangle, rounded corners, fill=gray!00,  inner sep=1pt, inner ysep=3pt]

\tikzstyle{dot2} = [ draw=  black, 
ellipse, fill=gray!00,  inner sep=1pt, inner ysep=3pt]

\tikzstyle{dot3} = [ draw=  gray!00, 
ellipse, fill=gray!00,  inner sep=1pt, inner ysep=3pt]

\DeclareSymbol{T4}{0}
 {\draw (0,0) node[dot1]{$\phantom{n_2^{(1)} = n^{(2)}}$} -- (0,15) node[dot1] {$\phantom{n^{(1)}}$}; 
 \draw (-13,0) node[dot1] {$n_1^{(1)}$} -- (0,15)node[ddot] {$\phantom{n^{(1)}}$} 
 -- (13,0) node[dot1] {$n_3^{(1)}$};
 \draw (0,-15) node[dot1]{$n_2^{(2)}$} -- (0,0) node[dot1] {$\phantom{n_2^{(1)} = n^{(2)}}$}; 
 \draw (-13,-15) node[dot1] {$n_1^{(2)}$} -- (0,0)node[dot1] {$n_2^{(1)} = n^{(2)}$} -- (13,-15) node[dot1] {$n_3^{(2)}$};
\draw (40,0) node[dot1]{{$n_2^{(3)}$}} -- (40,15) node[dot3]  {$\phantom{n^{(1)} = n^{(3)}}$} ; 
 \draw (27,0) node[dot1] {$n_1^{(3)}$} -- (40,15)node[dot3]  {$\phantom{n^{(1)} = n^{(3)}}$}  -- (53,0) node[dot1] {$n_3^{(3)}$};
\draw[line width=1pt] (0, 15)node[ddot, label=above:$r_1$]{$n^{(1)}$} 
.. controls(10,23) and (27,23) ..
(40, 15)node[dot2, label=above:$r_2$] {$n^{(1)} = n^{(3)}$} ;
 }

\usepackage[english, french]{babel}

\begin{document}
\baselineskip = 14pt
\selectlanguage{english}

\title[Normal form approach to unconditional Well-posedness on $\R$]
{Normal form approach to unconditional well-posedness of nonlinear dispersive PDEs on the real line}

\author[S. Kwon]{Soonsik Kwon}
\address{Soonsik Kwon\\ Department of Mathematical Sciences, Korea Advanced Institute of Science and Technology,
291 Daehak-ro Yuseong-gu, Daejeon 34141, Republic of Korea}

\email{soonsikk@kaist.edu}

\author[T. Oh]{Tadahiro Oh}
\address{
Tadahiro Oh, School of Mathematics\\
The University of Edinburgh\\
and The Maxwell Institute for the Mathematical Sciences\\
James Clerk Maxwell Building\\
The King's Buildings\\
Peter Guthrie Tait Road\\
Edinburgh\\ 
EH9 3FD\\
 United Kingdom}
\email{hiro.oh@ed.ac.uk}

\author[H. Yoon]{Haewon Yoon}

\address{Haewon Yoon\\Department of Mathematical Sciences, Korea Advanced Institute of Science and Technology,
291 Daehak-ro Yuseong-gu, Daejeon 34141, Republic of Korea}

\curraddr{National Center for Theoretical Sciences, No. 1 Sec. 4 Roosevelt Rd., National Taiwan University, Taipei, 10617, Taiwan}

\email{hwyoon@ncts.ntu.edu.tw}


\maketitle

\vspace{-3mm}

\begin{abstract}
In this paper, we revisit 
the  infinite iteration scheme of normal form reductions, 
introduced by the first and second authors (with Z.\,Guo), 
in constructing solutions to nonlinear dispersive PDEs.
Our main goal is to present a simplified approach to this method.
More precisely, 
we study   normal form reductions in an  abstract form
and  reduce multilinear estimates of arbitrarily  high degrees
to successive applications of basic  trilinear estimates.
As an application, 
we prove unconditional well-posedness of canonical nonlinear dispersive equations on the real line.
In particular, we implement this simplified approach to an  infinite iteration of normal form reductions
in the context of   the cubic nonlinear Schr\"odinger equation (NLS)
and  the modified KdV equation (mKdV) on the real line
and prove unconditional well-posedness in $H^s(\R)$
with  (i) $s\geq \frac 16$ for the cubic NLS and (ii) $s > \frac 14$ for the mKdV.
Our normal form approach also allows
us to construct weak solutions
to the cubic NLS in $H^s(\R)$, $0 \leq s < \frac 16$, 
and distributional solutions to the mKdV
in $H^\frac{1}{4}(\R)$
(with some uniqueness statements).

\end{abstract}

\begin{otherlanguage}{french}
\begin{abstract}

Dans cet article, nous revisitons
le sch\'ema d'it\'eration infinie des r\'eductions de  forme normale,
introduit par les premier et deuxi\`eme auteurs (avec Z.\,Guo), 
dans la construction des solutions des EDP dispersives non lin\'eaires.
Notre objectif principal est de pr\'esenter une approche simplifi\'ee \`a cette m\'ethode.
Plus pr\'ecis\'ement,
nous \'etudions les r\'eductions de  forme normale dans un cadre abstrait
et nous r\'eduisons les estimations multilin\'eaires de degr\'es arbitraires
aux applications successives des estimations trilin\'eaires 
fondamentales.
Comme application,
nous montrons 
que 
des \'equations 
dispersives non lin\'eaires canoniques 
sont inconditionnellement bien-pos\'ees
sur la droite r\'eelle.
En particulier, nous impl\'ementons cette approche simplifi\'ee 
\`a l'it\'eration infinie des r\'eductions de  forme normale
dans le contexte de l'\'equation de Schr\"odinger non lin\'eaire cubique (NLS)
et de l'\'equation de KdV modifi\'ee (mKdV) sur la droite r\'eelle
et nous prouvons qu'elles sont inconditionnellement bien pos\'ees dans $H^s(\R)$
avec  (i) $s\geq \frac 16$ dans le cas pour NLS cubique et (ii) $s > \frac 14$ dans le cas pour mKdV.
Notre approche de  forme normale nous permet \'egalement de construire
solutions faibles au NLS cubique dans $H^s(\R)$, $0 \leq s < \frac 16$, 
et solutions de distribution au
mKdV dans 
$H^\frac{1}{4}(\R)$
(avec certaine forme d'unicit\'e).

\end{abstract}
\end{otherlanguage}
%


\tableofcontents

\section{Introduction}

\subsection{Main results}

In this paper, we study the Cauchy problem for some canonical nonlinear dispersive equations on the real line. 
More specifically, 
we consider the following cubic nonlinear Schr\"odinger equation (NLS):
\begin{equation}\label{NLS}
\begin{cases}
i \dt u=\dx^2 u\pm |u|^2 u\\
u|_{t = 0} = u_0 \in H^s(\R), 
\end{cases}
\qquad (x,t)\in\R\times\R, 
\end{equation}

\noi
and 
the modified KdV equation (mKdV):
\begin{equation}\label{mKdV}
\begin{cases}
\dt u =\dx^3 u  \pm \dx (u^3)\\
u|_{t = 0} = u_0  \in H^s(\R), 
\end{cases}	
 \qquad (x,t)\in\R\times\R, 
\end{equation}

\noi	
where a solution $u$  is complex-valued in  \eqref{NLS}
and is real-valued in  \eqref{mKdV}.

The Cauchy problems \eqref{NLS} and \eqref{mKdV}
have been studied extensively by many mathematicians.
In particular, multilinear harmonic analysis played an important role
in establishing  well-posedness
of these equations in low regularities. 
Moreover, these equations are known to be 
one of the simplest examples of completely integrable equations
\cite{M1, M2, W, ZS}
and such integrable structures play an important role
in establishing a priori bounds for these equations in the low regularity setting
\cite{KT3, KVZ}.
In the following, however, we will not focus on  such an integrable structure
in an explicit manner.
Our main goal in this paper
is to introduce
a new methodology to establish well-posedness
of \eqref{NLS} and \eqref{mKdV}
(with stronger  uniqueness)
{\it without} relying on heavy harmonic analysis or complete integrability.

Let us briefly go over the well-posedness theory of \eqref{NLS}
and \eqref{mKdV}.
We say that the Cauchy problem \eqref{NLS}
or \eqref{mKdV} is locally well-posed in $H^s(\R)$
if given $u_0 \in H^s(\R)$, there exists
 a unique solution $u \in C([-T, T]; H^s(\R))$
to the equation
for some $T = T(u_0) > 0$. 
Moreover, we impose that the solution map:
$u_0 \in H^s(\R) \mapsto u \in C([-T, T] ; H^s(\R))$
be continuous.
If we can take $T > 0$ to be arbitrarily large, 
we say that the Cauchy problem is globally well-posed.
As we see below, one often needs to employ an auxiliary function space $X_T$
to establish well-posedness.
As a result,  we have uniqueness of solutions
only in $C([-T, T] ; H^s(\R)) \cap X_T$.
In this case, we say that uniqueness holds  conditionally.
If, instead,  uniqueness holds in the entire
$C([-T, T] ; H^s(\R))$, 
then we say that 
the Cauchy problem is {\it unconditionally} (locally) well-posed 
in $H^s(\R)$.
See \cite{Kato}.
Unconditional uniqueness is a notion of uniqueness which does not depend
on how solutions are constructed.
In the following, we summarize the known analytical results
on the well-posedness 
of  \eqref{NLS} and~\eqref{mKdV}.

A basic strategy for proving local well-posedness of \eqref{NLS} or \eqref{mKdV}
is   to write the equation in the Duhamel formulation:\footnote{Here, we only write 
the Duhamel formulation of the cubic NLS \eqref{NLS}.}
\[ u(t) = e^{-it \dx^2} u_0 \mp i \int_0^t  e^{-i(t - t') \dx^2} |u|^2 u (t') dt'\]

\noi
and solve the corresponding fixed point problem.
When $s >  \frac12$, 
Sobolev's embedding theorem
allows us to prove local well-posedness of the cubic NLS \eqref{NLS}
in $H^s(\R)$ via the contraction mapping principle.
In \cite{Tsu}, Tsutsumi used the Strichartz estimates
and proved local well-posedness of \eqref{NLS} in $L^2(\R)$, 
which immediately implied global well-posedness 
in $L^2(\R)$ thanks to the $L^2$-conservation.
Note that the uniqueness holds conditionally  in~\cite{Tsu}
due to the use of the Strichartz spaces.
By refining the analysis, Kato \cite{Kato} proved unconditional
well-posedness of \eqref{NLS} in $H^s(\R)$, $s \geq \frac 16$.

Another approach, inherited from quasilinear hyperbolic problems, 
relies on the energy estimates.
In \cite{KatoKdV}, Kato studied the mKdV \eqref{mKdV} from a viewpoint of a hyperbolic equation
and proved its local well-posedness 
in $H^s(\R)$, $s > \frac 32$.
In Kato's proof,  the dispersive part $\dx^3$ did not play any role.
In \cite{KPV93}, 
Kenig-Ponce-Vega 
exploited the dispersive nature of the equation
in the form of local smoothing and maximal function estimates
and proved local well-posedness of \eqref{mKdV} in $H^s(\R)$, $s \geq \frac 14$,
via the contraction mapping principle.
See also~\cite{Tao} for another proof of the local well-posedness in $H^\frac{1}{4}(\R)$,
utilizing the Fourier restriction norm method (i.e.~the $X^{s, b}$-spaces).
By combining the $X^{s, b}$-spaces with weights 
and the $I$-method, Kishimoto \cite{Kishimoto0}
then proved  global well-posedness of \eqref{mKdV} in $H^\frac{1}{4}(\R)$.\footnote{In \cite{Kishimoto0}, Kishimoto first proved an endpoint local well-posedness
of the KdV equation in $H^{-\frac{3}{4}}(\R)$
and then combined it  with the $I$-method and the Miura transform to establish
global well-posedness of the mKdV~\eqref{mKdV} in 
$H^\frac{1}{4}(\R)$.
In \cite{Guo}, 
Guo independently proved  local well-posedness of the KdV equation in $H^{-\frac{3}{4}}(\R)$.
His argument, however, does not seem to lead to the claimed global well-posedness
as it is presented in \cite{Guo} due to the use of the function spaces
non-compatible with the $I$-method.}
Note that uniqueness in \cite{KPV93, Tao, Guo, Kishimoto0} holds 
only conditionally.
More recently,  by combining the Fourier restriction norm method
and the energy method, 
Molinet-Pilod-Vento \cite{MPV} proved
unconditional well-posedness of \eqref{mKdV}
in $H^s(\R)$, $s> \frac 13$.

In the following, we present a new method for proving well-posedness of 
\eqref{NLS} and \eqref{mKdV} on the real line.
More precisely,  we apply normal form reductions
to the equation infinitely many times
and transform it into an new equation.
While this new equation involves nonlinear terms of arbitrarily high degrees, 
it turns out that these nonlinear terms can be estimated
in a rather straightforward manner 
by successive applications of a basic trilinear estimate
(called a localized modulation estimate)
{\it without} using any auxiliary function spaces
such as the Strichartz spaces and the $X^{s, b}$-spaces.
As a result, we obtain the following unconditional well-posedness
of \eqref{NLS} and \eqref{mKdV}.

\begin{theorem}\label{THM:NLS}
Let $s\geq\frac{1}{6}$.
Then,  the cubic NLS \eqref{NLS} is unconditionally globally  well-posed in $H^s(\R)$.

\end{theorem}

\begin{theorem}\label{THM:mKdV}
Let $s >\frac{1}{4}$.
Then,  the mKdV  \eqref{mKdV} is unconditionally globally  well-posed in $H^s(\R)$.
\end{theorem}

Theorem \ref{THM:mKdV} 
for the mKdV~\eqref{mKdV}
extends the previous unconditional uniqueness result in $H^s(\R)$, $s > \frac 13$, by Molinet-Pilod-Vento 
\cite{MPV} to $s > \frac 14$, thus almost matching the local well-posedness result in $H^\frac{1}{4}(\R)$
\cite{KPV93, Tao}.
On the other hand,  Theorem \ref{THM:NLS} for the cubic NLS~\eqref{NLS} was already proven 
in~\cite{Kato}.
Let us stress, however, that 
the main purpose of this paper  is to introduce a new method 
for constructing solutions to nonlinear dispersive PDEs on the real line
(and on $\R^d$ in general)
via (a simplified approach to) an infinite iteration of normal form reductions,
which can be applied to different classes of dispersive equations.
In our previous work \cite{GKO}, 
we introduced 
an infinite iteration of normal form reductions 
to construct solutions to  nonlinear dispersive PDEs in the periodic setting.
In particular, we proved
 unconditional uniqueness of the cubic NLS on the circle $\T$ 
in $H^s(\T)$, $s\geq \frac 16$.
On the one hand, the present  work can be viewed as an extension of \cite{GKO} to the non-periodic case.
On the other hand,  novelty of this work 
lies in presenting a simplified approach in treating multilinear estimates
appearing in  this normal form approach.

In order to make sense of the cubic nonlinearity in \eqref{NLS} or \eqref{mKdV}
as a distribution, 
we need to have $u \in L^3_\text{loc}(\R)$.
In view of the embedding: $H^\frac{1}{6}(\R)\subset L^3(\R)$, 
we see that $s \geq  \frac 16$ is necessary for proving 
unconditional uniqueness for \eqref{NLS} and \eqref{mKdV}
within the framework of the $L^2$-based Sobolev spaces.
Moreover, it is known that 
the solution map for the mKdV~\eqref{mKdV}:  
$u_0 \in H^s(\R) \mapsto u \in C([-T, T] ; H^s(\R))$
fails to be locally uniformly continuous
for $ s< \frac 14$  \cite{KPV01, CCT}.
Noting that a Picard iteration yields
smoothness of a solution map, 
we see that the regularity restriction $s \geq \frac 14$
is needed to prove local well-posedness of~\eqref{mKdV} via a Picard iteration
(even with conditional uniqueness).
Hence,   Theorems~\ref{THM:NLS} and~\ref{THM:mKdV}
are (almost) sharp in the regime where a Picard iteration is applicable
by some other consideration.
We also point out that well-posedness of \eqref{NLS} for $s < 0$
and (and \eqref{mKdV} for $s < \frac 14$, respectively) is a long-standing open problem.
See \cite{CCT3, KT1, KT2, CHT} for existence results (without uniqueness) below these threshold regularities.

\smallskip

While we need $s \geq \frac 16$ in order to make sense 
of the cubic nonlinearity $\N_{\text{NLS}}(u) := |u|^2u$ as a distribution, 
our normal form argument allows us to establish an a priori
bound on the difference of two (smooth) solutions for the cubic NLS \eqref{NLS} in $L^2(\R)$.
This allows us to establish an existence result of 
certain weak solutions.
Before we state our next result, 
let us  recall the following two notions of weak solutions.

We first recall the notion
of {\it weak solutions in the extended sense}.
See  \cite{Christ2, Christ, GKO}.

\begin{definition}
\label{DEF:sol2}
\rm
Let  $0 \leq s < \frac 16$ and $T>0$. 

\smallskip

\noi
(i) We define a sequence of Fourier cutoff operators to be a sequence of Fourier multiplier operators $\{T_N\}_{N\in \NB}$
on $\mathcal{S}'(\R)$ with multipliers $m_N:\R \to \mathbb{C}$ such that 

\smallskip

\begin{itemize}
\item[$\bullet$]
 $m_N$ has a compact support on $\R$ for each $N \in \NB$, 
 \item[$\bullet$]$m_N$ is uniformly bounded, 

\item[$\bullet$] $m_N$ converges pointwise to $1$, i.e. $\lim_{N\to \infty} m_N(\xi) = 1$ for any $\xi \in \R$.
\end{itemize}

\smallskip

\noi
(ii) 
Let $u \in C([-T, T]; H^s(\R))$. 
We say that $\N_{\text{NLS}}(u)$ exists and is equal to a distribution 
$v \in \mathcal{S}'(\R\times (-T, T))$
if for every sequence $\{T_N\}_{N\in \NB}$ of (spatial) Fourier cutoff operators, we have
\begin{equation*}
\lim_{N\to \infty} \N_{\text{NLS}}(T_N u) = v
\end{equation*}

\noi
in the sense of distributions on $\R\times (-T, T)$.

\smallskip

\noi
(iii) (weak solutions in the extended sense)
We say that $u  \in C([-T, T]; H^s (\R))$  is a weak solution of 
the  cubic NLS \eqref{NLS}
in the extended sense
if 

\smallskip

\begin{itemize}
\item[$\bullet$] $u|_{t=0} = u_0$, 

\item[$\bullet$] the nonlinearity $\N_{\text{NLS}}(u)$ exists in the sense of (ii) above, 

\item[$\bullet$] $u$ satisfies \eqref{NLS}
in the distributional sense  on $\R\times (-T, T)$,
where the nonlinearity $\N_{\text{NLS}}(u) $ is interpreted as above. 

\end{itemize}

\end{definition}

See also \cite{Gub} for a similar notion of weak solutions,
where the nonlinearity is defined as a distributional limit
of smoothed nonlinearities.

Next, we introduce the following notion of {\it sensible weak solutions}.
See  \cite{OW3, OW2, FO}.

\begin{definition}[sensible weak solutions]\label{DEF:sol} \rm
Let $0 \leq s < \frac 16$ and  $T>0$. 
Given $u_0 \in H^s(\R)$, 
we say that
 $u \in C([-T,T]; H^s(\R))$
is a sensible weak solution
to the  cubic NLS \eqref{NLS}  on $[-T, T]$ if,
for any sequence $\{u_{0, m}\}_{m \in \NB}$ of Schwartz   functions
tending to $u_0$ in $H^s(\R)$, 
the corresponding Schwartz  solutions $u_m$ with $u_m|_{t = 0} = u_{0, m}$
converge to $u$ 
in $C([-T,T];  H^s (\R))$.
Moreover, we impose that there exists  a distribution $v$
such that $\N_{\text{NLS}} (u_m)$ converges to  $v$ in the space-time distributional sense,
independent of the choice of the approximating sequence.

\end{definition}

By using the equation, 
the convergence of  $u_m$ to $u$ 
in $C([-T,T]; H^s (\R))$
implies that 
$\N_{\text{NLS}} (u_m)$ converges to some  $v$ in the space-time distributional sense.
Hence, the last part of Definition \ref{DEF:sol}
is not quite necessary.
We, however, keep it for clarity.

Note that  sensible weak solutions are unique by definition.
See \cite{KapT, KV} for analogous notions
of solutions (with uniqueness embedded in the definition).
On the other hand, weak solutions in the extended sense
are not  unique in general.
In fact, Christ \cite{Christ2}
proved non-uniqueness of weak solutions in the extended sense 
for the renormalized cubic NLS on $\T$
in negative Sobolev spaces.
These notions of  weak solutions 
in Definitions \ref{DEF:sol2} and~\ref{DEF:sol}
are  rather weak
and we need to interpret the cubic nonlinearity $\N_{\text{NLS}}(u)$
as a (unique) limit of smoothed nonlinearities $\N_{\text{NLS}}(T_N u)$
or the nonlinearities $\N_{\text{NLS}}(u_m)$ of smooth approximating solutions $u_m$.
This in particular
implies that 
weak solutions 
in the sense of  Definitions \ref{DEF:sol2} or \ref{DEF:sol}
 do not have to satisfy 
the equation even in the distributional sense.

Our normal form approach yields the following result
 without relying on any auxiliary function spaces.

\begin{theorem}\label{THM:NLS2}
Let $0 \leq s < \frac 16$.
Then, the  cubic NLS \eqref{NLS}
 is globally well-posed in $H^s(\R)$

\smallskip
\begin{itemize}
\item[$\bullet$]
in the sense of weak solutions in the extended sense and

\item[$\bullet$]
 in the sense of sensible weak solutions.

\end{itemize}

\smallskip

\end{theorem}

As for the mKdV \eqref{mKdV}, 
our normal form argument provides  an a priori
bound  in $H^\frac{1}{4}(\R)$.
In this regularity, 
the cubic nonlinearity
$ \dx (u^3)$
makes sense as a distribution
and thus  we do not need the notion of weak solutions in the extended sense
in Definition~\ref{DEF:sol2}.
On the other hand, we can define sensible weak solutions
to the mKdV \eqref{mKdV} as in Definition~\ref{DEF:sol}.

\begin{theorem}\label{THM:mKdV2}

The mKdV \eqref{mKdV} is globally well-posed in $H^\frac{1}{4}(\R)$
 in the sense of sensible weak solutions.
These solutions are indeed distributional solutions to \eqref{mKdV}.

\end{theorem}

Note that  solutions constructed in Theorems \ref{THM:NLS2} and \ref{THM:mKdV2}
 agree with those
from the previous well-posedness results in \cite{Tsu, KPV93, Tao}.
This easily follows from the unconditional uniqueness in higher regularities
(for example, in Theorems \ref{THM:NLS} and \ref{mKdV})
and the conditional well-posedness results 
in low regularities \cite{Tsu, KPV93, Tao}, 
which provides uniqueness as a limit of classical solutions.
We point out, however, that  the importance of Theorems \ref{THM:NLS2}
and \ref{THM:mKdV2} does not lie in their statements
but in the method of the construction of solutions.
Our normal form approach
transforms the equations
\eqref{NLS} and \eqref{mKdV}
to the normal form equations (see \eqref{I5} and \eqref{NF2}), 
at least for smooth solutions
belonging to $H^s(\R)$
with $s \geq \frac 16$ 
for the cubic NLS and $s > \frac 14$ for the mKdV.
We then prove unconditional global well-posedness
of the normal form equations
in $H^s(\R)$
with the regularities specified in 
Theorems~\ref{THM:NLS2} and~\ref{THM:mKdV2},
i.e.~$s \geq 0$ 
for the cubic NLS and $s \geq \frac 14$ for the mKdV.
See Theorem~\ref{THM:NF} below.
Then, 
Theorems~\ref{THM:NLS2} and~\ref{THM:mKdV2}
follow as corollaries to 
this unconditional well-posedness 
on the normal form equation.

\medskip

Lastly, note that while Theorems \ref{THM:NLS},  \ref{THM:mKdV},  \ref{THM:NLS2}, and \ref{THM:mKdV2}
claim global-in-time results, 
it suffices to prove these theorems only locally in time 
thanks to  the (conditional) global well-posedness 
\cite{Tsu, CKSTT, Kishimoto0}.
More precisely, in the following, we perform local-in-time construction 
of solutions on a time interval of length $T = T(\|u_0\|_{H^s})>0$
with $s \geq 0$ for the cubic NLS \eqref{NLS}
and $s\geq \frac 14$ for the mKdV \eqref{mKdV}.
Noting that the global well-posedness results in 
\cite{Tsu, CKSTT, Kishimoto0} provide an a priori estimate 
of the form: $\sup_{t \in [-T, T]} \| u(t) \|_{H^s} \les C(\|u_0\|_{H^s}, T)$
for any $T > 0$, 
we can simply iterate the local-in-time argument 
to prove  Theorems \ref{THM:NLS}, \ref{THM:mKdV}, 
 \ref{THM:NLS2}, and \ref{THM:mKdV2}.
Since our analysis is of local-in-time nature, 
the focusing/defocusing nature of the equations does not play any role.
Hence, we assume that the equations are defocusing,
i.e.~ with the $-$ signs in \eqref{NLS} and \eqref{mKdV}.

\begin{remark}\rm

In \cite{OW}, Y.~Wang and the second author
introduced the notion of 
  enhanced uniqueness, 
  which is uniqueness among all solutions (with the same initial data)
  equipped with {\it some} smooth approximating solutions.
They used an  infinite iteration 
of normal form reductions for the fourth order cubic NLS (4NLS)
in negative Sobolev spaces
and proved such enhanced uniqueness.
This notion of enhanced uniqueness allows us to compare
 solutions  belonging to various  auxiliary functions spaces
(so that the cubic nonlinearity makes sense in some appropriate manner).
On the one hand, this notion was useful
in~\cite{OW} since there was no known (conditional) well-posedness
for 4NLS in negative Sobolev spaces at that time.
We point out, however, that such notion becomes useless
once we have (i)~conditional well-posedness in relevant low regularity
and (ii)~unconditional well-posedness in high regularities.
In such a case, this notion of enhanced uniqueness
coincides with uniqueness as a limit of classical solutions.
This is precisely the situation for the cubic NLS and the mKdV under consideration.

\end{remark}


\subsection{Normal form approach}

In this subsection, we briefly explain our strategy  for proving Theorems 
\ref{THM:NLS},  \ref{THM:mKdV}, \ref{THM:NLS2},  and \ref{THM:mKdV2}.
As mentioned above, our main tool is the normal form method.
In particular, we apply normal form reductions to \eqref{NLS} and \eqref{mKdV}
infinitely many times to transform them into new  equations.
These new equations 
involve infinite series of nonlinearities of arbitrarily high degrees
and thus are more complicated algebraically than the original equations.
As we see later, however 
they are easier to handle analytically.
Namely, we renormalize the equations
into analytically simpler equations
at the expense of algebraic
and notational complexity.

In the following, we consider  the cubic NLS \eqref{NLS} as an example.
Letting $v(t) = e^{i t \dx^2} u(t) $ denote the interaction representation of $u$, 
we can rewrite the equation \eqref{NLS} as\footnote{For simplicity of the exposition, 
we drop the complex conjugate sign on $\ft v(\xi_2)$.}  
\begin{align}
\partial_t  v=  \N(v)
&: =   \F^{-1}\Bigg\{i 
\intt_{\xi = \xi_1-\xi_2+\xi_3}e^{-i\Phi(\bar{\xi})t} \prod_{j = 1}^3 \ft v(\xi_j, t)d\xi_1d\xi_2
\Bigg\},
 \label{I1}
\end{align}

\noi where the modulation function\footnote{In \cite{Tao}, this phase function is referred to as
a resonance function. For our analysis, resonance does not play any important role.
Instead, modulation (as in the Fourier restriction norm method) plays
an important role.
For this reason, we refer to $\Phi(\bar \xi)$ as a modulation function.}
 $\Phi(\bar{\xi})$ is defined by
\begin{align}
\Phi(\bar{\xi})
 =\Phi(\xi,\xi_1,\xi_2,\xi_3)
& =\xi^2-\xi_1^2+\xi_2^2-\xi_3^2\notag \\
& =2(\xi_2-\xi_1)(\xi_2-\xi_3)=2(\xi-\xi_1)(\xi-\xi_3).
\label{Phi1}
\end{align}

\noi
Note that the last two equalities hold under the condition $\xi=\xi_1-\xi_2+\xi_3$.
We point out that 
it is natural to consider the equation in terms of the interaction representation
if we want to exploit the oscillatory factor $e^{-i\Phi(\bar{\xi})t}$ in \eqref{I1}.
Such a formulation in terms of the interaction representation is classical 
and already appears in the work of
Kato \cite{KatoKdV} in the context of the (generalized) KdV equation.
By integrating \eqref{I1} in time, we obtain 
\begin{align}
v(t) =  u_0 + \int_0^t \N(v)(t') dt'.
\label{I2}
\end{align}

On the one hand, 
when $ s > \frac 12$, we can easily estimate \eqref{I2}
by the algebra property of $H^s(\R)$.
On the other hand, when $s \leq \frac 12$, we must exploit
the dispersion, namely, the oscillation coming from the 
oscillatory factor $e^{-i\Phi(\bar{\xi})t}$ in \eqref{I1}.
This is often manifested  in the form of the Strichartz estimates
and/or the Fourier restriction norm method.
In the following, we simply rely on integration by parts.
By taking the spatial Fourier transform of \eqref{I2}
and (formally) integrating by parts,\footnote{In fact, 
this integration by parts basically corresponds to the (Poincar\'e-Dulac) normal form reduction.
See the introduction of \cite{GKO} by the first two authors (with Z.\,Guo), relating the integration-by-parts 
(or differentiation-by-parts)
procedure with the normal form reductions.
See Arnold \cite{A} for a general discussion of the Poincar\'e-Dulac normal form reductions in the finite dimensional setting.} we have
\begin{align}
 \ft  v(\xi) = \ft u_0(\xi)
& -  \intt_{\xi = \xi_1-\xi_2+\xi_3}\frac{e^{-i\Phi(\bar{\xi})t'}}{\Phi(\bar \xi)}
\prod_{j = 1}^3 \ft v(\xi_j, t')d\xi_1d\xi_2\bigg|_{t' = 0}^t  \notag\\
& + \int_0^t 
\intt_{\xi = \xi_1-\xi_2+\xi_3}
\frac{e^{-i\Phi(\bar{\xi})t'}}{\Phi(\bar \xi)}\dt \bigg(\prod_{j = 1}^3 \ft v(\xi_j, t')\bigg)d\xi_1d\xi_2dt' .
 \label{I3}
\end{align}

\noi
Note that we have gained a full power of the modulation thanks to  $\Phi(\bar \xi)$ 
in the denominator.
Compare this with the usual application of 
 the Fourier restriction norm method
where one only gains  $\sim \frac 12$-power of the modulation.

At this point, there are several issues in \eqref{I3}.
First, note that the modulation function $\Phi(\bar \xi )$ appearing in the denominator may be 0.
This corresponds to the so-called {\it resonance}.
Even if $\Phi(\bar \xi ) \ne 0$, integration by parts 
does not seem to help if $\Phi(\bar \xi )$ is small, corresponding to the {\it nearly resonant} case.
In order to resolve this issue, 
we separately estimate the contributions 
from (i) nearly resonant case: $|\Phi(\bar \xi )| \leq N$
and (ii) (highly) non-resonant case: $|\Phi(\bar \xi )| > N$
for some parameter $N = N(\|u_0\|_{H^s}) > 1$.
In particular, we perform integration by parts only in the non-resonant case (ii).
Thanks to the restriction on the modulation, 
we can estimate the contribution from the nearly resonant case (i) in $C_t H^s_x$, $s \geq 0$
(and $s \geq \frac 14$ for the mKdV),  in a straightforward manner.
See Lemmas \ref{LEM:NLS1} and \ref{LEM:KdV1}.


The second issue is that we have increased the degree of the nonlinearity
in \eqref{I3}.
 In view of \eqref{I1}, 
 the last term in \eqref{I3} is now quintic.
 Indeed, by assuming that the time derivative falls on the first factor, 
  we can write the last term in \eqref{I3} as 
\begin{align}
&  \sim  \int_0^t 
\intt_{\xi = \xi_1-\xi_2+\xi_3}
\frac{e^{-i\Phi(\bar{\xi})t'}}{\Phi(\bar \xi)} \ft \N(v) (\xi_1, t')\prod_{j = 2}^3 \ft v(\xi_j, t')d\xi_1d\xi_2dt' \notag \\
&  \sim  \int_0^t 
\intt_{\substack{\xi = \xi_1-\xi_2+\xi_3\\
\xi_1 = \zeta_1 - \zeta_2 + \zeta_3}}
\frac{e^{-i(\Phi(\bar{\xi}) + \Phi(\bar \zeta))t'}}{\Phi(\bar \xi)}
\prod_{k = 1}^3 \ft v(\zeta_k, t')\prod_{j = 2}^3 \ft v(\xi_j, t')d\zeta_1d\zeta_2 d\xi_1d\xi_2dt',
\label{I4}
\end{align}

\noi
where $\Phi(\bar{\zeta}): =\Phi(\xi_1,\zeta_1,\zeta_2,\zeta_3)$.
The main idea is to perform integration by parts once again.
In order to exploit the oscillation of
$e^{-i(\Phi(\bar{\xi}) + \Phi(\bar \zeta))t'}$, 
we separately estimate the contributions 
from (i) nearly resonant case: $|\Phi(\bar \xi )+ \Phi(\bar \zeta )| \leq N_1$
and (ii) non-resonant case: $|\Phi(\bar \xi ) + \Phi(\bar \zeta )| > N_1$
for some suitable threshold $N_1 > 1$.\footnote{As we see later, we choose $N_1 \sim |\Phi(\bar \xi )|^{1-\dl}$
for some $\dl \in (0, 1)$.  See \eqref{C1}.}
Then, we integrate \eqref{I4} by parts only in the non-resonant case (ii),
thus introducing a septic nonlinearity.

By formally iterating this procedure indefinitely, 
we arrive at the  following  {\it normal form equation}:
\begin{align} 
v(t)  =u_0 + \sum_{j=2}^\infty  \N_0^{(j)} (v(t')) \bigg|_{t' = 0}^t
+ \int_0^t  \sum_{j=1}^\infty \N_1^{(j)}(v(t')) dt',
\label{I5}
\end{align}

\noi
where $\N_0^{(j)}(v)$ and $\N_1^{(j)}(v)$
are $(2j-1)$- and $(2j+1)$-multilinear terms,  respectively.
See~\eqref{NF2} below.
These multilinear terms  $\N_0^{(j)}(v)$ and $\N_1^{(j)}(v)$ appear
as a result of $(j-1)$-many iterations
of the normal form reductions.
Then, the main task is to estimate each term of the infinite series
in \eqref{I5} 
in the $C_t H^s_x$-norm
in a summable manner.
There are, however, three potential difficulties:
\begin{enumerate}
\item  The degrees of the nonlinearities
can be arbitrarily high.

\item In performing integration by parts in the $J$th step, 
the number of factors on which the time derivative falls
is $2J+1$.
Thus, the constants grow like $3\cdot5\cdot7\cdot \cdots \cdot (2J+1)$.

\item Our multilinear estimates need to provide small constants
on the terms {\it without} time integration,
i.e.~on the boundary terms, such as 
the second term on the right-hand side of  \eqref{I3} and $\N_0^{(j)} (v)$ in \eqref{I5}.
(We can introduce small constants for the terms
inside time integration by making the time interval of integration 
sufficiently short.)

\end{enumerate}

\noi
In Section \ref{SEC:3}, we will treat these issues and prove that the normal form equation
is unconditionally well-posed (Theorem \ref{THM:NF}).
Theorems \ref{THM:NLS}, \ref{THM:mKdV}, \ref{THM:NLS2}, and \ref{THM:mKdV2}
then follow as  corollaries to this unconditional well-posedness of the normal form equation.

In \cite{GKO}, we implemented an infinite iteration of normal form reductions
sketched above
in the context of the cubic NLS on the circle $\T$. 
In particular, we introduced the notion of ordered trees
(see Definition \ref{DEF:tree2})
and indexed all the multilinear terms by 
such ordered trees, 
handling the issues (1), (2), and (3).
Moreover, in handling the multilinear estimates, 
we exploited the discrete structure of the spatial frequency
space $\Z = (\T)^*$ in the form of the divisor counting argument.
In the non-periodic setting, 
such number theoretic tools are no longer available.
In this paper, we change our viewpoint
and view  these multilinear terms
as iterative compositions of trilinear operators (see Definition \ref{DEF:S} and \eqref{X2})
with  modulation restrictions.
We first establish trilinear localized modulation  estimates in Section \ref{SEC:2}
as a fundamental building block.
Then, by applying  such trilinear localized modulation estimates
in an iterative  manner, 
we estimate the multilinear terms of arbitrarily high degrees,
appearing in \eqref{I5}. 
This provides a simplified framework
for implementing an infinite iteration of normal form reductions.\footnote{During the preparation of 
 this manuscript, we learned that Kishimoto \cite{Kishimoto5, Kishimoto6}
independently 
used a similar abstraction of a basic multilinear estimate as a fundamental building block
in the application of an infinite iteration of normal form reductions 
to prove unconditional well-posedness for various dispersive PDEs in the periodic setting.}

Lastly, let us mention the role of two different topologies
for this normal form argument.
Roughly speaking, 
we 
\begin{itemize}
\item[(i)] establish a priori estimates in a stronger topology
(in $H^s(\R)$ with $s\geq 0$ for the cubic NLS and $s \geq \frac 14$ for the mKdV)
but

\smallskip

\item[(ii)] justify all the formal computations in a weaker topology
(in the Fourier-Lebesgue space $\F L^\infty(\R)$ defined in \eqref{FL} below)
for {\it smoother} solutions ($s \geq \frac 16$ for the cubic NLS
and $s > \frac 14$ for the mKdV), 
thus making sense of the identity  \eqref{I5} in the distributional sense.

\end{itemize}

\noi
By  formally performing an infinite iteration of normal form reductions, 
we derive the normal form equation \eqref{I5}  in Section~\ref{SEC:3}.
In establishing a priori estimates in $H^s(\R)$, 
we estimate each multilinear term in the $H^s$-norm
with $s \geq 0$ for the cubic NLS and $s\geq \frac 14$ for the mKdV.
In Section~\ref{SEC:4}, 
we justify all the formal computations performed in Section~\ref{SEC:3}, 
in particular the integration-by-parts steps,
where we switch time derivatives and integrations over spatial frequencies.
See \eqref{I3} for example.
For this purpose, we work in a weaker topology.
Indeed, we justify all the steps of the normal form reductions
for each {\it fixed} frequency $\xi \in \R$ of the interaction representation $\ft v(\xi)$
(and hence of each multilinear term in \eqref{I5}).
It is in this step where we need to 
 assume a higher regularity: $s\geq \frac 16$ 
 for the cubic NLS
 and $s> \frac 14$ for the mKdV.
 In the case of the mKdV, 
we also need to handle the derivative loss in the equation.
In particular, in each step of the normal form reductions
(i.e.~integration by parts), 
we use the equation \eqref{mKdV} to replace $\dt \ft v$ by the cubic nonlinearity (see \eqref{mKdV2}),
which introduces a derivative loss at each step.
Since we work for each fixed $\xi \in \R$, 
the derivative loss in the first ``generation'' (i.e.~in the original equation)
does not cause any problem.
We then  {\it shift
part of the derivative loss up by one generation}
to reduce the derivative loss in the last generation.
See Subsection \ref{SUBSEC:mKdV} for a further discussion.

\subsection{Remarks \& comments}
A precursor to this normal form approach  appeared in the work of
Babin-Ilyin-Titi \cite{BIT} for the KdV on $\T$, establishing unconditional well-posedness of the KdV in $L^2(\T)$.
See also \cite{KO} for an analogous unconditional well-posedness result for the periodic mKdV in $H^\frac{1}{2}(\T)$.
Note that two iterations were sufficient in \cite{BIT, KO}.
In \cite{GKO}, we further developed this normal form approach
and introduced an infinite iteration scheme of normal form reductions
in the context of the cubic NLS on the circle.
By performing normal form reductions infinitely many times, 
we proved unconditional well-posedness of the periodic cubic NLS in $H^{\frac 16}(\T)$.
In this series of work, 
 the viewpoint of unconditional well-posedness
 was first introduced in \cite{KO}, 
while the viewpoint of the (Poincar\'e-Dulac) normal form reductions
was first introduced in \cite{GKO}.

More recently, 
by combining an infinite iteration of normal form reductions
and the Cole-Hopf transform, 
we proved unconditional global well-posedness for  the quadratic derivative NLS on $\T$ 
for small mean-zero initial data \cite{CGKO}.
Moreover, this method allowed us to construct an infinite sequence of 
invariant quantities under the dynamics.
Kishimoto \cite{Kishimoto5}
adapted our infinite iteration approach 
and  proved unconditional well-posedness
for  higher dimensional NLS, 
the  Zakharov system on $\T^d$, $d = 1, 2$,  
the derivative cubic NLS on $\T$, 
the Benjamin-Ono and modified Benjamin-Ono equations in the periodic setting.

One may naturally expect that an infinite iteration of normal form reductions
is needed to prove Theorem \ref{THM:NLS} for the cubic NLS on the real line 
just as in the periodic case \cite{GKO}.\footnote{We point out recent works \cite{Pat, CHKP}
 on the construction of solutions to the cubic NLS on the real line via an infinite iteration of
 normal form reductions.  
 Their implementation of normal form reductions
  follows closely  the original argument in \cite{GKO}
 and unconditional uniqueness in modulation spaces 
 (including Theorem~\ref{THM:NLS} above) is established.
 In \cite{FO}, this construction was extended to almost critical Fourier-Lebesgue and modulation spaces.}
It is, however, to our surprise to see that  we also need to perform normal form reductions
infinitely many times in proving Theorem \ref{THM:mKdV}
for the mKdV on the real line.
This is in sharp contrast with the mKdV on the circle, 
where two iterations were sufficient \cite{KO}.
In this paper, we chose to study the cubic NLS \eqref{NLS} and the mKdV \eqref{mKdV}
as canonical examples.
As in the periodic case \cite{Kishimoto5}, 
our method of an  infinite iteration scheme of normal form reductions
is fairly general that it can be applied 
to study a wide variety of  equations in the Euclidean space $\R^d$ of general dimensions.

This normal form approach  has various applications
beyond  establishing unconditional uniqueness.
It has been used to exhibit nonlinear smoothing \cite{ET},
to prove a good approximation property in proving symplectic non-squeezing \cite{HK}, 
and establishing effective energy estimates with smoothing
in proving quasi-invariance of Gaussian measures on periodic functions under dispersive PDEs \cite{OT}.\footnote{Such an application of  normal form reductions in energy estimates is more classical and precedes the work of \cite{BIT}.}
More recently, the second author introduced a way to perform
normal form reductions infinitely many times in establishing 
energy estimates \cite{OST, OW}.
In particular, the notion of ordered trees was extended to that of ordered bi-trees
to accommodate normal form reductions on energy quantities.
Note that such an infinite iteration of normal form reductions
on an energy quantity basically amounts
to adding infinitely many correction terms in 
the $I$-method terminology, 
going far beyond the known application of the $I$-method  \cite{CKSTT}, 
where only finitely many correction terms were considered.

The main novelty of this paper is to reduce multilinear estimates to successive applications
of a basic trilinear localized modulation estimate and in fact 
to reduce the entire problem of proving unconditional well-posedness
to simply proving two basic trilinear estimates
(i.e.~localized modulation estimates in the strong norm
and in the weak norm: Lemmas~\ref{LEM:NLS1} and \ref{LEM:NLS3}
for the cubic NLS 
and Lemma~\ref{LEM:KdV1} and \ref{LEM:mk1}
for the mKdV).
Such reduction can easily be implemented
in the context of  our previous work 
\cite{GKO, OST, OW}, except for \cite{CGKO}
where the algebraic property of the equation played an important role
inducing cancellation of resonant terms via symmetrization
at each step of the normal form reductions.
See also the concluding remark
at the end of this paper.

\medskip

This  paper is organized as follows.
 In Section~\ref{SEC:2} we establish
 crucial trilinear estimates (localized modulation estimates)
 for the cubic NLS \eqref{NLS} and the mKdV \eqref{mKdV}.
  In Section~\ref{SEC:3}, we  perform an infinite iteration of normal form reductions
  and derive the normal form equation. 
We carry out  a computation in Section~\ref{SEC:3} at a formal level.
In Subsection \ref{SUBSEC:3.4}, we prove unconditional local well-posedness
of the normal form equation
in $H^s(\R)$
with $s \geq 0$ for the cubic NLS and $s \geq \frac 14$ for the mKdV
(Theorem~\ref{THM:NF})
and discuss the proofs of Theorems \ref{THM:NLS}, \ref{THM:mKdV},
\ref{THM:NLS2}, and \ref{THM:mKdV2}, assuming that smooth solutions
satisfy the normal form equation.
   In Section~\ref{SEC:4}, we justify the formal computation in Section~\ref{SEC:3}
   and then conclude the proofs of the main theorems.

\medskip

\noi
{\bf Notations.}
We use $A\les B$ to denote 
an estimate of the form  $A\leq CB$ for some constant $C>0$,
 which may vary from line to line and depend on various parameters.
We also use 
 $A\sim B$ to denote $A\les B\les A$, 
 while 
we use $A\ll B$ to denote  $A\leq \eps B$ for some small absolute constant $\eps >0$.
We  use $a+$  to denote $a + \eps$  for arbitrarily small $\eps \ll 1$,
where an implicit constant is allowed to depend on $\eps > 0$
(and it usually diverges as $\eps \to 0$).

Given a function $f$ on $\R$, we 
define its Fourier transform by 
\[\F(f)(\xi) = \ft f(\xi)= \int_\R f(x)e^{-2\pi i x \xi} dx .\]

\noi
We drop the harmless factor of $2\pi$ in the following.
We define the Fourier-Lebesgue space $\F L^p(\R)$, $p \geq 1$, 
by the norm:
\begin{align}
\| f\|_{\F L^p} = \| \ft f\|_{L^p}.
\label{FL}
\end{align}

Any summation over capitalized variables such as  $N_1$, $N_2$, $\cdots$ are presumed to be dyadic, 
i.e.~these variables range over dyadic numbers of the form $2^k$, $k\in\Z_{\geq 0}$.
We also use the following shorthand notations:
 $\xi_{ij}$ and $\xi_{i - j}$ for $\xi_i+\xi_j$ and $\xi_i-\xi_j$, respectively.
 
Given dyadic $N \geq1$, 
we use $\P_N$ to denote the Littlewood-Paley projector
onto the spatial frequencies $\{|\xi|\sim N\}$.
Given $k \in \Z$, we use   $\Pi_k$ to denote the (spatial) frequency projector onto the interval $[k, k+1)$:
\begin{align}
\Pi_kv(\xi)  = \ind_{[k, k+1)}(\xi) \cdot v(\xi).
\label{K9a}
\end{align}

We use $S(t)$ to denote the  linear propagator
for the linear Schr\"odinger equation: $  i \dt u = \dx^2 u $
and the Airy equation: $\dt u = \dx^3 u$, depending on the context.
Namely, 
 $S(t) = e^{-it \dx^2}$ for the linear Schr\"odinger equation
and $S(t) = e^{t \dx^3}$ for the Airy equation.
Then, 
given a function $u$ on $\R\times \R$, we define its interaction representation $v$
by 
\begin{align}
v(t) = S(-t) u(t).
\label{interaction}
\end{align}

\noi
We mainly perform our analysis in terms of the interaction representation.

In the following, we only consider positive times for the simplicity of the presentation.

\section{Localized modulation estimates}\label{SEC:2}
In this section, we establish 
 crucial trilinear estimates (called  localized modulation estimates)
for the cubic NLS \eqref{NLS} and the mKdV \eqref{mKdV}.
See Lemmas \ref{LEM:NLS1} and \ref{LEM:KdV1}.
While their proofs are very elementary, 
these trilinear estimates constitute a fundamental building block
for multilinear estimates on the nonlinear terms (of arbitrarily high degrees)
appearing in the normal form reductions in Section \ref{SEC:3}.

\subsection{Localized modulation estimates for the cubic NLS}

We first consider the cubic NLS \eqref{NLS}.
On the Fourier side, we write  \eqref{NLS} as
\begin{align}
i\partial_t\widehat{u}(\xi)=-\xi^2\widehat{u}(\xi)
-\intt_{\xi = \xi_1-\xi_2+\xi_3}\widehat{u}(\xi_1)\overline{\widehat{u}(\xi_2)}\widehat{u}(\xi_3)
d\xi_1d\xi_2.
\label{NLS1}
\end{align}

\noi
Let $v(t) = S(-t)u(t)$ be the interaction representation defined in \eqref{interaction}.
Then, we have 
 $\ft v(\xi,t)=e^{-i\xi^2t}\widehat{u}(\xi,t)$.
Define a trilinear operator $\N(v_1, v_2, v_3)$ by 
\begin{align}
\ft \N (v_1,v_2,v_3)(\xi,t)
&: =  
i\intt_{\xi = \xi_1-\xi_2+\xi_3}e^{-i\Phi(\bar{\xi})t}\ft v_1(\xi_1)\cj{\ft v_2(\xi_2)}\ft v_3(\xi_3)d\xi_1d\xi_2, 
\label{NLS1a}
\end{align}

\noi where the modulation function
 $\Phi(\bar{\xi})$ is as in \eqref{Phi1}.
With this notation, we can  write \eqref{NLS1} as 
\begin{equation}
\label{NLS2}
\partial_tv= \N(v, v, v). 
\end{equation}

\begin{remark} \label{REM:tdepend} \rm
(i)
When there is no confusion, 
 we simply denote  $\ft v(\xi,t)$ 
 and $\ft \N(v_1, v_2, v_3)$
 by  $ v(\xi,t)$
 and $ \N(v_1, v_2, v_3)$ in the following.
 For example, we write \eqref{NLS1a} as 
\begin{align*}
 \N (v_1,v_2,v_3)(\xi,t)
 =  
i\intt_{\xi = \xi_1-\xi_2+\xi_3}e^{-i\Phi(\bar{\xi})t} v_1(\xi_1)\cj{ v_2(\xi_2)} v_3(\xi_3)
d\xi_1d\xi_2
\end{align*}

\noi
under this convention.
Note that while the equation \eqref{NLS2}
can be interpreted 
as an equation
on the physical side or on the Fourier side
under this convention, 
this does not cause any confusion in terms of its meaning.
A similar comment applies to other multilinear operators.

\smallskip

\noi
(ii)
Due to the presence of 
the time-dependent phase factor $e^{- i \Phi(\bar \xi) t}$, 
the trilinear expression $\N(v_1, v_2, v_3)$ 
is non-autonomous and in fact depends on $t$. 
For simplicity of notations, however, we suppress such $t$-dependence
when there is no confusion.
We also set $\N(v) = \N(v, v, v)$, when all the three arguments are identical.
We apply this convention to  all the multilinear operators
appearing in this paper.

\end{remark}

For $ M\geq 1$ and $\al  \in \R$, we also define trilinear operators
$\N^\al _{\leq M}$, 
$\N^\al _{> M}$, 
and $\N^\al _{ M}$
with modulation restrictions:
\begin{align}
\N^\al _{\leq\,(>)\, M}(v_1,v_2,v_3)(\xi,t)
&:=  i\intt_{\substack{\xi = \xi_1-\xi_2+\xi_3\\|\Phi(\bar{\xi})-\al |\leq\,(>)\, M}}
e^{-i\Phi(\bar{\xi})t}v_1(\xi_1)\overline{v_2(\xi_2)}v_3(\xi_3)d\xi_1d\xi_2, 
\label{NLS2a}\\
\rule{0mm}{6mm}
\N^\al _{ M}(v_1,v_2,v_3)(\xi,t)
&:=  \N^\al _{\leq2 M}(v_1,v_2,v_3)(\xi,t)
- \N^\al _{\leq  M}(v_1,v_2,v_3)(\xi,t)
\notag \\
&\hphantom{:}= i\intt_{\substack{\xi = \xi_1-\xi_2+\xi_3\\  |\Phi(\bar{\xi})-\al |\sim M}}
e^{-i\Phi(\bar{\xi})t}v_1(\xi_1)\overline{v_2(\xi_2)}v_3(\xi_3)d\xi_1d\xi_2,
\notag
\end{align}

\noi
where $|\Phi(\bar{\xi})-\al |\sim  M$ is a shorthand notation for $M<|\Phi(\bar{\xi})-\al |\leq2 M$.
The following trilinear operator also plays an important role
in our analysis:
\[ \I^\al _{M}(v_1,v_2,v_3)(\xi,t)
:=  i\intt_{\substack{
\xi = \xi_1-\xi_2+\xi_3
\\  |\Phi(\bar{\xi})-\al | \sim M }}
\frac{e^{-i\Phi(\bar{\xi})t}}{\Phi(\bar{\xi})-\al }v_1(\xi_1)\overline{v_2(\xi_2)}v_3(\xi_3)
d\xi_1d\xi_2.\]

\noi
We also define 
$\I^\al_{> M}$
in an obvious manner.
In the subsequent part of this paper, we use the following conventions:

\begin{itemize}
\item[$\bullet$]
When $\al =0$, we  drop the superscript 
and simply write 
$\N_{M}$, $\N_{\leq  M}$, \dots, 
for 
$\N^0_M$, $\N^0_{\leq  M}$, \dots.

\smallskip

\item[$\bullet$]
In Section \ref{SEC:3}, 
these multilinear operators appear in an iterative manner.
For clarity, 
we often write $\N_{|\Phi (\bar \xi) - \al|\sim M}^\al$ for $\N^\al_M$,
thus explicitly showing the variable of restriction.

\end{itemize}

\begin{remark}\rm
Recall that the (time) resonance corresponds to $\Phi(\bar\xi) = 0$.
Thus, the term 
$\N^0_{\leq  M}$ corresponds to the {\it nearly resonant} contribution
to the nonlinearity $\N$ (with the cutoff size $M$).
\end{remark}

We now state the localized modulation estimates
for the cubic NLS.
These trilinear estimates play a key role in 
our analysis in Section \ref{SEC:3}.

\begin{lemma}[Localized modulation estimates for the cubic NLS]
\label{LEM:NLS1}
Let $s\geq0$.
Then,  we have
\begin{align}
\label{R1}
\|\N^\al _{\leq M}(v_1,v_2,v_3) \|_{H^s}
& \lesssim 
\jb{\al }^{0+} M^{\frac{1}{2}+}
\prod_{j = 1}^3\|v_j\|_{H^s},\\
\label{R2}
\|\N^\al _{\leq M}(v)-\N^\al _{\leq M}(w)\|_{H^s}
& \lesssim 
\jb{\al }^{0+} M^{\frac{1}{2}+}
\big(\|v\|_{H^s}^2+\|w\|_{H^s}^2\big)\|v-w\|_{H^s}, 
\end{align}

\noi
for any $M\geq 1$ and $\al  \in \R$. 
\end{lemma}

\begin{remark}\rm 
Recall that the trilinear operator
$\N^\al _{\leq M}(v_1,v_2,v_3)$
depends on $t \in \R$ in  an non-autonomous manner.
Hence, 
strictly speaking, 
we should have written the first estimate \eqref{R1} 
 as 
\begin{align*}
\sup_{t\in \R}  \|\N^\al _{\leq M}(v_1,v_2,v_3)\|_{H^s}
& \lesssim 
\jb{\al }^{0+} M^{\frac{1}{2}+}
\prod_{j = 1}^3\|v_j\|_{H^s}.
\end{align*}

\noi
Note that,  in the definition \eqref{NLS2a}, 
 the non-autonomous parameter $t\in \R$ appears
only in the 
oscillatory factor $e^{-i\Phi(\bar{\xi})t}$.
We, however, do not make use of 
this oscillatory factor
in the  proof of \eqref{R1}.
See \eqref{A1} below.
In particular,   \eqref{R1}
holds uniformly in  $t\in \R$.
In view of this observation, 
we simply write \eqref{R1}
with the understanding that the estimate holds uniformly 
in the non-autonomous parameter $t \in \R$.
We use this convention for  all the multilinear estimates
appearing in this paper.

Let us also note that the ``spatial'' estimate \eqref{R1} immediately implies
the following space-time estimate:
\begin{align*}
\|\N^\al _{\leq M}(v_1,v_2,v_3)\|_{L^\infty_T H^s_x}
& \lesssim 
\jb{\al }^{0+} M^{\frac{1}{2}+}
\prod_{j = 1}^3\|v_j\|_{L^\infty_T H^s_x}
\end{align*}

\noi
for all $v_j \in L^\infty([-T, T];  H^s(\R))$.
The same remark also applies to the other multilinear estimates.

\end{remark}

\begin{proof}
In the following, we only present the proof of  \eqref{R1}, 
since
the second estimate \eqref{R2} on the difference follows 
from \eqref{R1} and the multilinearity of $\N_{\leq M}^\al $.
By the triangle inequality with $s \geq 0$, we have 
$\jb{\xi}^s\les \jb{\xi_1}^s\jb{\xi_2}^s\jb{\xi_3}^s$
under $\xi = \xi_1-\xi_2+\xi_3$.
Hence, it suffices to prove~\eqref{R1} for $s = 0$.

By duality, the desired estimate \eqref{R1} follows once we prove 
\begin{align}
\Bigg|
\intt_{ \xi = \xi_1-\xi_2+\xi_3}
\ind_{|\Phi(\bar{\xi})-\al |\leq M}
v_1(\xi_1)v_2(\xi_2)v_3(\xi_3)v_4(\xi)d\xi_1d\xi_2d\xi
\Bigg|
\lesssim 
\jb{\al }^{0+} M^{\frac{1}{2}+}
\prod_{j = 1}^4 \|v_j\|_{L^2}
\label{A1}
\end{align}

\noi
for all non-negative functions $v_1, \dots, v_4\in L^2_\xi(\R)$.

\medskip

\noi
$\bullet$ {\bf Case 1:}  $\min(|\xi_{2-1}|, |\xi_{2-3}|)\leq 1$.
\\
\indent
 Let $\zeta = \xi_2 - \xi_1 = \xi_3 - \xi$.
Without loss of generality, 
we assume that $|\zeta|\leq1$.
Then, it follows from  H\"older's  inequality that 
\begin{align*}
\text{LHS of }\eqref{A1}
& = \bigg|\int_{|\zeta|\leq 1}  \int_{\xi_1} v_1(\xi_1) v_2 (\xi_1 + \zeta) d\xi_1\int_{\xi_3} v_3 (\xi_3) v_4 (\xi_3 - \zeta)d\xi_3 d\zeta\bigg|\notag \\
& \leq  \bigg\|  \int_{\xi_1} v_1(\xi_1) v_2 (\xi_1 + \zeta) d\xi_1\bigg\|_{L^\infty_\zeta}
\bigg\|\int_{\xi_3} v_3 (\xi_3) v_4 (\xi_3 - \zeta)d\xi_3 \bigg\|_{L^\infty_\zeta}\notag \\
& \lesssim \prod_{j = 1}^4 \|v_j\|_{L^2}. 
\end{align*}

\noi
This proves \eqref{A1}.

\medskip
\noi
$\bullet$ \textbf{Case 2:} 
 $\min(|\xi_{2-3}|, |\xi_{2-1}|)>1$.
\\
\indent
Without loss of generality, assume that $\xi - \xi_3 >1 $.
Under $|\Phi(\bar{\xi})-\al |\leq M$, 
it follows from \eqref{Phi1} that 
\begin{align}
\frac{\al -M}{2(\xi - \xi_3)}
\leq \xi - \xi_1\leq 
\frac{\al +M}{2(\xi - \xi_3)}.
\label{A2}
\end{align}

\noi
Then, 
by the standard Cauchy-Schwarz argument
 with \eqref{A2}, we have 
\begin{align}
\text{LHS of }\eqref{A1}
& \leq \bigg\|
\intt_{\xi = \xi_1 - \xi_2 + \xi_3} 
\ind_{|\Phi(\bar{\xi})-\al |\leq M}
v_1(\xi_1)v_2(\xi_2)v_3(\xi_3)d\xi_1 d\xi_3\bigg\|_{L^2_{\xi}}\|v_4 \|_{L^2}\notag\\
& \leq \sup_{\xi}
\bigg(\intt_{\xi = \xi_1 - \xi_2 + \xi_3} 
\ind_{|\Phi(\bar{\xi})-\al |\leq M}d\xi_1 d\xi_3 \bigg)^\frac{1}{2}
\prod_{j = 1}^4 \|v_j\|_{L^2} \notag\\
& \leq \sup_{\xi}
\bigg(\int_{1< \xi - \xi_3\les |\al | + M } 
\frac{M}{\xi-\xi_3}d\xi_3 \bigg)^\frac{1}{2}
\prod_{j = 1}^4 \|v_j\|_{L^2} \notag \\
& \les \jb{\al }^{0+} M^{\frac{1}{2}+}
\prod_{j = 1}^4 \|v_j\|_{L^2}, 
\label{A2a}
\end{align}

\noi
where we used the assumption that $|\xi - \xi_1| > 1$ 
and  $|\Phi(\bar{\xi})| \leq |\al |+ M$
in the third inequality.
This completes the proof of Lemma \ref{LEM:NLS1}.
\end{proof}

Next, we estimate the trilinear operators
 $\I^\al _{ M}$
 and 
 $\I^\al _{> M}$.

\begin{lemma}\label{LEM:NLS2}

Let $s\geq 0$.
Then, we have
\begin{align}
\label{B1}
\|\I^\al _{ M}(v)\|_{H^s}& \lesssim \jb{\al }^{0+} M^{-\frac{1}{2}+}\|v\|_{H^s}^3,\\
\label{B2}
\|\I^\al _{M}(v)-\I^\al _{ M}(w)\|_{H^s} 
& \lesssim \jb{\al }^{0+} M^{-\frac{1}{2}+}
 \big(\|v\|_{H^s}^2+\|w\|_{H^s}^2\big)\|v-w\|_{H^s},
\intertext{and} 
\label{B3}
\|\I^\al _{> M}(v)\|_{H^s}& \lesssim \jb{\al }^{0+} M^{-\frac{1}{2}+}\|v\|_{H^s}^3,\\
\label{B4}
\|\I^\al _{> M}(v)-\I^\al _{>  M}(w)\|_{H^s} 
& \lesssim \jb{\al }^{0+} M^{-\frac{1}{2}+}
 \big(\|v\|_{H^s}^2+\|w\|_{H^s}^2\big)\|v-w\|_{H^s}, 
\end{align}

\noi
for any $M\geq 1$ and $\al  \in \R$.

\end{lemma}

\begin{proof}
In the following, we only prove \eqref{B1} and \eqref{B3} 
since \eqref{B2} and \eqref{B4} follow in a similar manner.

Note that we did not exploit 
the oscillatory nature of the exponential factor  
$e^{- i\Phi(\bar{\xi}) t}$ in the proof of Lemma \ref{LEM:NLS1}.
See \eqref{A1}.
Hence, by Lemma \ref{LEM:NLS1},  we have 
\begin{align*}
\|\I^\al _{M}(v)\|_{H^s}
&=\bigg\|\intt_{\substack{\xi = \xi_1-\xi_2+\xi_3\\|\Phi(\bar{\xi})-\al | \sim M }} 
\frac{e^{- i\Phi(\bar{\xi}) t}}{\Phi(\bar{\xi})-\al }v(\xi_1)\cj{v(\xi_2)}v(\xi_3)d\xi_1d\xi_2\bigg\|_{H^s}\\
&\les \frac{1}{M} \bigg\|\intt_{\xi = \xi_1-\xi_2+\xi_3} 
\ind_{|\Phi(\bar{\xi})-\al | \sim M }
\prod_{j = 1}^3 |v(\xi_j)| d\xi_1d\xi_2\bigg\|_{H^s}\\
&\lesssim \jb{\al }^{0+} M^{-\frac{1}{2}+}\|v\|_{H^s}^3.
\end{align*}

\noi
This proves \eqref{B1}.
Similarly, we have 
\begin{align*}
\|\I^\al _{> M}(v)\|_{H^s}
&=\bigg\|\intt_{\substack{\xi = \xi_1-\xi_2+\xi_3\\|\Phi(\bar{\xi})-\al | > M }} 
\frac{e^{- i\Phi(\bar{\xi}) t}}{\Phi(\bar{\xi})-\al }v(\xi_1)\cj{v(\xi_2)}v(\xi_3)d\xi_1d\xi_2\bigg\|_{H^s}\\
&\leq\sum_{\substack{N\geq M\\\text{dyadic}}}\bigg\|
\intt_{\substack{\xi = \xi_1-\xi_2+\xi_3\\
N<|\Phi(\bar{\xi})-\al |\leq 2N}} 
\frac{ e^{- i\Phi(\bar{\xi}) t}}{\Phi(\bar{\xi})-\al }v(\xi_1)\cj{v(\xi_2)}v(\xi_3)d\xi_1d\xi_2\bigg\|_{H^s}\\
&\lesssim \jb{\al }^{0+}\sum_{\substack{N \geq M \\\text{dyadic}}}N^{-\frac{1}{2}+}\|v\|_{H^s}^3\\
&\lesssim \jb{\al }^{0+} M^{-\frac{1}{2}+}\|v\|_{H^s}^3.
\end{align*}

\noi
This proves \eqref{B3}.
\end{proof}

\subsection{Localized modulation  estimates for the mKdV}
In this subsection, we perform similar analysis on the mKdV \eqref{mKdV}
and establish localized modulation estimates
on the relevant trilinear operators.
Let $v(t) = S(-t)u(t)$ be the interaction representation defined in \eqref{interaction}.
Then, we have 
 $\ft v(\xi,t)=e^{i\xi^3t}\widehat{u}(\xi,t)$.
Define a trilinear operator $\N(v_1, v_2, v_3)$ by\footnote{We 
follow
the conventions introduced in Remark \ref{REM:tdepend}.} 
\begin{align}
\N(v_1, v_2, v_3) (\xi, t) 
:=  - i\intt_{\xi = \xi_1+\xi_2+\xi_3}\xi  e^{i\Psi(\bar{\xi})t}v_1(\xi_1)v_2(\xi_2)v_3(\xi_3)d\xi_1 d\xi_2, 
\label{mKdV2}
\end{align}

\noi
where the modulation function $\Psi(\bar{\xi})$ is given by 
\begin{align}
\Psi(\bar{\xi})=\Psi(\xi,\xi_1,\xi_2,\xi_3)=\xi^3-\xi_1^3-\xi_2^3-\xi_3^3=3(\xi_1+\xi_2)(\xi_2+\xi_3)(\xi_3+\xi_1). 
\label{Psi}
\end{align}

\noi
Here, the last equality holds under the condition $\xi = \xi_1+\xi_2+\xi_3$.
With this notation, we can write the mKdV \eqref{mKdV}
as 
\begin{equation}\label{KdV2}
\dt v= \N(v).
\end{equation}

As before, we define several trilinear operators. 
Given $ M\geq 1$ and $\al  \in \R$, 
we let
\begin{align*}
\N^\al _{\leq\,(>) M}(v_1,v_2,v_3)(\xi,t)
&:= - i\intt_{\substack{\xi = \xi_1+\xi_2+\xi_3\\|\Psi(\bar{\xi})-\al |\leq\,(>) M}}
\xi e^{i\Psi(\bar{\xi}) t}
v_1(\xi_1)v_2(\xi_2)v_3(\xi_3)d\xi_1d\xi_2,\\
\N^\al _{ M}(v_1,v_2,v_3)(\xi,t)
&:= - i\intt_{\substack{\xi = \xi_1+\xi_2+\xi_3\\ |\Psi(\bar{\xi})-\al |\sim M}}
\xi e^{i\Psi(\bar{\xi}) t}v_1(\xi_1)v_2(\xi_2)v_3(\xi_3)
d\xi_1d\xi_2.
\end{align*}

\noi
We also define the following  trilinear operator: 
\[\I^\al _{M}(v_1,v_2,v_3)(\xi,t)
:=-  i\intt_{\substack{\xi = \xi_1+\xi_2+\xi_3\\|\Psi(\bar{\xi})-\al |\sim M }} 
\frac{\xi e^{i\Psi(\bar{\xi}) t}}{\Psi(\bar{\xi})-\al }v_1(\xi_1)v_2(\xi_2)v_3(\xi_3)d\xi_1d\xi_2\]

\noi
and  define 
$\I^\al_{> M}$
in an obvious manner.

We now present  the localized modulation estimates for the mKdV.
While the proof does not employ any sophisticated analytical tools, 
it is  more involved than
the proof of 
Lemma~\ref{LEM:NLS1}.

\begin{lemma}[Localized modulation estimates for the mKdV]
\label{LEM:KdV1}
Let $s\geq \frac 14$.
Then, we have
\begin{align}\label{K1}
\|\N^\al _{\leq M}(v_1,v_2,v_3)\|_{H^s}
& \lesssim \jb{\al }^{0+} M^{\frac{1}{2}+}
\prod_{j=1}^3\|v_j\|_{H^s},\\
\label{K2}
\|\N^\al _{\leq M}(v)-\N^\al _{\leq M}(w)\|_{H^s}
& \lesssim \jb{\al }^{0+} M^{\frac{1}{2}+}
\big(\|v\|_{H^s}^2+\|w\|_{H^s}^2\big)\|v-w\|_{H^s}, 
\end{align}

\noi
for any $M\geq 1$ and $\al  \in \R$.

\end{lemma}

\begin{proof}
In the following, we only present the proof of  \eqref{K1}, 
since the second estimate \eqref{K2} on the difference follows 
from \eqref{K1} and the multilinearity of $\N_{\leq M}^\al $.
By the triangle inequality: $\jb{\xi}^\s \les \jb{\xi_1}^\s\jb{\xi_2}^\s\jb{\xi_3}^\s$
for $\s \geq 0$ under $\xi = \xi_1 + \xi_2 + \xi_3$,  
it suffices to prove \eqref{K1} for $s = \frac 14$.

By duality, the desired estimate \eqref{K1} follows once we prove 
\begin{align}
\Bigg|
\intt_{\xi = \xi_1 + \xi_2 + \xi_3}
\ind_{|\Psi(\bar{\xi})-\al |\leq M}
\cdot m (\bar\xi)\prod_{j = 1}^3 v_j(\xi_j) v_4(\xi) 
d\xi_1d\xi_2d\xi\Bigg|
\lesssim 
\jb{\al }^{0+} M^{\frac{1}{2}+}
\prod_{j = 1}^4 \|v_j\|_{L^2} 
\label{K3}
\end{align}

\noi
for all non-negative functions $v_1, \cdots, v_4 \in L^2_\xi(\R)$, 
where 
the multiplier $m(\bar \xi)$ is given by 
\begin{align}
 m(\bar \xi) =m(\xi, \xi_1, \xi_2, \xi_3)
= \frac{ |\xi| \jb{\xi}^\frac{1}{4}}{ \jb{\xi_1}^\frac{1}{4}\jb{\xi_2}^\frac{1}{4}\jb{\xi_3}^\frac{1}{4}}.
\label{K3c}
\end{align}

By the standard Cauchy-Schwarz argument,  we have 
\begin{align}
\text{LHS of }\eqref{K3}
& \leq \bigg\|
\intt_{\xi = \xi_1 + \xi_2 + \xi_3} 
\ind_{|\Psi(\bar{\xi})-\al |\leq M}\cdot m(\bar \xi) 
\prod_{j = 1}^3 v_j(\xi_j)
d\xi_1d\xi_2
\bigg\|_{L^2_{\xi}}\|v_4 \|_{L^2}\notag\\
& \leq \sup_{\xi}
\bigg(\intt_{\xi = \xi_1 + \xi_2 + \xi_3} 
\ind_{|\Psi(\bar{\xi})-\al |\leq M} \cdot m^2(\bar \xi) d\xi_1d\xi_2  \bigg)^\frac{1}{2}
\prod_{j = 1}^4 \|v_j\|_{L^2}.
\label{K3b}
\end{align}

\noi
Hence, it suffices to show that 
\begin{align}
\sup_{\xi}
\bigg(\intt_{\xi = \xi_1 + \xi_2 + \xi_3} 
\ind_{|\Psi(\bar{\xi})-\al |\leq M} \cdot m^2(\bar \xi) d\xi_1d\xi_2 \bigg)^\frac{1}{2} 
\les \jb{\al }^{0+} M^{\frac{1}{2}+}.
\label{K3a}
\end{align}

\noi
In the following, we either prove \eqref{K3a}
or directly establish  \eqref{K3}.

\medskip

\noi
$\bullet$ {\bf Case 1:}  $|\xi|\les 1$.
\\
\indent
By Cauchy-Schwarz, H\"older's,  and Young's inequalities
followed by H\"older's inequality, we have
\begin{align*}
\text{LHS of }\eqref{K3}
& \les 
\bigg\|\intt_{\xi = \xi_1 + \xi_2 + \xi_3}
\prod_{j = 1}^3 \jb{\xi_j}^{-\frac{1}{4}}v_j(\xi_j)d\xi_1d\xi_2\bigg\|_{L^2_{|\xi|\les 1}}\| v_4\|_{L^2_\xi} \notag\\
& \les 
\bigg\|\intt_{\xi = \xi_1 + \xi_2 + \xi_3}
\prod_{j = 1}^3 \jb{\xi_j}^{-\frac{1}{4}}v_j(\xi_j)d\xi_1d\xi_2\bigg\|_{L^\infty_{|\xi|\les 1}}\| v_4\|_{L^2}\\
& \les 
\prod_{j = 1}^3 \|\jb{\xi_j}^{-\frac{1}{4}}v_j(\xi_j)\|_{L^\frac{3}{2}_{\xi_j}}\| v_4\|_{L^2}
 \les\prod_{j = 1}^4 \|v_j\|_{L^2}.
\end{align*}

\smallskip

In the following, we consider the case $|\xi|\gg 1$.
Without loss of generality, we assume that $|\xi_{12}| \geq |\xi_{23}|\geq |\xi_{31}|$.

\medskip

\noi
$\bullet$ {\bf Case 2:}  $|\xi|\gg 1$ and $|\xi_{31}|\leq |\xi_{23}| \leq 1$. 
\\
\indent
 In this case, we have $|\xi +\xi_3| = |\xi_{31} + \xi_{23}|  \les 1$.
Since $|\xi|\gg1$, this yields 
\begin{align}
|\xi_{12}| = |\xi - \xi_3| \sim |\xi |\gg 1.
\label{K3a1}
\end{align}

\noi
Moreover, we have 
$|\xi_1|\sim|\xi_2|\sim|\xi_3|\sim|\xi|\gg1$.
Thus, we have
\[ m(\bar \xi) \sim |\xi|^\frac{1}{2}
\]

\noi
in this case.
Let 
$\zeta_1 = \xi_{23}$, $\zeta_2 = \xi_{31}$, and $\zeta_3 = \xi_{12}$.
Then,  it follows from \eqref{Psi} that 
\begin{align}
\Psi(\bar \xi) = 3 \zeta_1 \zeta_2 \zeta_3.
\label{Psi2}
\end{align}

\noi
In the following, we freely use (partial) changes 
of variables between $\xi_1, \xi_2, \xi_3, \xi$ and $\zeta_1, \zeta_2, \zeta_3$.
Note that we have  
$|\zeta_2| \leq |\zeta_1| \leq 1 $.

\medskip

\noi
{\bf Subcase 2.a:}
$|\al| \les M$.
\\
\indent
For fixed $|\xi|\gg1 $, 
the condition 
$|\Psi(\bar{\xi})-\al |\leq M$ with \eqref{K3a1} and \eqref{Psi2} implies that 
\begin{align*}
 |\zeta_2| \leq |\zeta_1|^\frac{1}{2}|\zeta_2|^\frac{1}{2} 
\les \frac{(|\al|  + M)^\frac{1}{2}}{|\xi|^\frac{1}{2}}
\les \frac{M^\frac{1}{2}}{|\xi|^\frac{1}{2}}.
\end{align*}

\noi
Then, by a change of variables
and Cauchy-Schwarz inequality, we have
\begin{align}
\text{LHS of }\eqref{K3}
& \les 
\sum_{\substack{N\gg 1\\\text{dyadic}}}
N^\frac{1}{2}
\int_{|\zeta_2|\les \frac{M^\frac{1}{2}}{N^\frac{1}{2}}}
\bigg(\int_{|\xi_1|\sim N}  v_1(\xi_1) v_3(-\xi_1 + \zeta_2) d\xi_1\bigg) \notag\\
& \hphantom{XXXXXXXX}
\times 
\bigg(\int_{|\xi|\sim N} v_2(\xi-\zeta_2) v_4(\xi) d\xi\bigg)
d\zeta_2 \notag\\
& \les 
M^\frac{1}{2} \|v_2\|_{L^2}\|v_3\|_{L^2}
\sum_{\substack{N\gg 1\\\text{dyadic}}}
\|\P_{N}v_1\|_{L^2}
\|\P_{N}v_4\|_{L^2}\notag\\
& \lesssim 
 M^{\frac{1}{2}}
\prod_{j = 1}^4 \|v_j\|_{L^2} ,\notag
\end{align}

\noi
yielding \eqref{K3}.
Here, $\P_N$  denotes the Littlewood-Paley projector
onto the spatial frequencies $\{|\xi|\sim N\}$.

\medskip

\noi
{\bf Subcase 2.b:}
$|\al| \gg M$.
\\
\indent
For fixed $M\geq1$, write $|\al|\sim 2^K M$ for some $K \in \NB$.
Note that  we have
\begin{align}
K \sim \log \bigg(\frac{|\al|}{M}\bigg)
\label{K4}.
\end{align}

\noi
If $|\zeta_2|\les \frac{M^\frac{1}{2}}{|\xi|^\frac{1}{2}}$, then 
we can proceed as in Subcase 2.a.
Hence, we assume that 
\[ |\zeta_1|\geq 
|\zeta_2|\gg \frac{M^\frac{1}{2}}{|\xi|^\frac{1}{2}}\]

\noi
in the following.

If $|\zeta_1|\ges \frac{|\al| + M }{M^\frac{1}{2}|\xi|^\frac{1}{2}}
\sim \frac{|\al|  }{M^\frac{1}{2}|\xi|^\frac{1}{2}}$, then 
the condition 
$|\Psi(\bar{\xi})-\al |\leq M$ implies that 
\begin{align*}
 |\zeta_2| 
\les \frac{|\al|  + M}{|\zeta_1| |\xi|}
\les \frac{M^\frac{1}{2}}{|\xi|^\frac{1}{2}},
\end{align*}

\noi
thus reducing to the previous case.
Therefore, it remains to consider the case 
\begin{align}
 \frac{M^\frac{1}{2}}{|\xi|^\frac{1}{2}}
\ll   |\zeta_1|\ll
 \frac{|\al| }{ M^\frac{1}{2} |\xi|^\frac{1}{2}}
 \sim  \frac{2^K  M^\frac{1}{2}}{|\xi|^\frac{1}{2}},
\label{K5}
\end{align}

\noi
where $K$ satisfies \eqref{K4}.

Now, suppose that 
$   |\zeta_1|
\sim  \frac{2^k  M^\frac{1}{2}}{|\xi|^\frac{1}{2}}$
for some $1\leq k \leq K$.
Then, for fixed $\xi$ and $\zeta_1$, the condition 
$|\Psi(\bar{\xi})-\al |\leq M$ implies that 
\begin{align} 
 \frac{\al  - M}{3 |\zeta_1|}
\leq |F(\zeta_2)|
\leq \frac{\al  + M}{3|\zeta_1|}, 
\label{K6}
\end{align}

\noi
where $F(\zeta_2)$ is defined by 
\begin{align}
F(\zeta_2) =   \zeta_2^2 - (2\xi - \zeta_1) \zeta_2.
\label{K6a}
\end{align}

\noi
Note that the graph of 
$F(\zeta_2)$ is a parabola with a vertex
$\sim (\xi, -\xi^2)$ in view of $|\zeta_1| \leq 1 \ll |\xi|$.
In particular, 
the slope of this parabola when $ |\zeta_2|\leq 1$
is $- 2\xi + O(1)$.
Hence, it follows from \eqref{K6} and the assumption on the size of $|\zeta_1|$
that 
$\zeta_2$ belongs to an interval $I_k = I_k(\zeta_1, \xi)$ of length
\begin{align} 
| I_k(\zeta_1, \xi)|\sim \frac{ M}{|\zeta_1||\xi|}\sim \frac{ M^\frac{1}{2}}{2^k |\xi|^\frac{1}{2}}.
\label{K7}
\end{align}

\noi
Then, from \eqref{K5} and \eqref{K7}, we obtain 
\begin{align}
\text{LHS of }\eqref{K3a}
& \les \sup_{\xi}|\xi|^\frac{1}{2}
\bigg(\sum_{k = 1}^K
\int_{|\zeta_1|
\sim  \frac{2^k  M^\frac{1}{2}}{|\xi|^\frac{1}{2}}}
\int_{\zeta_2\in I_k(\zeta_1, \xi)}
1\,  d\zeta_2 d\zeta_1 \bigg)^\frac{1}{2} \notag\\
& \les \sup_{\xi}|\xi|^\frac{1}{4} M^\frac{1}{4}
\bigg(\sum_{k = 1}^K
\int_{|\zeta_1|
\sim  \frac{2^k  M^\frac{1}{2}}{|\xi|^\frac{1}{2}}}
2^{-k} 
 d\zeta_1 \bigg)^\frac{1}{2} \notag\\
& \les K^\frac{1}{2}  M^\frac{1}{2} 
\les \jb{\al}^{0+} M^\frac{1}{2}, \notag
\end{align}

\noi
where the last inequality follows from 
 \eqref{K4}.

\medskip

\noi
$\bullet$ {\bf Case 3:}  $|\xi|\gg 1$ and $|\xi_{31}| \leq 1 < |\xi_{23}|\leq |\xi_{12}|$. 
\\
\indent
In this case, we have $|\xi_2|\sim |\xi| \gg 1$
and $\jb{\xi_1} \sim \jb{\xi_3}$.
Thus, we have
\begin{align}
m(\bar \xi) \sim \frac{|\xi|}{\jb{\xi_1}^{\frac 12}}.
\label{K8}
\end{align}

\medskip

\noi
{\bf Subcase 3.a:}
$|\xi_1| \ges |\xi|$.
\\
\indent
Since $|\xi| \gg 1 \geq |\xi_{31}| = |\xi- \xi_2|$, we have
$|\xi_{23} + \xi_{12}| = |\xi+ \xi_2| \sim |\xi| $.  
By the triangle inequality
with $|\xi_{23}|\leq |\xi_{12}|$,  we have $|\xi_{12}| \ges |\xi|\gg 1$.
Let $F(\zeta_2)$ be as in \eqref{K6a}.
Then, noting that 
\[ F'(\zeta_2)  = 2\zeta_2 - 2\xi + \zeta_1
= - \xi_{12} + \zeta_2 = - \xi_{12} + O(1),\]

\noi
it follows from \eqref{K6} that
$\zeta_2$ belongs to an interval $I = I(\zeta_1, \xi)$ of length
\begin{align} 
| I(\zeta_1, \xi)|\les \frac{ M}{|\zeta_1||\xi|}
\leq \frac{ M}{|\xi|}
\label{K9}
\end{align}

\noi
for each fixed $\xi$ and $\zeta_1$
and hence for each fixed $\xi$ and $\xi_1 = \xi - \zeta_1$.
Given $k \in \Z$, let   $\Pi_k$ be the  frequency projector onto the interval $[k, k+1)$
defined in \eqref{K9a}.
Then, 
using a variant of the Cauchy-Schwarz argument \eqref{K3b}
with \eqref{K8} and \eqref{K9}, 
we have
\begin{align}
\text{LHS of }\eqref{K3}
& \leq \bigg\|\sum_{|k| \gg 1} 
\int_{|\xi_1| \in [k, k+1)}
\int_{|\zeta_2|\leq 1}
\ind_{|\Psi(\bar{\xi})-\al |\leq M}\cdot m(\bar \xi) 
\notag\\
& \hphantom{XXXXXXX}
\times v_1(\xi_1) v_2(\xi - \zeta_2) v_3 ( -\xi_1 + \zeta_2)
d\zeta_2 d\xi_1 
\bigg\|_{L^2_{\xi}}\|v_4 \|_{L^2}\notag\\
& \leq \sup_{|k|\gg 1} \sup_{\xi}
\bigg(
\int_{|\xi_1| \in [k, k+1)}
\int_{\zeta_2 \in I(\zeta_1, \xi)}
\ind_{|\Psi(\bar{\xi})-\al |\leq M} \cdot m^2(\bar \xi) d\zeta_2 d\xi_1 \bigg)^\frac{1}{2}
\notag\\
& \hphantom{XXXXXXX}
\times \sum_{|k|\gg 1} \sum_{\l = 0}^2
\|\Pi_k v_1\|_{L^2}\|\Pi_{-k  -\l} v_3\|_{L^2}\|v_2\|_{L^2}\|v_4\|_{L^2}\notag\\
& \les M ^\frac{1}{2}
\prod_{j = 1}^4 \|v_j\|_{L^2}.
\label{K10}
\end{align}

\noi
{\bf Subcase 3.b:}
$|\xi_1| \ll |\xi|$.
\\
\indent
In this case, we have $|\zeta_1| \sim |\xi|$.
Then, arguing as in Subcase 3.a, we conclude that 
$\zeta_2$ belongs to an interval $I = I(\zeta_1, \xi)$ of length
\begin{align*} 
| I(\zeta_1, \xi)|\les \frac{ M}{|\zeta_1||\xi|}
\sim \frac{ M}{|\xi|^2}
\end{align*}

\noi
for each fixed $\xi$ and $\zeta_1=\xi - \xi_1$.
In particular, we have 
\begin{align*}
  \sup_{|k|\gg 1} \sup_{\xi}
\bigg(
\int_{|\xi_1| \in [k, k+1)}
\int_{\zeta_2 \in I(\zeta_1, \xi)}
\ind_{|\Psi(\bar{\xi})-\al |\leq M} \cdot m^2(\bar \xi) d\zeta_2 d\xi_1 \bigg)^\frac{1}{2}
 \les M ^\frac{1}{2}.
\end{align*}

\noi
The rest follows as in \eqref{K10}.

\medskip

\noi
$\bullet$ {\bf Case 4:}  $|\xi|\gg 1$ and $ |\xi_{12}|,  |\xi_{23}|,  |\xi_{31}| > 1$. 
\\
\indent
Noting that 
$\max(  |\xi_{12}|,  |\xi_{23}|,  |\xi_{31}| )\ges |\xi|\gg 1$,  
the condition 
$|\Psi(\bar{\xi})-\al |\leq M$ with \eqref{Psi} implies that 
\begin{align}
|\al| + M \ges \max( |\xi|, |\xi_{12}|,  |\xi_{23}|,  |\xi_{31}|).
\label{K10a}
\end{align}

\noi
In the following, the size relation of 
$ |\xi_{12}|,  |\xi_{23}|,  |\xi_{31}|$ does not play any role.
Without loss of generality, assume that 
$|\xi_1|\ge |\xi_2|\ge |\xi_3|$.

\medskip

\noi
{\bf Subcase 4.a:}
$|\xi_1| \sim |\xi| \gg |\xi_2|\geq  |\xi_3|$.
\\
\indent
In this case, 
by viewing $\Psi$ as a function of $\xi_2$  for fixed  $\xi$ and $\xi_3$, 
we have $|\dd_{\xi_2} \Psi (\bar \xi)| \sim |(\xi - \xi_3) (\xi - 2\xi_2 - \xi_3)| 
= |\xi_{12}||\xi_{1-2}|\ges |\xi|^2\gg 1$.
Hence, with \eqref{K10a}, we have
\begin{align*}
\text{LHS of } \eqref{K3a} 
& \les \sup_{\xi}
\bigg(\intt_{\xi = \xi_1 + \xi_2 + \xi_3} 
\ind_{|\Psi(\bar{\xi})-\al |\leq M}
\frac{|\xi|^2}{\jb{\xi_2}^\frac{1}{2}\jb{\xi_3}^\frac{1}{2}}d\xi_2 d\xi_3 \bigg)^\frac{1}{2} 
\\
& \les M^\frac{1}{2}
\bigg(\int_{|\xi_3|\ll |\xi|} 
\frac{1}{\jb{\xi_3}}d\xi_3 \bigg)^\frac{1}{2} 
\les M^\frac{1}{2}(\log |\xi|)^\frac{1}{2}
\les \jb{\al }^{0+} M^{\frac{1}{2}+},
\end{align*}

\noi
yielding \eqref{K3a}.

\medskip

\noi
{\bf Subcase 4.b:}
$|\xi_1|, |\xi_2|  \ges |\xi| \gg  |\xi_3|$.

Note that, in the first step of \eqref{K3b}, 
we can perform Cauchy-Schwarz inequality in $\xi_2$ instead of $\xi$. 
Then, \eqref{K3} follows once we prove
\begin{align}
\sup_{\xi_2}
\bigg(\intt_{\xi = \xi_1 + \xi_2 + \xi_3} 
\ind_{|\Psi(\bar{\xi})-\al |\leq M} \cdot m^2(\bar \xi)d\xi_1 d\xi_3 \bigg)^\frac{1}{2} 
\les \jb{\al }^{0+} M^{\frac{1}{2}+}.
\label{K10a1}
\end{align}

If $|\xi_1|\sim|\xi_2|\gg|\xi|\gg|\xi_3|$, then $|\xi+\xi_1|\sim |\xi_2|$, and $|\xi_{23}|\sim|\xi_2|$.
Then, 
by viewing $\Psi$ as a function of $\xi_1$
for fixed  $\xi_2$ and $\xi_3$, 
we have  
\begin{align}
|\dd_{\xi_1} \Psi (\bar \xi)| = |\xi_{23} (\xi + \xi_1)| \ges |\xi_2|^2
\label{K10a2}
\end{align}

\noi
 and thus
\begin{align}
\text{LHS of } \eqref{K10a1} 
& \les \sup_{\xi_2}
\bigg(\intt_{\xi = \xi_1 + \xi_2 + \xi_3} 
\ind_{|\Psi(\bar{\xi})-\al |\leq M}
\frac{|\xi|^\frac{3}{2}}{\jb{\xi_3}^\frac{1}{2}}d\xi_1 d\xi_3 \bigg)^\frac{1}{2} 
\notag \\
& \les M^\frac {1}{2}
\sup_{\xi_2} 
\frac{1}{\jb{\xi_2}^\frac{1}{4}}
\bigg(\int_{|\xi_3|\ll |\xi_2|}
\frac{1}{\jb{\xi_3}^\frac{1}{2}}d\xi_3 \bigg)^\frac{1}{2} 
\les M^\frac{1}{2}.
\label{K10a3}
\end{align}

If $|\xi_1|\sim|\xi_2|\sim|\xi|\gg|\xi_3|$, we have
\begin{align*}
 \max(|\xi + \xi_1|, |\xi + \xi_2|)
\ges |2\xi + \xi_{12}| =|3\xi-\xi_3| \sim |\xi|.
\end{align*}

\noindent
Without loss of generality, assume that 
$|\xi + \xi_1|\ges  |\xi|$. (Otherwise, we switch the role of $\xi_1$ and $\xi_2$ in \eqref{K10a1}.)
Then, \eqref{K10a2} and hence \eqref{K10a3} hold in this case as well.

\medskip

\noi
{\bf Subcase 4.c:}
$|\xi_1|, |\xi_2| , |\xi_3| \ges |\xi|  $.
\\
\indent
In this case, we have
\begin{align}
\max(|\xi + \xi_1|, |\xi + \xi_2|, |\xi + \xi_3|)
\ges |3\xi + \xi_{123}| =4  |\xi|.
\label{K10b}
\end{align}

\noi
Without loss of generality, assume that 
$|\xi + \xi_1|\ges  |\xi|$.
Then, 
by viewing $\Psi$ as a function of $\xi_1$
for fixed  $\xi_2$ and $\xi_3$,  
 we have 
\begin{align}
|\dd_{\xi_1} \Psi (\bar \xi)| \sim |\xi_{23} (\xi + \xi_1)| \ges |\xi| |\xi_{23}|.
\label{K11}
\end{align}

\noi 
Note that, 
By  performing Cauchy-Schwarz inequality in $\xi_3$ instead of $\xi$
in the first step of \eqref{K3b}, 
it suffices to prove
\begin{align}
\sup_{\xi_3}
\bigg(\intt_{\xi = \xi_1 + \xi_2 + \xi_3} 
\ind_{|\Psi(\bar{\xi})-\al |\leq M} \cdot m^2(\bar \xi)d\xi_1 d\xi_2 \bigg)^\frac{1}{2} 
\les \jb{\al }^{0+} M^{\frac{1}{2}+}.
\label{K12}
\end{align}

\noi
From \eqref{K11} and \eqref{K10a}, we have
\begin{align*}
\text{LHS of } \eqref{K12} 
& \les \sup_{\xi_3}
\bigg(\intt_{\xi = \xi_1 + \xi_2 + \xi_3} 
\ind_{|\Psi(\bar{\xi})-\al |\leq M}
|\xi| d\xi_1 d\xi_2 \bigg)^\frac{1}{2} 
\\
& \les M^\frac{1}{2}
 \sup_{\xi_3}
\bigg(\int_{1\leq |\xi_{23}|\les |\al| + M }
\frac{1}{|\xi_{23}|}d\xi_2 \bigg)^\frac{1}{2} 
\les \jb{\al }^{0+} M^{\frac{1}{2}+}.
\end{align*}

\noi
This completes the proof of Lemma \ref{LEM:KdV1}.
\end{proof}

\begin{remark}\rm
While 
the simple  Cauchy-Schwarz argument \eqref{K3b}
works for 
most of the cases 
in the proof of Lemma \ref{LEM:KdV1},
it does not seem to work for  Case 2 in the endpoint case: $s = \frac 14$.
We point out  that 
the  Cauchy-Schwarz argument  suffices for Case 2
in the non-endpoint case: $s > \frac 14$.

\end{remark}

As an immediate corollary to Lemma \ref{LEM:KdV1}, we 
obtain the following lemma.
The proof is analogous to that of Lemma \ref{LEM:NLS2}.

\begin{lemma}\label{LEM:KdV2}
Let $s \geq \frac 14$.
Then, we have
\begin{align*}
\|\I^\al _{ M}(v)\|_{H^s}& \lesssim \jb{\al }^{0+} M^{-\frac{1}{2}+}\|v\|_{H^s}^3,\\
\|\I^\al _{M}(v)-\I^\al _{ M}(w)\|_{H^s} 
& \lesssim \jb{\al }^{0+} M^{-\frac{1}{2}+}
 \big(\|v\|_{H^s}^2+\|w\|_{H^s}^2\big)\|v-w\|_{H^s}, 
\intertext{and} 
\|\I^\al _{> M}(v)\|_{H^s}& \lesssim \jb{\al }^{0+} M^{-\frac{1}{2}+}\|v\|_{H^s}^3,\\
\|\I^\al _{> M}(v)-\I^\al _{>  M}(w)\|_{H^s} 
& \lesssim \jb{\al }^{0+} M^{-\frac{1}{2}+}
 \big(\|v\|_{H^s}^2+\|w\|_{H^s}^2\big)\|v-w\|_{H^s},
\end{align*}

\noi
for any $M \geq 1$ and $\al  \in \R$.

\end{lemma}

\section{Normal form reductions}\label{SEC:3}
In this section,
we implement an infinite iteration scheme of normal form reductions at a formal level.
We perform normal form reductions in an iterative manner, 
transforming   part of the nonlinearity 
into nonlinearities of higher and higher degrees.
In the end, we formally arrive at an equation 
involving infinite series of nonlinearities of arbitrarily high  degrees
(Subsection \ref{SUBSEC:3.4}).

Such an infinite iteration of normal form reductions 
was first introduced in Guo-Kwon-Oh~\cite{GKO} 
in proving unconditional well-posedness of  the  cubic NLS on $\T$. 
 While the implementation of normal form reductions  in \cite{GKO}
 was systematic,  
 the multilinear estimates heavily depended on 
 the structure of the equation
 as well as some elementary number theory (the divisor counting argument).
In the following, 
we perform normal form reductions
in a rather abstract manner.
This allows us to handle the cubic NLS~\eqref{NLS} and the mKdV~\eqref{mKdV}
in an identical manner
by applying the localized modulation estimates obtained in  Section~\ref{SEC:2}.

Before proceeding further, 
we need to set up some notations.
In the following, 
we simply denote the Fourier coefficient 
  $v(\xi) = \ft v(\xi)$ by $v_\xi$.
When the complex conjugate sign on  $v_{\xi}$ does not play any significant role,
 we drop the complex conjugate sign.
We often drop the complex number $i$
and simply use $1$ for $\pm 1$ and $\pm i$.

In the following presentation of normal form reductions,
we restrict our attention 
to the cubic NLS \eqref{NLS}.
In view of the localized modulation estimates (Lemmas \ref{LEM:KdV1}
and  \ref{LEM:KdV2}), 
one can easily modify the argument to handle the mKdV \eqref{mKdV}.
All the computations in this section (such as
switching summations and integrals) are formal, assuming that $u$
(and hence $v$)
is a smooth solution.
In Section \ref{SEC:4}, we justify our formal computations
when $ u \in C_t H^s_x$
with (i)  $ s\geq \frac{1}{6}$
for the cubic  NLS
and (ii) $ s> \frac{1}{4}$ for the mKdV, respectively.

\subsection{Notation: index by trees}

When we apply a normal form reduction, 
i.e.~integration by parts as in \eqref{I3},\footnote{In fact, 
we proceed without an integration symbol in the following. 
Namely, we perform differentiation by parts.}
a time derivative can fall on any of the factors $v_{\xi_j}$,
transforming the nonlinearity into that of a higher degree.
In each step of normal form reductions, 
we need to keep track of  
where a time derivative falls.
which may be a cumbersome task in general.
In \cite{GKO}, 
we introduced the notion of {\it ordered trees} for
 indexing such terms arising in the general steps of normal form reductions.
In order to carry out our analysis, we
will need to supplement more notations related to ordered trees
in the following.

\begin{definition} \rm
Given a partially ordered set $\TT$ with partial order $\leq$, we say that $b\in\TT$ with $b\leq a$ and $b\neq a$ is a child of $a\in\TT$, if $b\leq c\leq a$ implies either $c=a$ or $c=b$. If the latter condition holds, we also say that $a$ is the parent of $b$.
\end{definition}

As in \cite{Christ, Oh}, 
our trees refer to a particular subclass of ternary trees.

\begin{definition} \rm
A tree $\TT$ is a finite partially ordered set satisfying the following properties:
\begin{itemize}
\item[$\bullet$]
Let $a_1$, $a_2$, $a_3$, $a_4\in\TT$. If $a_4\leq a_2\leq a_1$ and $a_4\leq a_3\leq a_1$, then we have $a_2\leq a_3$ or $a_3\leq a_2$.

\item[$\bullet$]
 A node $a\in\TT$ is called terminal, if it has no child. A non-terminal node $a\in\TT$ is a node with exactly three children denoted by $a_1$, $a_2$ and $a_3$.\footnote{Note
that the order of children plays an important role in our discussion.
We refer to $a_j$ as the $j$th child of a non-terminal node $a \in \TT$.
In terms of the planar graphical representation of a tree, 
we set  the $j$th node from the left as the $j$th child $a_j$ of $a \in \TT$. 
}

\item[$\bullet$]
 There exists a maximal element $r\in\TT$ (called the root node) such that $a\leq r$ for all $a\in\TT$. We assume that the root node is non-terminal.

\item[$\bullet$]
 $\TT$ consists of the disjoint union of $\TT^0$ and $\TT^\infty$, where $\TT^0$ and $\TT^\infty$ denote the collection of non-terminal nodes and terminal nodes, respectively.
\end{itemize}
\end{definition}

Note that the number $|\TT|$ of nodes in a tree $\TT$ is $3j+1$ for some $j\in\mathbb{N}$, 
where $|\TT_0|=j$ and $|\TT^\infty|=2j+1$. 
We use $T(j)$ to denote the collection of trees of the $j$th generation, 
namely,  with $j$ parental nodes.

Next, we recall the notion of ordered trees introduced in \cite{GKO}. 
Roughly speaking, an ordered tree ``remembers how it grew''.
\begin{definition}\label{DEF:tree2}\rm
We say that a sequence $\{\TT_j\}_{j=1}^J$ is a chronicle of $J$ generations, if
\begin{itemize}
\item[$\bullet$]
 $\TT_j\in T(j)$ for each $j=1$, $\cdots$, $J$,

\item[$\bullet$]
 $\TT_{j+1}$ is obtained by changing one of the terminal nodes in $\TT_j$,
denoted by $p^{(j)}$, 
 into a non-terminal node (with three children), $j=1$, $\cdots$, $J-1$.

\end{itemize}

\noi
Given a chronicle $\{\TT_j\}_{j=1}^J$ of $J$ generations, 
we refer to $\TT_J$ as an {\it ordered tree of the $J$th generation}. 
We use $\TF(J)$ to denote the collection of the ordered trees of the $J$th generation.
 Note that the cardinality of $\TF(J)$ is given by
	\begin{equation}\label{cj}
	|\TF(J)|=1\cdot 3\cdot 5\cdot\cdots\cdot (2J-1)=: c_J
	\end{equation}

\end{definition}

\begin{remark}\rm

Given two ordered trees $\TT_J$ and $\wt{\TT}_J$
of the $J$th generation, 
it may happen that $\TT_J = \wt{\TT}_J$ as trees (namely as graphs) 
while $\TT_J \ne \wt{\TT}_J$ as ordered trees according to Definition \ref{DEF:tree2}.
Henceforth, when we refer to an ordered tree $\TT_J$ of  the $J$th generation, 
it is understood that there is an underlying chronicle $\{ \TT_j\}_{j = 1}^J$.

\end{remark}

\begin{definition} \label{DEF:tree3}\rm
(i) Given an ordered tree $\TT_J \in \TF(J)$ 
with a chronicle $\{\TT_j\}_{j = 1}^J$, 
we define a ``projection'' $\pi_j$, $j = 1, \dots, J$,  from $\TT_J$
to subtrees in $\TT_J$ of one generation 
by setting
\begin{itemize}
\item[$\bullet$]
 $\pi_1(\TT_J) = \TT_1$, 

\item[$\bullet$] 
$\pi_j(\TT_J)$ to be the tree formed by the three terminal nodes in $\TT_j \setminus \TT_{j-1}$
and its parent,  $j = 2, \dots, J$.   
Intuitively speaking, 
$\pi_j(\TT_J)$ is the tree added in transforming $\TT_{j-1}$ into $\TT_j$.

\end{itemize}

\noi
We use $r^{(j)}$ to denote the root node of $\pi_j(\TT_J)$
and refer to it as the {\it $j$th root node}.
By definition, we have
\begin{align}
r^{(j)} = p^{(j-1)}.
\label{tree3}
\end{align}

\noi
Note that
$p^{(j-1)}$ is not necessarily a node in 
$\pi_{j-1}(\TT_J)$.

\smallskip
\noi
(ii) Given $j \in \{1, \dots, J-1\}$, 
$p^{(j)}$  appears as a terminal node of  $\pi_{k}(\TT)$
for exactly one $k \in \{1, 2\dots, j-1\}$.
In particular, 
$p^{(j)}$ is the $\l$th child of the $k$th root note $r^{(k)}$
for some $\l\in \{1, 2, 3\}$.
We define {\it the order of $p^{(j)}$}, denoted by 
$\#p^{(j)}$, 
to be this number $\l\in \{1, 2, 3\}$.

\smallskip
\noi
(iii) 
We  define the {\it essential terminal nodes} $\pi_j^\infty(\TT_J)$ of the $j$th generation by 
setting
\[\pi_j^\infty(\TT_J): = \pi_j(\TT_J)^\infty \cap\TT_J^\infty
= (\TT_j \setminus \TT_{j-1})\cap \TT_J^\infty.\]

\noi
By definition, 
$\pi_j^\infty(\TT_J)$ may be empty.
Note that $\{ \pi_j^\infty(\TT_J)\}_{j = 1}^J$ forms 
a partition of $\TT_J^\infty$.
\end{definition}

We record the following simple observation.
This will be useful in Subsections \ref{SUBSEC:3.3}
and~\ref{SUBSEC:4.3}.

\begin{remark}\label{REM:order}\rm
Let $\TT \in \TF(J)$ be an ordered tree. 
Then, for each fixed  $j = 2, \dots, J$, 
there exists a path\footnote{A path is  a sequence of nodes $a_1, a_2, \dots, a_K$
such that 
$a_k$ and $a_{k+1}$ are adjacent.}
$a_1, a_2, \dots, a_K$,
starting at the root node $r = r^{(1)}$
and ending at 
the $j$th root node $r^{(j)}$
such that $a_k \ne r^{(\l)}$
for any $k = 1, \dots, K$
and $\l \geq  j+1$.
Namely, we can move from $r^{(1)}$
to $r^{(j)}$
without hitting a root node of a higher generation.

More concretely, given $r^{(j)}$, 
we know that it  appears as a terminal node of  $\pi_{j_1}(\TT)$
for exactly one $j_1 \in \{1, 2\dots, j-1\}$.
Similarly,  $r^{(j_1)}$ appears as a terminal node of  $\pi_{j_2}(\TT)$
for exactly one $j_2 \in \{1, 2\dots, j_1-1\}$.
We can iterate this process, which must terminate in a finite number of steps
with $j_k = 1$.
This generates the shortest path
$r^{(j_k)}, r^{(j_{k-1})}, \dots, r^{(j_1)}, r^{(j)}$
from $r^{(1)}$
to $r^{(j)}$
and we denote it by $P(r^{(1)}, r^{(j)})$.
Similarly, given $a\in \TT\setminus\{r^{(1)}\}$, 
one can easily construct the shortest path from $r^{(1)}$ to $a$
since $a$ is a terminal node of $\pi_k(\TT)$ for some $k$.
We denote this shortest path by 
$P(r^{(1)}, a)$.

\end{remark}

Given an ordered tree, 
we need to 
consider all possible frequency assignments
to nodes that are ``consistent''.

\begin{definition}\label{DEF:index}
\rm
Given an ordered tree $\TT \in \TF(J)$,  
we define an {\it index function} $\pmb{\xi}:\TT\to\mathbb{R}$ such that
\begin{align}
\xi_a=\xi_{a_1}-\xi_{a_2}+\xi_{a_3}
\label{index}
\end{align}

\noi
for $a\in\TT^0$, where $a_1$, $a_2$, and $a_3$ denote the children of $a$.
Here,  we identified $\pmb{\xi}:\TT\to\R$ with $\{\xi_a\}_{a\in\TT}\in\mathbb{R}^\TT$. We use $\Xi(\TT)\subset\R^{\TT}$ to denote the collection of such index functions $\pmb{\xi}$.
\end{definition}

\begin{remark}\rm 
(i) If we associate functions $v_a = v_a(\xi_a)$ to each node $a \in \TT$, 
then the relation~\eqref{index} implies that $v_a = v_{a_1}* \cj{v_{a_2}}*v_{a_3}$.

\smallskip

\noi
(ii) For the mKdV, we need to replace \eqref{index} 
by $\xi_a=\xi_{a_1}+\xi_{a_2}+\xi_{a_3}$.

\end{remark}

Given an ordered tree 
$\TT_J \in \TF(J)$  with a chronicle $\{ \TT_j\}_{j = 1}^J$ 
and associated index functions $\pmb{\xi} \in \Xi(\TT_J)$,
 we use superscripts to keep track of  ``generations'' of frequencies.

Consider $\TT_1$ of the first generation.
We define the first generation of frequencies by
\[\big(\xi^{(1)}, \xi^{(1)}_1, \xi^{(1)}_2, \xi^{(1)}_3\big) :=(\xi_r, \xi_{r_1}, \xi_{r_2}, \xi_{r_3}),\]

\noi
where $r_j$ denotes the three children of the root node $r$.

In general, the ordered tree $\TT_j$ 
of the $j$th generation is obtained from $\TT_{j-1}$ by
changing one of its terminal nodes $a  \in \TT^\infty_{j-1}$
into a non-terminal node.
Then, we define
the $j$th generation of frequencies by
\[\big(\xi^{(j)}, \xi^{(j)}_1, \xi^{(j)}_2, \xi^{(j)}_3\big) :=(\xi_a, \xi_{a_1}, \xi_{a_2}, \xi_{a_3}),\]

\noi
\noi
where $a_j$ denotes the three children of the  node  $a  \in \TT^\infty_{j-1}$.
Note that the parent node $a$ is nothing but the $j$th root node $r^{(j)}$
defined in Definition \ref{DEF:tree3}.

\medskip

Our main analytical tool is the localized modulation estimates
from Section \ref{SEC:2}.
Hence, it is important to keep track of the modulation 
for frequencies in  each generation.
We use $\mu_j$  to denote the corresponding modulation function  introduced at the $j$th generation.
Namely, we set\footnote{For the mKdV, 
the modulation function $\mu_j$ is given by 
\begin{align*}
\mu_j 
:= \big(\xi^{(j)}\big)^3 - \big(\xi_1^{(j)}\big)^3 - \big(\xi_2^{(j)}\big)^3- \big(\xi_3^{(j)}\big)^3.
\end{align*}
}
\begin{align*}
\mu_j & = \mu_j \big(\xi^{(j)}, \xi^{(j)}_1, \xi^{(j)}_2, \xi^{(j)}_3\big)
:= \big(\xi^{(j)}\big)^2 - \big(\xi_1^{(j)}\big)^2 + \big(\xi_2^{(j)}\big)^2- \big(\xi_3^{(j)}\big)^2 \notag \\
& = 2\big(\xi_2^{(j)} - \xi_1^{(j)}\big) \big(\xi_2^{(j)} - \xi_3^{(j)}\big)
= 2\big(\xi^{(j)} - \xi_1^{(j)}\big) \big(\xi^{(j)} - \xi_3^{(j)}\big), 
\end{align*}

\noi
where the last two equalities hold in view of \eqref{index}.
We also use the following shorthand notation:
\[\wt \mu_j := \sum_{k = 1}^j \mu_k.\]

\subsection{Normal form reductions: second and third generations} 
\label{SUBSEC:NF2}

We are now ready to  perform normal form reductions.
 As we mentioned earlier, 
 we only consider the cubic NLS \eqref{NLS} in $H^s(\R)$, $s \geq 0$, 
 in the following.
 Since  our implementation  is carried out at an abstract level, 
a minor modification suffices
for the  mKdV in $H^s(\R)$,  $ s\ge \frac14 $.

Fix dyadic $N > 1$ (to be determined later).
We first write \eqref{NLS2} as
\begin{align*}
\partial_t v=\N(v)
& =\N_{\leq N}(v)+\N_{> N }(v)\notag\\
& =: \N_{1}^{(1)}(v) + \N_{2}^{(1)}(v).
\end{align*}

\noi
By Lemma \ref{LEM:NLS1}, 
we can estimate the low modulation part:
\begin{align}
\| \N_{1}^{(1)}(v) \|_{H^s} =  \|\N_{\le N} (v) \|_{H^s} \lesssim N^{\frac 12+} \|v\|_{H^s}^3
\label{1g}
 \end{align}

\noi
for $s\geq 0$.
The main point is that the restriction 
$|\Phi(\bar \xi) |\leq N$
provides a restriction on the possible range of frequencies.

The high modulation part $\N_{2}^{(1)}(v) = \N_{> N }(v)$ with $|\Phi(\bar \xi) |> N$
can not benefit such a frequency restriction.
In this case, 
we exploit 
a rapid oscillation due to the high modulation, 
introducing cancellation under a time integration.
For this purpose, we iteratively apply differentiation by parts
and transform $\N_{2}^{(1)}(v)$
into  infinite series of multilinear terms.

Let $C_0$ denote the domain of  $\N_{2}^{(1)}(v) = \N_{>N}(v)$:
\begin{align}
C_0: =\big\{|\mu_1|> N \big\}.
\label{C0}
\end{align}

\noi
By taking differentiation by parts\footnote{When we apply differentiation by parts,
we keep the minus sign on the second term
for emphasis.}
 with \eqref{NLS2}, we have
\begin{align}
\N^{(1)}_2(v)(\xi,t)
& =\N_{> N }(v)(\xi,t)\notag \\
&=\intt_{\substack{\pmb{\xi}\in\Xi(\TT_1)\\\pmb{\xi}_r=\xi}}\ind_{C_0} e^{-i\mu_1 t}
\prod_{a\in\TT_1^\infty}v_{\xi_a}\notag \\
&=\dt\Bigg[\intt_{\substack{\pmb{\xi}\in\Xi(\TT_1)\\\pmb{\xi}_r=\xi}}
\ind_{C_0} \frac{e^{-i\mu_1 t}}{\mu_1}\prod_{a\in\TT_1^\infty}v_{\xi_a}\Bigg]
-\intt_{\substack{\pmb{\xi}\in\Xi(\TT_1)\\\pmb{\xi}_r=\xi}} 
\ind_{C_0} \frac{e^{-i\mu_1t}}{\mu_1}
\dt\bigg(\prod_{a\in\TT_1^\infty}v_{\xi_a}\bigg)\notag \\
&=\partial_t\Bigg[\intt_{\substack{\pmb{\xi}\in\Xi(\TT_1)\\\pmb{\xi}_r=\xi}}
\ind_{C_0}
\frac{e^{-i\mu_1 t}}{\mu_1}\prod_{a\in\TT_1^\infty}v_{\xi_a}\Bigg]
-\sum_{\TT_2\in\TF(2)}
\intt_{\substack{\pmb{\xi}\in\Xi(\TT_2)\\\pmb{\xi}_r=\xi}} 
\ind_{C_0}\frac{e^{-i(\mu_1+\mu_2)t}}{\mu_1}\prod_{a\in\TT_2^\infty}v_{\xi_a}\notag \\
& =: \dt \N_{0}^{(2)}(v)(\xi,t)
+ \N^{(2)}(v)(\xi,t).
\label{2g}
\end{align}

\noi
From Lemma \ref{LEM:NLS2}, 
we have the following estimate on  the boundary term $\N_{0}^{(2)}(v)$.

\begin{lemma}\label{LEM:2g0}
Let $s\geq 0$. Then, we have
\begin{align*}
\|\N_{0}^{(2)}(v)\|_{H^s}& \lesssim N^{-\frac{1}{2}+}\|v\|_{H^s}^3, \\ 
\|\N_{0}^{(2)}(v)-\N_{0}^{(2)}(w)\|_{H^s}
& \lesssim N^{-\frac{1}{2}+}\big(\|v\|_{H^s}^2+\|w\|_{H^s}^2\big)\|v-w\|_{H^s}.
\end{align*}

\end{lemma}

%
%
%
%

Next,  we decompose the frequency space into
\begin{align}
C_1: =\big\{|\mu_1+\mu_2|\leq5^3|\mu_1|^{1-\delta}\big\}
\label{C1}
\end{align}

\noi
and its complement $C_1^c$,\footnote{Clearly, the number $5^3$ in \eqref{C1} 
does not play any role  at this point.
However, we insert it to match with \eqref{CJ}.
See also  \eqref{C2} and \eqref{C3}.} where $\delta>0$ is a small constant.
 Then,  we decompose $\N^{(2)}$ as
\begin{align}
\N^{(2)}=\N_{1}^{(2)}+\N_{2}^{(2)},
\label{S1b}
\end{align}

\noi
where $\N^{(2)}_{1}:= \N^{(2)}|_{C_1}$ is defined as the restriction of $\N^{(2)}$ on $C_1$ 
and $\N_{2}^{(2)}:=\N^{(2)}-\N^{(2)}_{1}$, 
namely 
$\N_{2}^{(2)}$ is  the restriction of $\N^{(2)}$ on $C_1^c$. 
Note that we have 

\[\N_{>N}=\dt\N^{(2)}_{0}+\N_{1}^{(2)}+\N_{2}^{(2)}\]

\noi
at this point.
Thanks to the restriction \eqref{C1} on the modulation, 
we can estimate the first term $\N_{1}^{(2)}$. 
However,  we do not have a direct control of $\N_{2}^{(2)}$. 
In the following, we apply another normal form reduction 
to $\N_{2}^{(2)}$. 

\begin{lemma}\label{LEM:2g}
Let $s\geq  0$.
Then, we have 
\begin{align}
\label{2ga}
\|\N_{1}^{(2)}(v)\|_{H^s}
& \lesssim N^{-\frac{\delta}{2}+}\|v\|_{H^s}^5, \\
\|\N_{1}^{(2)}(v)-\N_{1}^{(2)}(w)\|_{H^s}
& \lesssim N^{-\frac{\delta}{2}+}\big(\|v\|_{H^s}^4+\|w\|_{H^s}^4\big)\|v-w\|_{H^s}, 
\label{2gb}
\end{align}

\noi
for $0 < \dl < 1$.

\end{lemma}

\begin{proof}
We only present the proof of \eqref{2ga}
since \eqref{2gb} follows in a similar manner
in view of the multilinearity of $\N_{1}^{(2)}$.
Moreover, by the triangle inequality, it suffices to prove \eqref{2ga}
for $s = 0$.
From \eqref{2g} and \eqref{S1b}
with  \eqref{C1}, we have
\begin{align*}
\N_{1}^{(2)}(v)(\xi,t)
& =\sum_{\TT_2\in\TF(2)}
\intt_{\substack{\pmb{\xi}\in\Xi(\TT_2)\\\pmb{\xi}_r=\xi}}
\ind_{C_0}
\frac{e^{- i\mu_1t}}{\mu_1}\prod_{a_1\in\pi_1^\infty(\TT_2)}v_{\xi_{a_1}}
\cdot \ind_{C_1}
e^{- i\mu_2t}\prod_{a_2\in\pi_2^\infty(\TT_2)}v_{\xi_{a_2}} \notag\\
& =\sum_{\TT_2\in\TF(2)}
\intt_{\substack{\pmb{\xi}\in\Xi(\TT_2)\\\pmb{\xi}_r=\xi}}
\ind_{C_0}
\frac{e^{- i\mu_1t}}{\mu_1}\prod_{a_1\in\pi_1^\infty(\TT_2)}v_{\xi_{a_1}}
\intt_{\substack{\pmb{\xi}^{(2)}\in\Xi( \pi_2(\TT_2))\\\pmb{\xi}^{(2)}_{r^{(2)}}=\xi^{(2)}}}
\ind_{C_1}
e^{- i\mu_2t}\prod_{a_2\in\pi_2^\infty(\TT_2)}v_{\xi_{a_2}}\\
& =\sum_{\TT_2\in\TF(2)}
 \intt_{\substack{\pmb{\xi}\in\Xi(\TT_2)\\\pmb{\xi}_r=\xi}}
\ind_{C_0}
\frac{e^{- i\mu_1t}}{\mu_1}\prod_{a_1\in\pi_1^\infty(\TT_2)}v_{\xi_{a_1}}
\cdot \N^{\mu_1}_{\leq5^3|\mu_1|^{1-\delta}}(v)(\xi^{(2)},t).
\end{align*}

\noi
In the second line, 
we slightly abused notations
in the domain of the second integration
for clarity
since, strictly speaking,  it is already included in the domain of the first integral.
Note that 
the second integral  is over three variables
$\{\xi_{a_2}\}_{a_2\in\pi_2^\infty(\TT_2)}$,
while 
the first integral 
is over two variables
$\{\xi_{a_1}\}_{a_1\in\pi_1^\infty(\TT_2)}$, 
with one constraint
$\pmb{\xi}_r=\xi$.

Then, from Lemmas \ref{LEM:NLS1} and \ref{LEM:NLS2}
with  \eqref{cj} and  \eqref{C0},  
we have 
\begin{align*}
\|\N^{(2)}_{1}(v)\|_{L^2}
&\lesssim\sum_{T_2\in\TF(2)}
\sum_{\substack{M\geq N\\\text{dyadic}}}
\|\I_{M}(v, v,\N^{\mu_1} _{\leq 5^3|\mu_1|^{1-\delta}}(v))\|_{L^2}\\
&\lesssim 
\sum_{\substack{M\geq N\\\text{dyadic}}}
M^{-\frac{1}{2}+}\|v\|_{L^2}^2\|\N^M_{\lesssim 5^3M^{1-\delta}}(v)\|_{L^2}\\
&\lesssim
  N^{-\frac{\delta}{2}+}\|v\|^5_{L^2}.
\end{align*}

\noi
This proves \eqref{2ga}.
\end{proof}

Next, we apply a normal form reduction to $\N_{2}^{(2)}$.
On the support of $\N_{2}^{(2)}$, namely,  on $C_0 \cap C_1^c$,  we have
\begin{align}
|\mu_1+\mu_2|> 5^3|\mu_1|^{1-\delta}
> N^{1- \dl}.
\label{C1a}
\end{align}

\noi
By applying  differentiation by parts once again, we have
\begin{align}
\N_{2}^{(2)}(v)(\xi)
& = \dt\Bigg[\sum_{\TT_2\in\TF(2)}
\intt_{\substack{C_0\cap C_1^c\\ \pmb{\xi}\in\Xi(\TT_2)\\\pmb{\xi}_r=\xi}}
\frac{e^{- i(\mu_1 + \mu_2)t}}{ \mu_1(\mu_1 +  \mu_2)}\prod_{a\in\TT_2^\infty}v_{\xi_a}\Bigg] \notag\\
& \hphantom{XX}
-\sum_{\TT_2\in\TF(2)}
\intt_{\substack{C_0\cap C_1^c\\
\pmb{\xi}\in\Xi(\TT_2)\\\pmb{\xi}_r=\xi}} 
\frac{e^{- i(\mu_1 + \mu_2) t}}{ \mu_1(\mu_1 + \mu_2)}\partial_t\bigg(\prod_{a\in\TT_2^\infty}v_{\xi_a}\bigg)
\notag \\
& = \dt\Bigg[\sum_{\TT_2\in\TF(2)}
\intt_{\substack{C_0\cap C_1^c\\ \pmb{\xi}\in\Xi(\TT_2)\\\pmb{\xi}_r=\xi}}
\frac{e^{- i(\mu_1 + \mu_2)t}}{ \mu_1(\mu_1 +  \mu_2)}\prod_{a\in\TT_2^\infty}v_{\xi_a}\Bigg]\notag \\
& \hphantom{XX}
-\sum_{\TT_3\in\TF(3)}
\intt_{\substack{C_0\cap C_1^c\\
\pmb{\xi}\in\Xi(\TT_3)\\\pmb{\xi}_r=\xi}} 
\frac{e^{- i(\mu_1 + \mu_2+ \mu_3) t}}{ \mu_1(\mu_1 + \mu_2)}\prod_{a\in\TT_3^\infty}v_{\xi_a}
\notag \\
& =:\dt \N^{(3)}_{0}(v)(\xi)+ \N^{(3)}(v)(\xi).
\label{3g}
\end{align}

We can easily estimate the boundary term  $\N^{(3)}_{0}(v)$ as follows.
\begin{lemma} \label{LEM:3g0}
Let $s\geq 0$. 
Then, we have
\begin{align}
\label{3g0a}
\|\N^{(3)}_{0}(v)\|_{H^s}
& \les N^{-1+\frac{\delta}{2}+}\|v\|^5_{H^s},  \\
\label{3g0b}
\|\N^{(3)}_{0}(v)-\N^{(3)}_0(w)\|_{H^s}
&\les N^{-1+\frac{\delta}{2}+}\big(\|v\|_{H^s}^4+\|w\|_{H^s}^4\big)\|v-w\|_{H^s}, 
\end{align}

\noi
for $ 0 < \dl < 1$.

\end{lemma}

\begin{proof}
We only present the proof of \eqref{3g0a}
since \eqref{3g0b} follows in a similar manner.
Moreover, by the triangle inequality, it suffices to prove \eqref{3g0a}
for $s = 0$.
We proceed as in  the proof of Lemma \ref{LEM:2g}.
By an iterative application of  Lemma \ref{LEM:NLS2}
with \eqref{C1a}, 
we have
\begin{align*}
\|\N^{(3)}_{0}(v)\|_{L^2}
&\leq\Bigg\|\sum_{\TT_2\in\TF(2)}
\intt_{\substack{C_0\cap C_1^c\\
\pmb{\xi}\in\Xi(\TT_2)\\\pmb{\xi}_r=\xi}}
\frac{e^{i(\mu_1+\mu_2)t}}{\mu_1(\mu_1+\mu_2)}\prod_{a\in\TT_2^\infty}v_{\xi_a}\Bigg\|_{L^2}\\
& \les \sum_{T_2\in\TF(2)}
\sum_{\substack{M\geq N\\\text{dyadic}}}
\|\I_{M}(v, v,\I^{\mu_1} _{> 5^3|\mu_1|^{1-\delta}}(v))\|_{L^2}\\
&\lesssim 
\sum_{\substack{M\geq N\\\text{dyadic}}}
M^{-\frac{1}{2}+}\|v\|_{L^2}^2\|\I^M_{\ges 5^3M^{1-\delta}}(v)\|_{L^2}\\
&\les N^{-1+\frac{\delta}{2}+}\|v\|_{L^2}^5,
\end{align*}

\noi
yielding the desired estimate \eqref{3g0a}.
\end{proof}

As in the first step, we decompose $\N^{(3)}$ as 
\[\N^{(3)}=\N_{1}^{(3)}+\N_{2}^{(3)}, \]

\noi
where $\N^{(3)}_{1}$ is the restriction of $\N^{(3)}$ onto
\begin{align}
C_2: =\big\{|\wt{\mu}_3|\leq7^3|\wt{\mu}_2|^{1-\delta}\big\} \cup \big\{|\wt{\mu}_3|\leq7^3|\mu_1|^{1-\delta}\big\}
\label{C2}
\end{align}
	
\noi
and $\N_{2}^{(3)}: =\N^{(3)}- \N_1^{(3)}$.	
At this point,  we have
	\[\N_{> N} =\sum_{j = 2}^3 \dt\N^{(j)}_{0}+\sum_{j = 2}^3\N_{1}^{(j)}+\N_{2}^{(3)}.\]

As before,  the modulation restriction \eqref{C2} 
allows us to estimate 
 the first term $\N_{1}^{(3)}$.

\begin{lemma}\label{LEM:3g}
Let $s\geq 0$.
Then, we have 
\begin{align}\label{3ga}
\|\N^{(3)}_{1}(v)\|_{H^s}
& \les N^{-\frac{1}{2}+}\|v\|^5_{H^s}, \\
\label{3gb}
\|\N^{(3)}_{1}(v)-\N^{(3)}_{1}(w)\|_{H^s}
& \les N^{-\frac{1}{2}+}\big(\|v\|_{H^s}^4+\|w\|_{H^s}^4\big)\|v-w\|_{H^s},
\end{align}

\noi
for $0 < \dl < 1$.

\end{lemma}

\begin{proof}
We only present the proof of \eqref{3ga}
since \eqref{3gb} follows in a similar manner.
Moreover, by the triangle inequality, it suffices to prove \eqref{3ga}
for $s = 0$.
As in the proof of Lemma \ref{LEM:2g}, 
with a slight abuse of notations, we have
\begin{align}
\N_{1}^{(3)}(v)(\xi,t)
& = \sum_{\TT_3\in\TF(3)}
\intt_{\substack{\pmb{\xi}\in\Xi(\TT_3)\\\pmb{\xi}_r=\xi}}
\ind_{C_0} \frac{e^{- i\mu_1t}}{\mu_1}\prod_{a_1\in\pi_1^\infty(\TT_3)}v_{\xi_{a_1}}
\intt_{\substack{\pmb{\xi}^{(2)}\in\Xi(\pi_2(\TT_3))\\\pmb{\xi}^{(2)}_{r^{(2)}}=\xi^{(2)}}}
\ind_{C_1^c} \frac{e^{- i\mu_2t}}{\wt{\mu}_2}\prod_{a_2\in\pi_2^\infty(\TT_3)}v_{\xi_{a_2}}\notag \\
&\hphantom{XX}
\times
\intt_{\substack{\pmb{\xi}^{(3)}\in\Xi(\pi_3(\TT_3))\\\pmb{\xi}^{(3)}_{r^{(3)}}=\xi^{(3)}}}
\ind_{C_2}
e^{- i\mu_3t}\prod_{a_3\in\pi_3^\infty(\TT_3)}v_{\xi_{a_3}}.
\label{3g1}
\end{align}

\noi
Note that the last integral is over three variables
$\{\xi_{a_3}\}_{a_3\in\pi_3^\infty(\TT_3)}$,
while 
the first and second integrals
are over two and two variables
(or one and three variables)
$\{\xi_{a_1}\}_{a_1\in\pi_1^\infty(\TT_3)}$
and $\{\xi_{a_2}\}_{a_2\in\pi_2^\infty(\TT_3)}$, 
with one constraint
$\pmb{\xi}_r=\xi$.

We first consider the case 
$ |\wt{\mu}_3|\leq7^3|\wt{\mu}_2|^{1-\dl} $. 
For each fixed ordered tree $\TT_3 \in \TF(3)$, 
each septilinear term in \eqref{3g1}
can be written as 
\begin{align}
\N_{1}^{(3)}\big|_{\TT_3}
& = 
 \I_{|\mu_1|> N } \Big(v,v, 
\I^{\mu_1}_{|\mu_2 + \mu_1|> 5^3 |\mu_1|^{1-\dl}}
\big(v,v, \N^{\wt{\mu}_2}_{|\mu_3+\wt{\mu}_2|\le 7^3|\wt{\mu}_2|^{1-\delta} }(v,v,v) \big)\Big)
 \label{3g2}
\\
\intertext{or}
\N_{1}^{(3)}\big|_{\TT_3}
& =  \I_{|\mu_1|> N } \Big(
\I^{\mu_1}_{|\mu_2 + \mu_1|> 5^3 |\mu_1|^{1-\dl}}(v,v,v), v,  
\N^{\wt{\mu}_2}_{|\mu_3+\wt{\mu}_2|\le 7^3|\wt{\mu}_2|^{1-\delta} }(v,v,v) \Big)
\label{3g3}
\end{align}

\noi
up to permutations of terminal nodes within a subtree of one generation.
In the following, we only consider \eqref{3g2}
since \eqref{3g3} can be estimated in a similar manner.
By dyadically decomposing $\mu_1$ and $\wt \mu_2$, we have
\begin{align*}
\eqref{3g2}& \sim 
\sum_{\substack{N_1\geq N\\\text{dyadic}}}
\sum_{\substack{N_2 \ges N_1^{1-\dl}\\ \text{dyadic}}}
 \I_{|\mu_1|\sim N_1} \Big(v,v, 
\I^{\mu_1}_{|\wt \mu_2 |\sim N_2 }
\big(v,v, \N^{\wt{\mu}_2}_{|\mu_3+\wt{\mu}_2|\le 7^3|\wt{\mu}_2|^{1-\delta} }(v,v,v) \big)\Big).
\end{align*}

\noi
Then, by Lemmas \ref{LEM:NLS1} and \ref{LEM:NLS2}, 
we can estimate \eqref{3g2} as 
\begin{align*}
\|\eqref{3g2} \|_{L^2} 
&\les  \sum_{\substack{N_1\geq N\\\text{dyadic}}}
\sum_{\substack{N_2 \ges N_1^{1-\dl}\\ \text{dyadic}}}
N_1^{- \frac{1}{2}+} \|v\|_{L^2}^2
\big\|\I^{\mu_1}_{|\wt \mu_2 |\sim N_2 }
\big(v,v, \N^{\wt{\mu}_2}_{|\mu_3+\wt{\mu}_2|\le 7^3|\wt{\mu}_2|^{1-\delta} }(v,v,v) \big)\big\|_{L^2}\\
&\les \sum_{\substack{N_1\geq N\\\text{dyadic}}}
 \sum_{\substack{N_2 \ges N_1^{1-\dl}\\ \text{dyadic}}}
N_1^{- \frac{1}{2}+} N_2^{-\frac{1}{2}+}\|v\|_{L^2}^4
\big\| \ind_{|\wt\mu_2|\sim N_2}\cdot \N^{\wt{\mu}_2}_{|\mu_3+\wt{\mu}_2|\le 7^3|\wt{\mu}_2|^{1-\delta} }(v,v,v)\big\|_{L^2}\\
&\les \sum_{\substack{N_1\geq N\\\text{dyadic}}}
 \sum_{\substack{N_2 \ges N_1^{1-\dl}\\ \text{dyadic}}}
N_1^{- \frac{1}{2}+} N_2^{-\frac{\dl}{2}+}\|v\|_{L^2}^7\\
&\les N^{- \frac{1}{2}- \frac{\dl}{2}+ \frac{\dl^2}{2} + } \|v\|_{L^2}^7.
\end{align*}

Next, we consider the case 
$ |\wt{\mu}_3|\leq7^3|\mu_1|^{1-\delta} $.
In this case, we need to estimate
the terms of the form \eqref{3g2} and \eqref{3g3}
with 
$|\mu_3+\wt{\mu}_2|\le 7^3|\wt{\mu}_2|^{1-\delta} $
replaced by 
${|\mu_3+\wt{\mu}_2|\le 7^3|\mu_1|^{1-\delta} }$.
Proceeding as before with  Lemmas \ref{LEM:NLS1} and \ref{LEM:NLS2}, 
we have
\begin{align*}
\|\eqref{3g2} \|_{L^2} 
&\les  \sum_{\substack{N_1\geq N\\\text{dyadic}}}
 \sum_{\substack{N_2 \ges N_1^{1-\dl}\\ \text{dyadic}}}
N_1^{- \frac{1}{2}+} N_2^{-\frac{1}{2}+}\|v\|_{L^2}^4
\big\| \ind_{|\wt\mu_2|\sim N_2}\cdot
 \N^{\wt{\mu}_2}_{|\mu_3+\wt{\mu}_2|\le 7^3|\mu_1|^{1-\delta} }(v,v,v)\big\|_{L^2}\\
&\les  \sum_{\substack{N_1\geq N\\\text{dyadic}}}
 \sum_{\substack{N_2 \ges N_1^{1-\dl}\\ \text{dyadic}}}
N_1^{- \frac{1}{2}+} N_2^{-\frac{1}{2}+}N_1^{\frac{1}{2}-\frac{\dl}{2}+}\|v\|_{L^2}^7\\
&\les N^{- \frac{1}{2}+} \|v\|_{L^2}^7.
\end{align*}

\noi
This completes the proof of Lemma \ref{LEM:3g}.
\end{proof}

As in the previous step, we can not  estimate 
 $\N_{2}^{(3)}$ in a direct manner.
 Hence, we perform the third step of normal form reductions:
\begin{align*}
\N_{2}^{(3)}(v)(\xi)
& = \dt\Bigg[\sum_{\TT_3\in\TF(3)}
\intt_{\substack{C_0\cap C_1^c\cap C_2^c \\ \pmb{\xi}\in\Xi(\TT_3)\\\pmb{\xi}_r=\xi}}
\frac{e^{- i\wt \mu_3 t}}{ \prod_{j = 1}^3 \wt\mu_j }\prod_{a\in\TT_3^\infty}v_{\xi_a}\Bigg]\\
& \hphantom{XX}
-\sum_{\TT_4\in\TF(4)}
\intt_{\substack{C_0\cap C_1^c\cap C_2^c \\
\pmb{\xi}\in\Xi(\TT_4)\\\pmb{\xi}_r=\xi}} 
\frac{e^{- i\wt \mu_4  t}}{ \prod_{j =1}^3 \wt \mu_j}\prod_{a\in\TT_4^\infty}v_{\xi_a}
\notag \\
& =:\dt \N^{(4)}_{0}(v)(\xi)+ \N^{(4)}(v)(\xi).
\end{align*}

\noi
The boundary term $\N^{(4)}_{0}(v)$
can be estimated as in Lemmas \ref{LEM:2g0} and \ref{LEM:3g0}.
As for $\N^{(4)}(v)$, we decompose it as
$\N^{(4)} = \N^{(4)}_1+ \N^{(4)}_2$
corresponding to the restrictions onto 
\begin{align}
C_3=\big\{|\wt{\mu}_4|\leq9^3|\wt{\mu}_3|^{1-\delta}\big\} \cup \big\{|\wt{\mu}_4|\leq9^3|\mu_1|^{1-\delta}\big\}
\label{C3}
\end{align}

\noi
and its complement $C_3^c$, respectively.
On the one hand, the modulation restriction \eqref{C3} allows
us to estimate 
$\N^{(4)}_1$
as in Lemmas \ref{LEM:2g} and \ref{LEM:3g}.
On the other hand, 
we apply the fourth step of normal form reductions
to $\N^{(4)}_2$.
In this way, we continue  normal form reductions in an indefinite manner.
 In the next subsection, we describe this
 procedure in the general $J$th step.

\subsection{General $J$th step}
\label{SUBSEC:3.3}

In this subsection, we discuss the general $J$th step in this normal form procedure.
Given  an ordered tree $\TT \in \TF(J)$, 
we  introduce the following multilinear operators $\SF_0(\TT; \,\cdot\,) $
and $\SF_1(\TT; \, \cdot\,) $,
which allow us
 to estimate the multilinear terms (associated with the ordered tree $\TT$) in an efficient manner.
For simplicity of notations, we set $M_j$ by 
\begin{align*}
M_j : = \max (|\wt \mu_j |, |\mu_1|).
\end{align*}

\begin{definition}\label{DEF:S} \rm

Let $k = 0, 1$. 
Then, we define 
$\SF_0$
and $\SF_1$
as mappings: 
\[ \TT\in \bigcup_{j = 1}^\infty \TF(j)
\longmapsto 
\text{a $(2j+1)$-linear map } \SF_k(\TT;\,\cdot\,)   \text{ on } \S(\R)^{\otimes2j + 1}, \  k = 0, 1,  \]

\noi
by the following rules. Let $v \in \S(\R)$.
	
\smallskip

\noi
\underline{Definition of $\SF_0(\TT;v)  $:}

\begin{itemize}

\item[(i)] Replace a terminal node (denoted as ``\,$\<1>$\,'') 
by $v$.

\item[(ii)] Replace the $J$th root node $r^{(J)}$  (denoted as ``\,$\<1a>$\,'')  
by  the  trilinear operator $\I^{\wt \mu_{J-1}}_{|\mu_J + \wt \mu_{J-1}| > (2J+1)^3 M_{J-1}^{1-\dl}  } $ 
whose  arguments are given by 
 the functions associated with its three children (namely $v$ in this case).

\item[(iii)] Let $j = J-1$.
Replace the $j$th root node (denoted as ``\,$\<1'>$\,'') 
by the  trilinear operator $\I^{\wt \mu_{j-1}}_{|\mu_j + \wt \mu_{j-1}| > (2j+1)^3 M_{j-1}^{1-\dl}  } $ 
whose  arguments are given by 
 the functions associated with its three children.
Repeat this process for  $j = J-2, J - 3, \dots, 2$.

\item[(iv)] Replace the root  node $r = r^{(1)}$ (denoted as ``\,$\<1''>$\,'') 
by the  trilinear operator $\I_{|\mu_1 | > N  } $ 
whose  arguments are given by 
 the functions associated with its three children.

\end{itemize}

\smallskip

\noi
\underline{Definition of $\SF_1(\TT ;v) $:}

\begin{itemize}

\item[(i)] Replace a terminal node (denoted as ``\,$\<1>$\,'') 
by $v$.

\item[(ii)] Replace the $J$th root node $r^{(J)}$  (denoted as ``\,$\<1a>$\,'')  
by  the  trilinear operator $\N^{\wt \mu_{J-1}}_{|\mu_J + \wt \mu_{J-1}| \leq (2J+1)^3 M_{J-1}^{1-\dl}  } $ 
whose  arguments are given by 
 the functions associated with its three children
 (namely $v$ in this case).

\item[(iii)] Let $j = J-1$.
Replace the $j$th root node (denoted as ``\,$\<1'>$\,'') 
by the  trilinear operator $\I^{\wt \mu_{j-1}}_{|\mu_j + \wt \mu_{j-1}| > (2j+1)^3 M_{j-1}^{1-\dl}  } $ 
whose  arguments are given by 
 the functions associated with its three children.
Repeat this process for  $j = J-2, J - 3, \dots, 2$.

\item[(iv)] Replace the root  node $r = r^{(1)}$ (denoted as ``\,$\<1''>$\,'') 
by the  trilinear operator $\I_{|\mu_1 | > N  } $ 
whose  arguments are given by 
 the functions associated with its three children.

\end{itemize}


\end{definition}

Note that the only difference between the two definitions
appears in Step  (ii).
The operators $\SF_0(\TT;\,\cdot\,)$ and $\SF_1(\TT;\,\cdot\,)$
are a priori defined from $\S(\R)^{\otimes 2j + 1} $ to $\S'(\R)$.
In the following, we show that 
they are bounded on $L^2(\R)$.

\begin{remark}\label{REM:multi}\rm 
In the above definition, 
we only defined $\SF_0(\TT;v)$ and $\SF_1(\TT;v)$, 
namely, when all the $2j+1$ arguments are identical.
Let us now describe how to define 
 $\SF_k(\TT;v_1, \dots, v_{2j+1})$, $k = 0, 1$,  in general.

Given a tree $\TT \in \TF(j)$, 
label its terminal nodes
by $a_1, \dots, a_{2j+1}$
(say, by moving from left to right in the planar graphical representation of the tree).
Given functions $v_1, \dots, v_{2j+1} \in \S(\R)$,
we only need to modify Step  (i) in Definition \ref{DEF:S}
as follows:
\begin{itemize}

\item[(i')] Replace  terminal nodes $a_\l  \in \TT^\infty$
by $v_\l$.
\end{itemize}

\end{remark}


\DeclareSymbol{S1}{0}
 {\draw (0,0) node[dot]{} -- (0,4) node[dia2, label= above:${\ \ r =r^{(1)} }$] {}; 
 \draw (-3,0) node[dot] {} -- (0,4)node[dia2] {} -- (3,0) node[dot] {};
 \draw (3,-4) node[dot]{} -- (3,0) node[ddot, label=right:${\rule[-3mm]{0mm}{5mm}r^{(2)}}$] {}; 
 \draw (0,-4) node[dot] {} -- (3,0)node[ddot] {} -- (6,-4) node[dot] {};
 \draw (6,-8) node[dot]{} -- (6,-4) node[dia3, label=right:${\rule[-3mm]{0mm}{5mm}r^{(3)}}$] {}; 
 \draw (3,-8) node[dot] {} -- (6,-4)node[dia3] {} -- (9,-8) node[dot] {};
 }

\DeclareSymbol{S2}{0}
 {\draw (0,0) node[dot]{} -- (0,4) node[dia2, label= above:${\ \ r =r^{(1)} }$] {}; 
 \draw (-5,0) node[dot] {} -- (0,4)node[dia2] {} -- (5,0) node[dot] {};
 \draw (-5,-4) node[dot]{} -- (-5,0) node[ddot, label=left:${\rule[-3mm]{0mm}{5mm}r^{(2)}}$] {}; 
 \draw (-8,-4) node[dot] {} -- (-5,0)node[ddot] {} -- (-2,-4) node[dot] {};
 \draw (5,-4) node[dot]{} -- (5,0) node[dia3, label=right:${\rule[-3mm]{0mm}{5mm}r^{(3)}}$] {}; 
 \draw (2,-4) node[dot] {} -- (5,0)node[dia3] {} -- (8,-4) node[dot] {};
 }


Before proceeding further, 
let us consider the following examples of ordered trees of the third  generation:
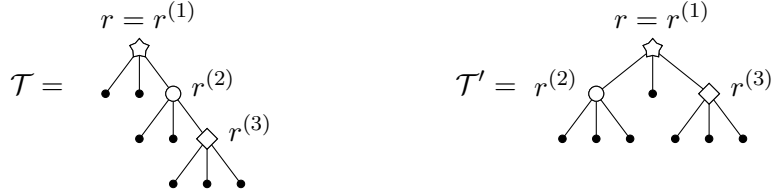
\[ 
\TT = \<S1>
\qquad  \qquad \qquad
\TT' = \<S2>
\]

\medskip
\noi
It is easy to  see that 
$\SF_1(\TT;v)$ and $\SF_1(\TT';  v)$ correspond to 
the septilinear terms 
\eqref{3g2} and 
 \eqref{3g3}, respectively.

Next, let $\BT$ be  the collection 
of formal sums of elements in  
$ \bigcup_{j = 1}^\infty \TF(j)$.
Then, we extend the definitions
of $\SF_0$
and $\SF_1$
to elements in $\BT$
by imposing the ``additivity'':
\begin{align}
 \SF_k\bigg(\sum_{\al \in \mathcal{A}}\TT^\al; \, \cdot \bigg) 
: =  \sum_{\al \in \mathcal{A}} \SF_k(\TT^\al; \,\cdot\,)
\label{X1}
\end{align}

\noi
for a finite index set $\mathcal{A}$.
With this definition, we can write $\N^{(3)}_0 (v)$ and $\N^{(3)}_1 (v)$ from the previous subsection 
as
\begin{align*}
\N^{(3)}_0 (v) = \SF_0\bigg(\sum_{\TT \in \TF(2)} \TT;  v\bigg)
\qquad \text{and}\qquad \N^{(3)}_1 (v) = \SF_1\bigg(\sum_{\TT \in \TF(3)} \TT; v\bigg).
\end{align*}

Now, we are ready to discuss the general $J$th step
of  the normal form reductions.
Define $C_j$ by 
\begin{align}
C_j
& =\big\{|\wt{\mu}_{j+1}|\leq(2j+3)^3M_j^{1-\delta}\big\}\notag \\
& =\big\{|\wt{\mu}_{j+1}|\leq(2j+3)^3|\wt{\mu}_j|^{1-\delta}\big\}
\cup\big\{|\wt{\mu}_{j+1}|\leq(2j+3)^3|\mu_1|^{1-\delta}\big\}
\label{CJ}
\end{align}

\noi
for $j \in \NN$.
Then, after $J$ steps, we have
\begin{align}
\N^{(J)}_{2}(v)(\xi) 
&=\sum_{\TT_{J}\in\TF(J)}
\intt_{\substack{C_0\,\cap\,\bigcap_{j=1}^{J-1}C_j^c\\
\pmb{\xi}\in\Xi(\TT_{J})\\\pmb{\xi}_r=\xi}} 
\frac{e^{- i\wt{\mu}_Jt}}{\prod_{j=1}^{J-1}\wt{\mu}_j}\prod_{a\in\TT_{J}^\infty}v_{\xi_a} \notag \\
&=\dt
\Bigg[\sum_{\TT_{J}\in\TF(J)}
\intt_{\substack{C_0\,\cap\,\bigcap_{j=1}^{J-1}C_j^c\\
\pmb{\xi}\in\Xi(\TT_{J})\\ \pmb{\xi}_r=\xi}} 
\frac{e^{- i\wt{\mu}_Jt}}{\prod_{j=1}^{J}\wt{\mu}_j}\prod_{a\in\TT_{J}^\infty}v_{\xi_a}\Bigg]\notag \\
&\hphantom{X}
 -\sum_{\TT_{J+1}\in\TF(J+1)}
\intt_{\substack{C_0\,\cap\,\bigcap_{j=1}^{J-1}C_j^c\cap C_{J}\\
\pmb{\xi}\in\Xi(\TT_{J+1})\\ \pmb{\xi}_r=\xi}} 
\frac{e^{- i\wt{\mu}_{J+1}t}}{\prod_{j=1}^{J}\wt{\mu}_j}\prod_{a\in\TT_{J+1}^\infty}v_{\xi_a}\notag \\
&\hphantom{X}
-\sum_{\TT_{J+1}\in\TF(J+1)}
\intt_{\substack{C_0\, \cap\, \bigcap_{j=1}^{J}C_j^c\\
\pmb{\xi}\in\Xi(\TT_{J+1})\\ \pmb{\xi}_r=\xi}}
\frac{e^{- i\wt{\mu}_{J+1}t}}{\prod_{j=1}^{J}\wt{\mu}_j}\prod_{a\in\TT_{J+1}^\infty}v_{\xi_a}\notag \\
& =: \dt \N_{0}^{(J+1)}(v)(\xi)
+  \N_{1}^{(J+1)}(v)(\xi) + \N_{2}^{(J+1)}(v)(\xi).
\label{jg}
\end{align}

\noi
As in the previous subsection, let
\begin{align}
 \N^{(J+1)}
:=   \N_{1}^{(J+1)} + \N_{2}^{(J+1)}.
\label{jg'}
\end{align}

In view of Definition \ref{DEF:S}, we have 
\begin{align}
\N^{(J+1)}_0 (v) = \SF_0\bigg(\sum_{\TT \in \TF(J)} \TT;  v\bigg)
\qquad \text{and}\qquad \N^{(J+1)}_1 (v) = \SF_1\bigg(\sum_{\TT \in \TF(J+1)} \TT; v\bigg).
\label{X2}
\end{align}

\noi
In the following, 
we estimate 
$ \N_{0}^{(J+1)}$
and $\N_{1}^{(J+1)}$ for general $J \in \mathbb{N}$.
As for the last term $\N_{2}^{(J+1)}$ in \eqref{jg}, 
we perform a normal form reduction once again
and obtain \eqref{jg} with $J$ replaced by $J+1$. 
In Section \ref{SEC:4}, 
 we show that the remainder term 
$\N_{2}^{(J+1)}$ tends to $0$ in an appropriate sense as $J \to \infty$.

\begin{lemma}\label{LEM:jg0}

Let $s\geq 0$.
Then,  we have
\begin{align}
\label{jg0a}
\|\N^{(J+1)}_{0}(v)\|_{H^s}
& \les N^{-\frac{J}{2}+\frac{J-1}{2}\dl+}\|v\|^{2J+1}_{H^s},  \\
\label{jg0b}
\|\N^{(J+1)}_{0}(v)-\N^{(J+1)}_0(w)\|_{H^s}
&\les 
N^{-\frac{J}{2}+\frac{J-1}{2}\dl+}
\big(\|v\|_{H^s}^{2J}+\|w\|_{H^s}^{2J}\big)\|v-w\|_{H^s}, 
\end{align}

\noi
for $0 < \dl < 1$.
\end{lemma}

\begin{proof}
We only present the proof of \eqref{jg0a}
since \eqref{jg0b} follows in a similar manner.
Note that there is an extra factor $\sim J$ 
when we estimate the difference in \eqref{jg0b}
since
$|a^{2J+1} - b^{2J+1}| \lesssim     
\big(\sum_{j = 1}^{2J+1} a^{2J+1-j}b^{j-1} \big) |a - b |$
has $O(J)$ many terms.
This, however, does not cause a problem since the constant we obtain decays 
like a  power of a factorial in $J$ (as we see below in \eqref{jg03}.)
The same comment applies to Lemma \ref{LEM:jg} below.

Moreover, we claim that  it suffices to prove \eqref{jg0a}
for $s = 0$.
When $s > 0$, 
we argue as follows.
Fix an ordered tree $\TT\in \TF(J)$
and an index function $\pmb{\xi} \in \Xi(\TT)$ with $\pmb{\xi}_r = \xi$.
By the triangle inequality, we have 
$\max_{k = 1, 2, 3} \jb{\xi^{(j)}_k} \geq \frac{1}{3}\jb{\xi^{(j)}}$,
since we have
$\xi^{(j)} = \xi^{(j)}_1 - \xi^{(j)}_2 + \xi^{(j)}_3$.
Hence, 
there exists at least one terminal node $a \in \TT^\infty$
such that 
\[\jb{\xi}^s 
 \leq 3^{Js} \jb{\xi_a}^s.\]

\noi
Note that the constant grows exponentially in $J$.
However, this exponential growth does not cause a problem
thanks to the factorial decay in the denominator 
in \eqref{jg03} below.

From  \eqref{X1} and \eqref{X2}, we have 
\begin{align}
\|\N^{(J+1)}_0 (v) \|_{L^2} \leq  c_J \sup_{\TT\in \TF(J)} \| \SF_0( \TT;  v)\|_{L^2},
\label{jg01}
\end{align}

\noi
where $c_J = |\TF(J)|$ as in \eqref{cj}.
We now decompose $\SF_0( \TT;  v)$
into dyadic pieces in terms of modulations $\wt \mu_j$.
Given dyadic $N_j$, $j = 1, \dots, J$, 
define $\wt M_j $ by 
\begin{align}
\wt M_j : = \max (N_j , N_1).
\label{X0}
\end{align}

\noi
With $\bar N = (N_1, \dots, N_{J})$, 
we define 
$\SF_{0, \bar N}( \TT;  v)$
by making the following modifications
in  Steps (ii), (iii),  and (iv) 
of the definition of $\SF_0(\TT; v)$:

\smallskip

\begin{tabular}{rccc}
(ii) 
& $\I^{\wt \mu_{J-1}}_{|\mu_J + \wt \mu_{J-1}| > (2J+1)^3 M_{J-1}^{1-\dl}  } $ 
& \quad $\LRA$ \qquad 
& $\I^{\wt \mu_{J-1}}_{|\mu_J + \wt \mu_{J-1}| \sim N_J 
} $,\\  

(iii)
& $\I^{\wt \mu_{j-1}}_{|\mu_j + \wt \mu_{j-1}| > (2j+1)^3 M_{j-1}^{1-\dl}  } $ 
& \quad $\LRA$ \qquad 
& $\I^{\wt \mu_{j-1}}_{|\mu_j + \wt \mu_{j-1}| \sim N_j  } $,  \\

(iv)
& $\I_{|\mu_1 | > N  }$
& \quad $\LRA$ \qquad 
& $  \I_{|\mu_1 | \sim N_1  } $ .

\end{tabular}

\smallskip

\noi
Then, we have 
\begin{align}
\SF_0( \TT;  v) \sim 
 \sum_{\substack{N_1\geq  N\\\text{dyadic}}}
\sum_{\substack{N_2 \geq   5^3  \wt M_1^{1-\dl}\\ \text{dyadic}}}
\cdots
\sum_{\substack{N_{J} \geq   (2J+1)^3 \wt M_{J-1}^{1-\dl}\\ \text{dyadic}}}
\SF_{0, \bar N} ( \TT;  v).
\label{jg02}
\end{align}

Fix  an ordered tree $\TT \in \TF(J)$. 
In view of  Remark \ref{REM:order}, 
we can estimate 
$\SF_{0, \bar N} ( \TT;  v)$ 
by applying 
Lemma \ref{LEM:NLS2} in a successive manner
in the following order: 
\[  \I_{|\mu_1 | \sim N_1  }, \
\I^{ \mu_{1}}_{|\mu_2 +  \mu_{1}| \sim  N_2 } ,
\  \I^{\wt \mu_{2}}_{|\mu_3 +  \wt \mu_{2}| \sim N_3 } ,
\dots, \ 
\I^{\wt \mu_{J-1}}_{|\mu_J + \wt \mu_{J-1}| \sim N_J } .\]

\noi
Then, 
it follows from  Lemma \ref{LEM:NLS2} with \eqref{jg01}, \eqref{jg02}, and \eqref{X0}, 
that 
\begin{align}
\|\N^{(J+1)}_0 (v) \|_{L^2} 
& \leq  c_J \sup_{\TT\in \TF(J)} \| \SF_0( \TT;  v)\|_{L^2} \notag\\
& \les 
c_J
 \sum_{\substack{N_1\geq N\\\text{dyadic}}}
\sum_{\substack{N_2 \geq  5^3 \wt M_1^{1-\dl}\\ \text{dyadic}}}
\cdots
\sum_{\substack{N_{J} \geq (2J+1)^3 \wt M_{J-1}^{1-\dl}\\ \text{dyadic}}}
N_1^{-\frac{1}{2}+}
\prod_{j = 2}^{J}N_j^{-\frac12+} \|v\|_{L^2}^{2J+1}\notag \\
& \les 
\frac{c_J}{\prod_{j = 2}^J(2j+1)^{\frac{3}{2}- }}
 \sum_{\substack{N_1\geq N\\\text{dyadic}}}
N_1^{-\frac{J}{2}+\frac{J-1}{2}\dl+}\|v\|_{L^2}^{2J+1}
\label{jg03} \\
&  \les 
N^{-\frac{J}{2}+\frac{J-1}{2}\dl+}\|v\|_{L^2}^{2J+1}.\notag
\end{align}

\noi
This completes the proof of Lemma \ref{LEM:jg0}.
\end{proof}

A similar argument yields the following bounds
on $\N^{(J+1)}_{1}(v)$.

\begin{lemma}\label{LEM:jg}

Let $s\geq 0$.
Then,  we have
\begin{align}
\label{jga}
\|\N^{(J+1)}_{1}(v)\|_{H^s}
& \les 
N^{-\frac{J-1}{2}+\frac{J-2}{2}\dl+}
\|v\|^{2J+3}_{H^s},  \\
\|\N^{(J+1)}_{1}(v)-\N^{(J+1)}_1(w)\|_{H^s}
&\les 
N^{-\frac{J-1}{2}+\frac{J-2}{2}\dl+}
\big(\|v\|_{H^s}^{2J+2}+\|w\|_{H^s}^{2J+2}\big)\|v-w\|_{H^s}, 
\notag
\end{align}

\noi
for $0 < \dl < 1$.
\end{lemma}

\begin{proof}
Arguing as in the proof of Lemma \ref{LEM:jg0}, 
it suffices to prove  \eqref{jga}
for $s = 0$.
From  \eqref{X1} and \eqref{X2}, we have 
\begin{align}
\|\N^{(J+1)}_1 (v) \|_{L^2} \leq  c_{J+1} \sup_{\TT\in \TF(J+1)} \| \SF_1( \TT;  v)\|_{L^2}.
\label{jg1}
\end{align}

\noi
As in the proof of Lemma \ref{LEM:jg0}, 
we decompose $\SF_1( \TT;  v)$
into dyadic pieces in terms of modulations $\wt \mu_j$.
With $\bar N = (N_1, \dots, N_{J+1})$, 
we define 
$\SF_{1, \bar N}( \TT;  v)$
by making the following modifications
in  Steps (ii), (iii),  and (iv) 
of the definition of $\SF_1(\TT; v)$
(with $J$ replaced by $J+1$):

\smallskip

\begin{tabular}{rccc}
(ii) 
& $\N^{\wt \mu_{J}}_{|\mu_{J+1} + \wt \mu_{J}| \leq (2J+3)^3 M_{J}^{1-\dl}  } $ 
& \quad $\LRA$ \qquad 
& $\N^{\wt \mu_{J}}_{|\mu_{J+1} + \wt \mu_{J}| \sim N_{J+1}  } $,\\  

(iii)
& $\I^{\wt \mu_{j-1}}_{|\mu_j + \wt \mu_{j-1}| > (2j+1)^3 M_{j-1}^{1-\dl}  } $ 
& \quad $\LRA$ \qquad 
& $\I^{\wt \mu_{j-1}}_{|\mu_j + \wt \mu_{j-1}| \sim   N_j} $,  \\

(iv)
& $\I_{|\mu_1 | > N  }$
& \quad $\LRA$ \qquad 
& $  \I_{|\mu_1 | \sim N_1  } $ ,

\end{tabular}

\smallskip

\noi
where $\wt M_j$ is as in \eqref{X0}.	
Then, we have 
\begin{align}
\SF_1( \TT;  v) \sim 
 \sum_{\substack{N_1\geq N\\\text{dyadic}}}
\sum_{\substack{N_2 \geq 5^3 \wt M_1^{1-\dl}\\ \text{dyadic}}}
\cdots
\sum_{\substack{N_{J} \geq (2J +1)^3 \wt M_{J-1}^{1-\dl}\\ \text{dyadic}}}
\ \sum_{\substack{N_{J+1} \leq  2^{-1} \cdot(2J +3 )^3 \wt M_{J}^{1-\dl}\\ \text{dyadic}}}
\SF_{1, \bar N} ( \TT;  v).
\label{jg2}
\end{align}

Fix  an ordered tree $\TT \in \TF(J+1)$. 
Proceeding as before, 
we can estimate 
$\SF_{1, \bar N} ( \TT;  v)$ 
by applying 
Lemmas \ref{LEM:NLS1} and \ref{LEM:NLS2} in a successive manner
in the following order: 
\begin{align*}
  \I_{|\mu_1 | \sim N_1  }, \
\I^{ \mu_{1}}_{|\mu_2 +  \mu_{1}| \sim  N_2 } ,
\dots, \ 
\I^{\wt \mu_{J-1}}_{|\mu_J + \wt \mu_{J-1}| \sim N_J  }, 
\ \N^{\wt \mu_{J}}_{|\mu_{J+1} + \wt \mu_{J}| \sim N_{J+1}}.
\end{align*}

We first consider the contribution from the case $\wt M_J \sim N_J $.
It follows from  Lemmas \ref{LEM:NLS1} and \ref{LEM:NLS2}  with \eqref{jg1}, \eqref{jg2}, and \eqref{X0}
that 
\begin{align*}
\|\N^{(J+1)}_1 & (v) \|_{L^2} 
 \leq  c_{J+1} \sup_{\TT\in \TF(J)} \| \SF_1( \TT;  v)\|_{L^2} \notag\\
& \les c_{J+1}
 \sum_{\substack{N_1\geq N\\\text{dyadic}}}
\sum_{\substack{N_2 \geq  5^3 \wt M_1^{1-\dl}\\ \text{dyadic}}}
\cdots
\sum_{\substack{N_{J} \geq (2J+1)^3 \wt M_{J-1}^{1-\dl}\\ \text{dyadic}}}
N_1^{-\frac{1}{2}+}
\prod_{j = 2}^{J}N_j^{-\frac12+}
\notag\\
& \hphantom{XXXXXXXX}
 \times \sum_{\substack{N_{J+1} \leq  2^{-1} \cdot (2J +3 )^3 N_{J}^{1-\dl}\\ \text{dyadic}}}
N_{J+1}^{\frac{1}{2}+}\|v\|_{L^2}^{2J+3}
\notag \\
& \les c_{J+1}(2J+3)^{\frac{3}{2}+}
 \sum_{\substack{N_1\geq N\\\text{dyadic}}}
\sum_{\substack{N_2 \geq  5^3 \wt M_1^{1-\dl}\\ \text{dyadic}}}
\cdots
\sum_{\substack{N_{J} \geq (2J+1)^3 \wt M_{J-1}^{1-\dl}\\ \text{dyadic}}}
N_1^{-\frac{1}{2}+} N_J^{-\frac{\dl}{2}+}
\notag\\
& \hphantom{XXXXXXXX}
 \times 
\prod_{j = 2}^{J-1}N_j^{-\frac12+}\|v\|_{L^2}^{2J+3}
\notag\\
& \les 
\frac{c_{J+1}(2J+3)^{\frac{3}{2}+}}{\prod_{j = 2}^{J-1}(2j+1)^{\frac{3}{2}- }}
 \sum_{\substack{N_1\geq N\\\text{dyadic}}}
N_1^{-\frac{1}{2}+}N_1^{ - \frac{\dl}{2}+\frac{\dl^2}{2}+}
N_1^{-\frac{J-2}2+ \frac{J-2}{2}\dl +}\|v\|_{L^2}^{2J+3}
\notag \\
& \les 
 \sum_{\substack{N_1\geq N\\\text{dyadic}}}
N_1^{-\frac{J-1}{2}+\frac{J-3}{2}\dl+\frac{\dl^2}{2}+}\|v\|_{L^2}^{2J+3}
 \les 
N^{-\frac{J-1}{2}+\frac{J-3}{2}\dl+\frac{\dl^2}{2}+}\|v\|_{L^2}^{2J+3}.
\end{align*}

Next, we consider the contribution from the case $\wt M_J \sim N_1 $.
Proceeding as above, we have
\begin{align*}
\|\N^{(J+1)}_1 & (v) \|_{L^2} 
 \leq  c_{J+1} \sup_{\TT\in \TF(J)} \| \SF_1( \TT;  v)\|_{L^2} \notag\\
& \les 
c_{J+1}
 \sum_{\substack{N_1\geq N\\\text{dyadic}}}
\sum_{\substack{N_2 \geq  5^3 \wt M_1^{1-\dl}\\ \text{dyadic}}}
\cdots
\sum_{\substack{N_{J} \geq (2J+1)^3 \wt M_{J-1}^{1-\dl}\\ \text{dyadic}}}
N_1^{-\frac{1}{2}+}
\prod_{j = 2}^{J}N_j^{-\frac12+}
\notag\\
& \hphantom{XXXXXXXX}
 \times \sum_{\substack{N_{J+1} \leq 2^{-1}\cdot  (2J +3 )^3  N_1^{1-\dl}\\ \text{dyadic}}}
N_{J+1}^{\frac{1}{2}+}\|v\|_{L^2}^{2J+3}
\notag \\
& \les c_{J+1}(2J+3)^{\frac{3}{2}+}
 \sum_{\substack{N_1\geq N\\\text{dyadic}}}
\sum_{\substack{N_2 \geq  5^3 \wt M_1^{1-\dl}\\ \text{dyadic}}}
\cdots
\sum_{\substack{N_{J} \geq (2J+1)^3 \wt M_{J-1}^{1-\dl}\\ \text{dyadic}}}
 N_1^{-\frac{\dl}{2}+}
\prod_{j = 2}^{J}N_j^{-\frac12+}\|v\|_{L^2}^{2J+3}
\notag\\
& \les 
\frac{c_{J+1}(2J+3)^{\frac{3}{2}+}}{\prod_{j = 2}^{J}(2j+1)^{\frac{3}{2}- }}
 \sum_{\substack{N_1\geq N\\\text{dyadic}}}
N_1^{-\frac{J-1}{2}+\frac{J-2}{2}\dl+}\|v\|_{L^2}^{2J+3}
\notag \\
& \les 
N^{-\frac{J-1}{2}+\frac{J-2}{2}\dl+}\|v\|_{L^2}^{2J+3}.
\end{align*}

\noi
This completes the proof of Lemma \ref{LEM:jg}.
\end{proof}

\begin{remark}\rm
As mentioned at the beginning of this section, 
we can perform an analogous analysis for the mKdV \eqref{mKdV}.
In this case, it follows from Lemmas \ref{LEM:KdV1} and \ref{LEM:KdV2}
that 
Lemmas~\ref{LEM:jg0} and \ref{LEM:jg}
hold for $s \geq \frac 14$.
\end{remark}


\subsection{Normal form equation}
\label{SUBSEC:3.4}

After the $J$th step of the normal form reductions, 
we transformed the original equation \eqref{NLS2}
to 
\begin{align}
\dt v (\xi)& = \N_{\leq N}(v)(\xi) + \N_{> N}(v)(\xi)\notag \\
& = \sum_{j=2}^{J+1}
\partial_t\N^{(j)}_{0}(v) (\xi)+\sum_{j=1}^{J+1}\N_{1}^{(j)}(v) (\xi)+\N_{2}^{(J+1)}(v) (\xi).
\label{NF0}
\end{align}

\noi
By iterating this procedure indefinitely, 
we {\it formally}\footnote{This means that 
the derivation can be easily justified for  smooth solutions
but not for rough solutions.
Here, we assume that the remainder term $\N^{(J+1)}_2(v)(\xi)$ tends to $0$
as $J \to \infty$.
In Section \ref{SEC:4}, 
we justify all the computations
for rough functions, namely, 
$u \in C_t H^s$ with  $s \geq \frac16$
for the cubic NLS
and $s> \frac 14$ for the mKdV.}
arrive at the following limit equation:
\begin{equation}\label{NF1}
\partial_t v(\xi) = 
\partial_t\bigg( \sum_{j=2}^{\infty}\N^{(j)}_{0}(v) (\xi)\bigg)+\sum_{j=1}^{\infty}\N_{1}^{(j)}(v) (\xi).
\end{equation}

\noi
Integrating \eqref{NF1} in time and applying the Fourier inversion formula, we 
obtain the following  {\it normal form equation}:
\begin{align} 
v(t) & = \G_{u_0}(v)\notag\\
:  \! & =u_0 + \bigg[\sum_{j=2}^\infty  \N^{(j)}_0 (v(t)) - \sum_{j=2}^\infty\N^{(j)}_0(u_0) \bigg]
+ \int_0^t  \sum_{j=1}^\infty \N^{(j)}_1(v(t')) dt'.
\label{NF2}
\end{align}

\begin{theorem}\label{THM:NF}
The normal form equation \eqref{NF2} is unconditionally locally well-posed in $H^s(\R)$
with 
\begin{align}
\text{\textup{(i)} $s \geq 0$ for the cubic NLS \eqref{NLS}\quad 
and \quad \textup{(ii)} $s \geq \tfrac 14$ for the mKdV \eqref{mKdV}.}
\label{reg}
\end{align}
\end{theorem}

\begin{proof}
Given $u_0 \in H^s(\R)$, let $R = 1 + \|u_0\|_{H^s}$.	
Given $T > 0$, we use $B_{2R}$ to denote
the closed ball of radius $2R$ in $C_TH^s: = C([0, T]; H^s(\R))$ centered at the origin.
By \eqref{1g}, Lemmas \ref{LEM:2g0} 
- \ref{LEM:3g}, 
\ref{LEM:jg0}, and \ref{LEM:jg}, we have
\begin{align}
\|\G_{u_0}  (v)\|_{C_TH^s}  
\leq  \ \|u_0\|_{H^s} & +   
C \sum_{j = 2}^J N^{-\frac{j-1}{2} +\frac{j-2}{2}\dl+} 
\big(\|v\|_{C_TH^s}^{2j-1} + \|u_0\|_{H^s}^{2j-1}\big)\notag \\
  &   + 
CT \bigg\{
 N^{\frac{1}{2}+}\|v\|_{C_TH^s}^3
 +\sum_{j = 2}^J 
 N^{-\frac{j-2}{2} +\frac{j-3}{2}\dl+}
 \|v\|_{C_TH^s}^{2j+1}\bigg\}
 \label{NF3}
\end{align}

\noi
for $s$ satisfying \eqref{reg}.
Note that 
the estimate \eqref{1g} on $\N_1^{(1)}$
is the only estimate with a positive power of $N$.
However, it appears inside the time integral in \eqref{NF2}.
This allows us to (i)~choose $N = N(R) \ges 1$, 
guaranteeing the convergence of the geometric series in \eqref{NF3}
for $v \in B_{2R}$,
and then (ii)~choose $T = T(N) = T(R)>0$ sufficiently small
to control $T  N^{\frac{1}{2}+}\|v\|_{C_TH^s}^3$.
A similar estimate also holds on the difference
$\|\G_{u_0}  (v) - \G_{u_0}  (w)\|_{C_TH^s} $
for $v, w \in B_{2R}$. 
Then, by a standard fixed point argument, 
we can show that the normal form equation~\eqref{NF2} is locally well-posed 
in $C([0,T]; H^s(\R))$, 
provided that  $s$ satisfies \eqref{reg}.
While the fixed point argument yields
this uniqueness only in the ball $B_{2R}\subset C([0,T]; H^s(\R))$, 
we can apply a standard continuity argument
to upgrade uniqueness to that in the entire
 $C([0,T]; H^s(\R))$ (by possibly shrinking
 the local existence time).
 See Remark 2.9 in \cite{CGKO} for example. 
Lastly,  by considering the difference
of two solutions
$v_1, v_2 \in C([0, T]; H^s(\R))$ to \eqref{NF2}, 
we also obtain 
\begin{equation*}\label{NF4}
\|v_1-v_2\|_{C_TH^s} \lesssim  \|v_1(0) -v_2(0) \|_{H^s}
\end{equation*}

\noi
for small $T= T(\|v_1(0)\|_{H^s}, \|v_2(0)\|_{H^s})>0$. 
This proves Theorem \ref{THM:NF}.
\end{proof}

In the following, we sketch the proofs of Theorems \ref{THM:NLS2} and \ref{THM:mKdV2},
assuming that smooth solutions to the cubic NLS \eqref{NLS} (or the mKdV \eqref{mKdV})
  satisfy the normal form equation~\eqref{NF2} (which we will show in the next section).
By  starting with two smooth solutions 
$u_1, u_2 \in C([0, T]; H^\infty(\R))$ to the cubic NLS \eqref{NLS}
(or the mKdV \eqref{mKdV}), 
 the above analysis yields
 \begin{equation}\label{NF5}
\|u_1-u_2\|_{C_TH^s} \lesssim  \|u_1(0) -u_2(0) \|_{H^s}
\end{equation}

\noi
for $s$ satisfying \eqref{reg}.
The difference estimate~\eqref{NF5} in particular
implies the convergence of approximating solutions
(to a unique limiting function),
yielding 
the local well-posedness
in the sense of sensible weak solutions claimed in Theorems~\ref{THM:NLS2}
and~\ref{THM:mKdV2}.
See \cite{OW2} for details.
Furthermore, 
arguing as in~\cite{GKO}, 
we can also show that solutions to the normal form equation~\eqref{NF2}
are indeed weak solutions in the extended sense  to the original equation. 
Since the argument is straightforward, we omit details.

If we justify that  solutions
to the cubic NLS in $C([0, T]; H^s(\R))$, $s\geq \frac{1}{6}$
(and $s> \frac 14$ for the mKdV),  
indeed satisfy the normal form equation \eqref{NF2}, 
then
the difference estimate~\eqref{NF5}
yields uniqueness.
Since our analysis does not involve any auxiliary function spaces, 
such uniqueness is unconditional, thus establishing
Theorems \ref{THM:NLS} and \ref{THM:mKdV}.
In the next section, 
we justify all the steps of the normal form reductions
and thus the derivation of the normal form equation \eqref{NF2}
under the regularity assumption above.

\section{Unconditional well-posedness}\label{SEC:4}

In this section, we present the proof of Theorems \ref{THM:NLS} and \ref{THM:mKdV}.
In view of the (conditional) well-posedness results in $H^s(\R)$:
 $s \geq 0$ for the cubic NLS \cite{Tsu} and $s \geq \frac{1}{4}$ for the mKdV \cite{KPV93,  CKSTT, Kishimoto0},
 we focus on proving unconditional uniqueness, locally in time.
As mentioned above, 
the main task is to make the formal computations in Section \ref{SEC:3} rigorous.
Once this is achieved, the difference estimate \eqref{NF5} yields unconditional uniqueness.
In the following, 
we justify 
\begin{itemize}

\item[(i)] the application of the product rule and 

\item[(ii)] switching time derivatives and integrals in spatial frequencies
 (for each fixed $\xi \in \R$)

\end{itemize}

\noi
in the normal form reductions \eqref{2g}, \eqref{3g}, and \eqref{jg}.
Moreover, we show that 
\begin{itemize}
\item[(iii)]
the remainder term  
 $ \N_{2}^{(J+1)}(v)(\xi)$ in \eqref{NF0} tends to $0$ as $ J\to \infty $
 (for each fixed $\xi \in \R$).\footnote{ This part is not explicitly written in \cite{GKO}.
It is, however, easy to see that  the computation in \cite[(5.3)]{GKO}
 and its generalization for the $J$th generation (which follows
 as a minor modification of \cite[Lemma 3.11]{GKO} with 
 \eqref{NZ0a} below) would imply (iii) for the cubic NLS on  $\T$.}

\end{itemize}

\noi
It is crucial to note that we verify (i) - (iii) 
for each {\it fixed} $\xi \in \R$, 
namely, in a weaker topology
than the $H^s$-topology used in Section \ref{SEC:3}.
Moreover, while all the multilinear estimates (Lemmas \ref{LEM:2g0}
- \ref{LEM:3g}, 
\ref{LEM:jg0}, and \ref{LEM:jg}) 
 for the cubic NLS 
in Section \ref{SEC:3} hold  for  $s\geq 0$, 
we need an extra regularity $s \geq \frac{1}{6}$ 
in justifying (i), (ii), and (iii).
As for the mKdV, 
the regularity $s \geq \frac 14$ suffices
for the multilinear estimates 
in Section \ref{SEC:3},
while a slightly higher regularity $s > \frac{1}{4}$
is needed for 
   justifying the normal form reductions.

\subsection{Unconditional well-posedness for the cubic NLS}
\label{SUBSEC:NLS}

Let $u $ be a solution to \eqref{NLS}
in   $C(\R; H^s(\R))$
for some $s \geq \frac{1}{6}$
and let $v$ be the corresponding  interaction representation defined by \eqref{interaction}.
On the one hand, by Sobolev's inequality, we have
$|u|^2 u \in C(\R; L^1(\R))$.
On the other hand, 
it follows from  \eqref{NLS1a} and \eqref{NLS2} that 
 $\ft v(\xi)$ satisfies
\begin{align*}
 \dt \ft v(\xi, t) = i e^{-i t\xi^2} \F(|u|^2 u)(\xi, t)
 \end{align*}

\noi
for each $\xi \in \R$.
Hence, by Riemann-Lebesgue lemma, we conclude that
$\ft v(\xi)$ is a $C^1$-function in $t$
 for each fixed $\xi \in \R$.
This justifies (i)
the application of the product rule in Section \ref{SEC:3},
provided that $s \geq \frac 16$.

Next, we justify the exchange of time derivatives and  integrals in spatial frequencies.
Before proceeding further, we first need to present several multilinear estimates.
From \eqref{NLS2}
with Hausdorff-Young's inequality,  Sobolev's inequality, and the unitarity 
of the linear propagator $S(t) = e^{-it \dx^2}$, we have 
 \begin{align}
\| \dt v \|_{\F L^{\infty}}
= \| \N(v)  \|_{\F L^{\infty}}
\leq \|u \|_{L^3}^3 
\les \| u \|_{H^\frac{1}{6}}^3 =  \| v \|_{H^\frac{1}{6}}^3, 
\label{NZ0a}
\end{align}

\noi
where the $\F L^\infty$-norm is defined in \eqref{FL}.
Note that the same estimate holds
for $\N^\al _{ M}$ and $\N^\al _{\leq M}$.\footnote{In this case, 
we simply take the absolute values of the Fourier coefficients of each argument
and drop a modulation restriction.
For example, we  have
 \begin{align*}
\| \N^\al _{ M} (v) \|_{\F L^{\infty}}
\leq \big\|  \F^{-1} (|\ft v|)  \big\|_{L^3}^3 
\les \big\| \F^{-1} (|\ft v|)  \big\|_{H^\frac{1}{6}}^3 =  \| v \|_{H^\frac{1}{6}}^3, 
\end{align*}

\noi
where we used the fact that the $H^\frac{1}{6}(\R)$ is a Fourier lattice in the last step.}

We also need the following $\F L^\infty$-estimates, i.e.~uniform estimates in spatial frequencies.
Our main goal is to prove Lemma \ref{LEM:NLS5} below,
controlling the $\F L^\infty$-norms of the multilinear terms
$\N^{(J+1)}(v)$ and $\N_2^{(J+1)}(v)$ in terms of the $H^\frac{1}{6}$-norm
of $v$.

\begin{lemma}[Localized modulation estimates for the cubic NLS in the weak norm]
\label{LEM:NLS3}
Given $\eps > 0$, there exists $C_\eps > 0$  such that 
\begin{align}
\label{NZ1}
\|\N^\al _{\leq M}(v_1, v_2, v_3) \|_{\F L^{\infty}}
& \leq C_\eps 
M^\frac{1}{2}
\min_{j = 1, 2, 3} \bigg(
\|v_j\|_{\F L^{\infty}}
  \prod_{\substack{k = 1\\k\ne j}}^3
\|v_k\|_{H^\eps} \bigg)
\end{align}

\noi
for any $M\geq 1$ and $\al  \in \R$.
\end{lemma}

%

\begin{proof}
By duality,  \eqref{NZ1} follows once we prove\footnote{Recall 
our convention of denoting $\ft v(\xi)$ by $v(\xi)$ when there is no confusion.
See Remark \ref{REM:tdepend}.}
 \begin{align}
\Bigg|
\intt_{\xi = \xi_1 - \xi_2 + \xi_3} 
\ind_{|\Phi(\bar{\xi})-\al |\leq M}
& v_1(\xi_1) v_2(\xi_2)v_3(\xi_3)v_4(\xi) d\xi_1 d\xi_2 d\xi
\Bigg|\notag\\
& \lesssim 
M^\frac{1}{2}
\min_{j = 1, 2, 3} \bigg(\|v_j\|_{L^\infty_\xi}  \prod_{\substack{k = 1\\k\ne j}}^3 \| \jb{\xi}^\eps v_k\|_{L^2_\xi}\bigg)
 \|v_4\|_{L^1_\xi}
\label{NZ2}
\end{align}

\noi
for all non-negative functions $v_1, \dots, v_4\in L^2_\xi(\R)$.

\medskip

\noi
$\bullet$ {\bf Case 1:}  $\max(|\xi_{2-1}|, |\xi_{2-3}|)\leq 1$.
\\
\indent
Let $\zeta = \xi_2 - \xi_1$
and $\wt \zeta = \xi_2 - \xi_3$. 
Then, thanks to the restriction $|\zeta|, |\wt \zeta|\leq 1$, we have
\begin{align*}
\text{LHS of }\eqref{NZ2}
& \leq \sup_\xi  \bigg|\int_{|\zeta|\leq 1}  
\int_{|\wt \zeta|\leq 1}  
 v_1(\xi + \wt \zeta) v_2 (\xi + \zeta + \wt \zeta )  v_3 (\xi + \zeta)  d\wt \zeta d\zeta\bigg|
\cdot \|v_4\|_{L^1_\xi} \\
& \les
\min_{j = 1, 2, 3} \bigg(\|v_j\|_{L^\infty_\xi}  \prod_{\substack{k = 1\\k\ne j}}^3 \|v_k\|_{L^2_\xi}\bigg)
\|v_4\|_{L^1_\xi}. 
\end{align*}

\noi
$\bullet$ \textbf{Case 2:} 
 $\max(|\xi_{2-3}|, |\xi_{2-1}|)>1$.
\\
\indent
Without loss of generality, assume that $\xi - \xi_3 >1$.
Proceeding as in \eqref{A2a} with  \eqref{A2}, we have 
\begin{align}
\text{LHS of }\eqref{NZ2}
& \leq \bigg\|
\intt_{\xi = \xi_1 - \xi_2 + \xi_3} 
\ind_{|\Phi(\bar{\xi})-\al |\leq M}
v_1(\xi_1)v_2(\xi_2)v_3(\xi_3) d\xi_1 d\xi_3\bigg\|_{L^\infty_{\xi}}\|v_4 \|_{L^1_\xi}\notag\\
& \leq \sup_{\xi}
\bigg(\intt_{\xi = \xi_1 - \xi_2 + \xi_3} 
\ind_{|\Phi(\bar{\xi})-\al |\leq M} \jb{\xi_3}^{-\eps} d\xi_1 d\xi_3 \bigg)^\frac{1}{2}
\|v_1\|_{L^2_\xi} \|v_2\|_{L^\infty_\xi} \|\jb{\xi}^\eps v_3\|_{L^2_{\xi}}
 \|v_4\|_{L^1_\xi}\notag\\
& \leq \sup_{\xi}
 \bigg(\int_{\xi - \xi_3 > 1} 
\frac{M}{(\xi-\xi_3)\jb{\xi_3}^\eps}d\xi_3 \bigg)^\frac{1}{2}
\|v_1\|_{L^2_\xi} \|v_2\|_{L^\infty_\xi} \|\jb{\xi}^\eps v_3\|_{L^2_{\xi}}
 \|v_4\|_{L^1_\xi}\notag\\
& \les M^\frac{1}{2}
\|v_1\|_{L^2_\xi} \|v_2\|_{L^\infty_\xi} \|\jb{\xi}^\eps v_3\|_{L^2_{\xi}}
 \|v_4\|_{L^1_\xi}.
\label{NZ2a1}
\end{align}

\noi
An analogous computation 
yields
\begin{align}
\text{LHS of }\eqref{NZ2}
& \les M^\frac{1}{2}
\|v_1\|_{L^\infty_\xi} \|v_2\|_{L^2_\xi} \|\jb{\xi}^\eps v_3\|_{L^2_{\xi}}
 \|v_4\|_{L^1_\xi}.
\label{NZ2a2}
\end{align}

\noi
Lastly, with $\jb{\xi - \xi_3}^\eps = \jb{\xi_1 - \xi_2}^\eps
\les  \jb{\xi_1}^\eps \jb{\xi_2}^\eps$, we have
\begin{align}
\text{LHS of }\eqref{NZ2}
& \leq \bigg\|
\intt_{\xi = \xi_1 - \xi_2 + \xi_3}
\ind_{|\Phi(\bar{\xi})-\al |\leq M}
v_1(\xi_1)v_2(\xi_2)v_3(\xi_3) d\xi_1 d\xi_3\bigg\|_{L^\infty_{\xi}}\|v_4 \|_{L^1_\xi}\notag\\
& \leq \sup_{\xi}
\bigg(\intt_{\xi = \xi_1 - \xi_2 + \xi_3} 
\ind_{|\Phi(\bar{\xi})-\al |\leq M} \jb{\xi - \xi_3}^{-\eps} d\xi_1 d\xi_3 \bigg)^\frac{1}{2}
\bigg(\prod_{j = 1}^2 \|\jb{\xi}^\eps v_j\|_{L^2_\xi} \bigg) \| v_3\|_{L^\infty_{\xi}}
 \|v_4\|_{L^1_\xi}\notag\\
& \leq \sup_{\xi}
 \bigg(\int_{\xi - \xi_3 > 1} 
\frac{M}{(\xi-\xi_3)\jb{\xi - \xi_3}^\eps}d\xi_3 \bigg)^\frac{1}{2}
\bigg(\prod_{j = 1}^2 \|\jb{\xi}^\eps v_j\|_{L^2_\xi} \bigg) \| v_3\|_{L^\infty_{\xi}}
 \|v_4\|_{L^1_\xi}\notag\\
& \les M^\frac{1}{2}
\bigg(\prod_{j = 1}^2 \|\jb{\xi}^\eps v_j\|_{L^2_\xi} \bigg) \| v_3\|_{L^\infty_{\xi}}
 \|v_4\|_{L^1_\xi}.
\label{NZ2a3}
\end{align}

\noi
Putting \eqref{NZ2a1}, \eqref{NZ2a2}, and \eqref{NZ2a3},
we obtain \eqref{NZ2} in this case.
\end{proof}

As a corollary to Lemma \ref{LEM:NLS3}, we have the following
estimates on $\I^\al_M$ and $\I^\al_{> M}$.

\begin{lemma}\label{LEM:NLS4}
Given $\eps > 0$, there exists $C_\eps > 0$  such that 
\begin{align*}
\|\I^\al _{ M}(v)\|_{\F L^{ \infty}}& \leq C_\eps M^{-\frac{1}{2}}
\|v\|_{H^\eps}^2 \|v\|_{\F L^{ \infty}},\\
\|\I^\al _{> M}(v)\|_{\F L^{\infty}}& \leq C_\eps  M^{-\frac{1}{2}}
\|v\|_{H^\eps}^2 \|v\|_{\F L^{\infty}}, 
\end{align*}

\noi
for any $M\geq 1$ and $\al  \in \R$.

\end{lemma}

Let $\N^{(J+1)}$ be as in \eqref{jg'}.
Then,  from \eqref{jg}, we have
\begin{align}
\N^{(J+1)}(v)(\xi) 
&= \sum_{\TT_{J+1}\in\TF(J+1)}
\intt_{\substack{C_0\, \cap\, \bigcap_{j=1}^{J-1}C_j^c\\
\pmb{\xi}\in\Xi(\TT_{J+1})\\ \pmb{\xi}_r=\xi}}
\frac{e^{- i\wt{\mu}_{J+1}t}}{\prod_{j=1}^{J}\wt{\mu}_j}\prod_{a\in\TT_{J+1}^\infty}v_{\xi_a}\notag \\
& 
=  
\sum_{\TT_{J}\in\TF(J)}
\sum_{b \in\TT_{J}^\infty}
\intt_{\substack{C_0\,\cap\,\bigcap_{j=1}^{J-1}C_j^c\\
\pmb{\xi}\in\Xi(\TT_{J})\\ \pmb{\xi}_r=\xi}} 
\frac{e^{- i\wt{\mu}_Jt}}{\prod_{j=1}^{J}\wt{\mu}_j}
\, \dt v_{\xi_b}\prod_{a\in\TT_{J}^\infty\setminus \{b\} }v_{\xi_a}.
\label{jg'1}
\end{align}

\noi
Now, given $\TT_{J}\in\TF(J)$, 
we label its terminal nodes by $a_1, \dots, a_{2J+1}$.
Then, it follows from Definition \ref{DEF:S}
and \eqref{jg'1}  with 
\eqref{X1} and  Remark \ref{REM:multi}
 that 
\begin{align}
\N^{(J+1)}(v) 
& = 
\sum_{\TT_{J}\in\TF(J)}
\sum_{k = 1}^{2J+1}
\SF_0\bigg( \TT_J;  {\bf v}_k\bigg), 
\label{jg'2}
\end{align}

\noi
where 
${\bf v}_k = (v, \dots, v, \underbrace{\dt v}_{k\text{th spot}}, v, \dots, v)$.
Compare this with $\N_0^{(J+1)}$ in \eqref{jg} and  \eqref{X2}.

\begin{lemma}\label{LEM:NLS5}
Let $\N^{(J+1)}(v)$ be as in \eqref{jg'}.
Then, 
we have
\begin{align}
\|\N^{(J+1)}(v)\|_{\F L^\infty}
& \les N^{-\frac{J}{2}+\frac{J-1}{2}\dl}\|v\|_{H^\frac{1}{6}}^{2J+ 3}, 
\label{jg'3}\\
\|\N_2^{(J+1)}(v)\|_{\F L^\infty}
& \les N^{-\frac{J}{2}+\frac{J-1}{2}\dl} \|v\|_{H^\frac{1}{6}}^{2J+3}, 
\label{jg'5}
\end{align}

\noi
for $0< \dl < 1$.

\end{lemma}

\begin{proof}
We use the representation \eqref{jg'2}
for $\N^{(J+1)}(v)$.
Proceeding as in the proof of Lemma~\ref{LEM:jg0}\footnote{Note that we have an $O(J)$ loss due to the summation in $k$.  This, however, does not cause any trouble thanks to the fast decay in \eqref{jg03}.}
with Lemma \ref{LEM:NLS4}, we have 
\begin{align}
\|\N^{(J+1)}(v)\|_{\F L^\infty}
& \les N^{-\frac{J}{2}+\frac{J-1}{2}\dl}\|v\|^{2J}_{H^\eps} \|\dt v\|_{\F L^\infty}. 
\label{jg'4}
\end{align}

\noi
Then, \eqref{jg'3} follows from \eqref{jg'4} and 
\eqref{NZ0a}.
The second estimate \eqref{jg'5}
differs from \eqref{jg'3}
only in the modulation restriction $C_J^c$  for $\dt v$ 
(viewed as a cubic term).
Noting that the product estimate \eqref{NZ0a} holds even with 
the modulation restriction, 
we see that the second estimate \eqref{jg'5} follows in an analogous manner.
\end{proof}

\begin{remark}\label{REM:abs}\rm
Note that we do not make use of the oscillatory factor
in establishing the estimates \eqref{NZ0a}, \eqref{jg'3}, and \eqref{jg'5}.
In particular, the integrals in spatial frequencies in 
\eqref{NZ0a}, \eqref{jg'3}, and \eqref{jg'5} converge absolutely.
\end{remark}

Let us now consider the first step of the normal form reductions.
By rearranging \eqref{2g}, 
we have
\begin{align*}
\dt\Bigg[\intt_{\substack{\pmb{\xi}\in\Xi(\TT_1)\\\pmb{\xi}_r=\xi}}
& \ind_{C_0} \frac{e^{-i\mu_1 t}}{\mu_1}\prod_{a\in\TT_1^\infty}v_{\xi_a}\Bigg]
=
\intt_{\substack{\pmb{\xi}\in\Xi(\TT_1)\\\pmb{\xi}_r=\xi}}
\dt \bigg(\ind_{C_0} \frac{e^{-i\mu_1 t}}{\mu_1}\prod_{a\in\TT_1^\infty}v_{\xi_a}\bigg)
\notag\\
& = \intt_{\substack{\pmb{\xi}\in\Xi(\TT_1)\\\pmb{\xi}_r=\xi}}\ind_{C_0} e^{-i\mu_1 t}
\prod_{a\in\TT_1^\infty}v_{\xi_a}
+ 
\sum_{\TT_1\in\TF(1)}
\intt_{\substack{\pmb{\xi}\in\Xi(\TT_1)\\\pmb{\xi}_r=\xi}} 
\ind_{C_0}\frac{e^{-i\mu_1t}}{\mu_1}
\dt \bigg(\prod_{a\in\TT_1^\infty}v_{\xi_a}\bigg)\notag \\
& = \N^{(1)}_2(v)(\xi) - \N^{(2)}(v)(\xi).
\end{align*}

\noi
Then, in view of 
Lemma \ref{LEM:NLS5}
with   Remark \ref{REM:abs}, 
we can justify the switching of 
 the time derivative and
the integration in the first equality above
by applying 
Fubini's theorem  
to the  integrated (in time) formulation.


Similarly, 
by rearranging \eqref{jg}, we have
\begin{align*}
\dt
\Bigg[\sum_{\TT_{J}\in\TF(J)} & 
\intt_{\substack{C_0\,\cap\,\bigcap_{j=1}^{J-1}C_j^c\\
 \pmb{\xi}\in\Xi(\TT_{J})\\ \pmb{\xi}_r=\xi}} 
 \frac{e^{- i\wt{\mu}_Jt}}{\prod_{j=1}^{J}\wt{\mu}_j}\prod_{a\in\TT_{J}^\infty}v_{\xi_a}\Bigg]\notag \\
& = 
\sum_{\TT_{J}\in\TF(J)}
\intt_{\substack{C_0\,\cap\,\bigcap_{j=1}^{J-1}C_j^c\\
\pmb{\xi}\in\Xi(\TT_{J})\\ \pmb{\xi}_r=\xi}} 
\dt \bigg(\frac{e^{- i\wt{\mu}_Jt}}{\prod_{j=1}^{J}\wt{\mu}_j}\prod_{a\in\TT_{J}^\infty}v_{\xi_a}\bigg)\notag \\
& = 
\N^{(J)}_{2}(\xi) 
-  \N^{(J+1)}(\xi) 
\end{align*}

\noi
in the general case. 
Once again,  
in view of  
Lemma \ref{LEM:NLS5}
with   Remark \ref{REM:abs}, 
we can justify the switching of 
 the time derivative and
the integration in the first equality above
by applying 
Fubini's theorem  
to the  integrated (in time) formulation.


Lastly, 
it follows from \eqref{jg'5} in Lemma \ref{LEM:NLS5} that, for each fixed $\xi \in \R$,  
the remainder term $\N_2^{(J+1)}(v) (\xi) $ tends to $0$
as $J \to \infty$, provided that $u \in C(\R; H^\frac{1}{6}(\R))$.
This justifies the derivation of the normal form equation\footnote{At this point, 
the normal form equation \eqref{NF2} is justified
only for each fixed $\xi \in \R$
 on the Fourier side. 
In view of Lemmas~\ref{LEM:jg0}, and \ref{LEM:jg}, we can a posteori  show that the normal form equation~\eqref{NF2} 
indeed holds in $C([0, T]; H^\frac{1}{6}(\R))$. See the proof of Proposition 2.1 in \cite{OW2}.
A similar comment applies to the mKdV.
}
\eqref{NF2}
and hence the difference estimates \eqref{NF4} and~\eqref{NF5}
for the cubic NLS.
By iterating the local-in-time argument, this yields unconditional uniqueness in the class 
$C(\R; H^\frac{1}{6}(\R))$.
This completes the proof of Theorem~\ref{THM:NLS}.

\begin{remark} \rm

As in \cite{GKO}, 
it is also possible to justify the exchange of  time derivatives and  integrals in spatial frequencies
in the distributional sense 
under a milder regularity assumption that  $v\in C(\R; L^2(\R))$.
Given a family  $\{f_\xi  \}_{\xi \in \R}$ of temporal distributions in $\mathcal D'_t$.
 we {\it define} $\int f_{\xi} d\xi \in \mathcal D'_t$ by 
\begin{align}
 \bigg\langle\int  f_{\xi} d\xi , \phi\bigg\rangle
: =  \int \jb{f_\xi, \phi} d\xi
\label{dis1}
\end{align}

\noi
for $\phi \in \mathcal{D}_t$, 
provided that the integral on the right-hand side is well defined (in the Lebesgue sense)
for each $\phi \in \mathcal D_t$.
Then, as a distributional derivative,  
$\dt \int f_{\xi} d\xi \in \mathcal D'_t$ is given by 
\begin{align*}
 \bigg\langle \dt \int  f_{\xi} d\xi , \phi\bigg\rangle 
&  = - \bigg\langle\int  f_{\xi} d\xi , \dt \phi\bigg\rangle
 \stackrel{\eqref{dis1}}{=} -  \int \jb{f_\xi,\dt  \phi} d\xi
 =   \int \jb{\dt f_\xi,  \phi} d\xi\\
&  \!\!\!\! 
 \stackrel{\eqref{dis1}}{=}   \bigg\langle \int \dt  f_{\xi} d\xi , \phi\bigg\rangle ,
 \end{align*}

\noi
provided that 
 $\int f_{\xi} d\xi $ is well defined in the sense of \eqref{dis1}.
 Namely, we have 
\begin{align}
 \dt \int  f_{\xi} d\xi
= \int \dt  f_{\xi} d\xi 
\label{dis3}
\end{align}

\noi
as elements in $\mathcal{D}'_t$,
as long as  $\int f_{\xi} d\xi $ exists. 
Compare this with Lemma 5.1 in \cite{GKO}.
As usual, we have
\begin{align}
 \bigg\langle\int  f_{\xi} d\xi , \phi\bigg\rangle
 =  \int \jb{f_\xi, \phi} d\xi
  =  \iint f_\xi(t)  \phi(t) dt  d\xi
\label{dis2}
\end{align}

\noi
for locally integrable functions $f_\xi(t)$.

Now, let us consider the exchange of
 the time differentiation and the integration in spatial frequencies in \eqref{2g}.
Lemma \ref{LEM:2g0} with $s = 0$
states that, for almost every $\xi \in \R$,  
the integral 
\[
\N_0^{(2)}(v)(\xi)
= \intt_{\substack{\pmb{\xi}\in\Xi(\TT_1)\\\pmb{\xi}_r=\xi}}
\ind_{C_0}
\frac{e^{-i\mu_1 t}}{\mu_1}\prod_{a\in\TT_1^\infty}v_{\xi_a}(t)
=: \intt_{\substack{\pmb{\xi}\in\Xi(\TT_1)\\\pmb{\xi}_r=\xi}}X(\pmb{\xi}, t) \]

\noi
converges absolutely and uniformly on compact time intervals, 
if $v \in C(\R; L^2(\R))$.
Then, for almost every $\xi \in \R$, we have

\begin{align*}
\jb{\N_0^{(2)}(v)(\xi),  \phi}
&  = \bigg\langle \intt_{\substack{\pmb{\xi}\in\Xi(\TT_1)\\\pmb{\xi}_r=\xi}}
X(\pmb{\xi}), \phi\bigg\rangle 
\stackrel{\eqref{dis1}}{=} 
\intt_{\substack{\pmb{\xi}\in\Xi(\TT_1)\\\pmb{\xi}_r=\xi}}
\langle X(\pmb{\xi}), \phi\rangle \notag \\
& \!\!\!\! 
 \stackrel{\eqref{dis2}}{=} 
\intt_{\substack{\pmb{\xi}\in\Xi(\TT_1)\\\pmb{\xi}_r=\xi}}
\int
X(\pmb{\xi}, t) \phi(t) dt\\
& = \int \N_0^{(2)}(v)(\xi, t)  \phi(t) dt, 
\end{align*}


\noi
where the last equality follows from Lemma \ref{LEM:2g0}
and Fubini's theorem, 
since the right-hand side  is absolutely integrable for $v \in C(\R; L^2(\R))$ and $\phi \in \mathcal{D}_t$.
This in particular show that 
\[\intt_{\substack{\pmb{\xi}\in\Xi(\TT_1)\\\pmb{\xi}_r=\xi}}
X(\pmb{\xi})\]

\noi
is well defined as an integral of temporal distributions in the sense described above.
Therefore, from \eqref{dis3}, we conclude that,  for almost every $\xi \in \R$,   
\[
\dt \bigg[
\intt_{\substack{\pmb{\xi}\in\Xi(\TT_1)\\\pmb{\xi}_r=\xi}}
\ind_{C_0}
\frac{e^{-i\mu_1 t}}{\mu_1}\prod_{a\in\TT_1^\infty}v_{\xi_a}\bigg]
= \intt_{\substack{\pmb{\xi}\in\Xi(\TT_1)\\\pmb{\xi}_r=\xi}}
\ind_{C_0}\dt \bigg(
\frac{e^{-i\mu_1 t}}{\mu_1}\prod_{a\in\TT_1^\infty}v_{\xi_a}\bigg).\]

\noi
in the (temporal) distributional sense.
A similar argument can be used to justify  the exchange of the time differentiation and the integration 
in the $J$th step of the normal form reductions in this mild sense,
 provided that $v\in C(\R; L^2(\R))$.
We, however, point out that 
the justification of (i) and (iii) 
requires a higher regularity of $s \geq \frac 16$,
which is sufficient for switching  time derivatives
and integrals in the usual sense.

\end{remark}

\subsection{Unconditional well-posedness for the mKdV}
\label{SUBSEC:mKdV}

In this subsection, we discuss
the proof of Theorem~\ref{THM:mKdV}.
As in Subsection \ref{SUBSEC:NLS}, our goal is to justify (i), (ii), and (iii)
in the normal form reductions for the mKdV \eqref{mKdV}.
While the structure of the argument 
follows closely  that of the proof of  
Theorem~\ref{THM:NLS},
we need to reformulate the problem 
in order to handle the derivative in the nonlinearity.

Given  a solution $u $ to \eqref{mKdV}, 
 let $v$ be the corresponding interaction representation defined by \eqref{interaction}.
It follows from \eqref{mKdV2} and \eqref{KdV2} that 
 $\ft v(\xi)$ satisfies
\begin{align*}
 \dt \ft v(\xi, t) = - i \xi e^{i t\xi^3} \F(u^3)(\xi, t)
 \end{align*}

\noi
for each $\xi \in \R$.
Arguing as in Subsection \ref{SUBSEC:NLS}, 
we see that 
$\ft v(\xi)$ is a $C^1$-function in $t$
 for each {\it fixed} $\xi \in \R$, 
 provided that 
$u \in   C(\R; H^\frac{1}{6}(\R))\subset C(\R; L^3(\R))$.
This justifies (i)
the application of the product rule in the normal form reductions, 
provided that $s \geq \frac 16$.

Next, we discuss the issues (ii) and (iii).
For this purpose, we first need to review the normal form reductions
in Section \ref{SEC:3}.
By writing out the first step \eqref{2g}
of the normal form reductions for the mKdV, we have
\begin{align}
\N^{(1)}_2(v)(\xi,t)
&=\intt_{\substack{\pmb{\xi}\in\Xi(\TT_1)\\\pmb{\xi}_r=\xi}}\ind_{C_0} \xi e^{i\mu_1 t}
\prod_{a\in\TT_1^\infty}v_{\xi_a}\notag \\
&=\partial_t\Bigg[\intt_{\substack{\pmb{\xi}\in\Xi(\TT_1)\\\pmb{\xi}_r=\xi}}
\ind_{C_0} \xi^{(1)} 
\frac{e^{i\mu_1 t}}{\mu_1}\prod_{a\in\TT_1^\infty}v_{\xi_a}\Bigg] \notag\\
& \hphantom{XXXXX}
-\sum_{\TT_2\in\TF(2)}
\intt_{\substack{\pmb{\xi}\in\Xi(\TT_2)\\\pmb{\xi}_r=\xi}} 
\ind_{C_0}
\bigg(\prod_{j = 1}^2 \xi^{(j)}\bigg)
\frac{e^{i(\mu_1+\mu_2)t}}{\mu_1}\prod_{a\in\TT_2^\infty}v_{\xi_a}\notag \\
& = \dt \N_{0}^{(2)}(v)(\xi,t)
+ \N^{(2)}(v)(\xi,t).
\label{m2}
\end{align}

\noi
The main issue here is the derivative loss in the last generation.
More precisely, 
 an analogue of the $\F L^\infty$-estimate \eqref{NZ0a} on $\dt v$
does not hold for the mKdV, even if we use the $H^\frac{1}{4}$-norm on the right-hand side.
We instead have the following lemma.

\begin{lemma}\label{LEM:Extra}
The following estimate holds: 
\begin{equation}
\bigg|\intt_{\xi=\xi_1+\xi_2+\xi_3}|\xi|^\frac{1}{4}v_1(\xi_1)v_2(\xi_2)v_3(\xi_3)d\xi_1d\xi_2\bigg|
\les \prod_{j = 1}^3 \|v_j\|_{H^\frac{1}{4}}
\label{Extra}
\end{equation}

\noi
for any $\xi \in \R$.
\end{lemma}

\begin{proof}
Without loss of generality, 
assume that $|\xi_1|\ges|\xi|$
and set  $w_1(\xi_1)=\jb{\xi_1}^\frac{1}{4}v_1(\xi_1)$.
Then,  by Hausdorff-Young's inequality followed by Sobolev's inequality, we have
\begin{align*}
\text{LHS of } \eqref{Extra}
&\les \big\| \F^{-1}(|w_1|*|v_2|*|v_3|) \big\|_{L^1_x}\\
&\leq \| \F^{-1}(| w_1|)\|_{L^2}\| \F^{-1}(|v_2|)\|_{L^4}\| \F^{-1}(|v_3|)\|_{L^4}\\
&\leq \prod_{j = 1}^3 \|v_j\|_{H^\frac{1}{4}}.
\end{align*}

\noi
This proves Lemma \ref{LEM:Extra}.
\end{proof}

\begin{remark}\rm
Let $\s_0 \geq 0$.
Then, proceeding as in the proof of Lemma \ref{LEM:Extra}, 
we have
\begin{equation*}
\sup_{\xi}\bigg|\intt_{\xi=\xi_1+\xi_2+\xi_3}|\xi|^{\s_0}v_1(\xi_1)v_2(\xi_2)v_3(\xi_3)d\xi_1d\xi_2\bigg|
\les \prod_{j = 1}^3 \|v_j\|_{H^\s}
\end{equation*}

\noi
for $\s \geq \max(\s_0, \frac 14)$.
Note that the regularity restriction on $\s$ is sharp by considering the case
$|\xi_1| \sim |\xi| \gg |\xi_2|, |\xi_3|$ and its permutations.
In particular, when $\s = \frac{1}{4}$, 
we can absorb precisely  $\frac{1}{4}$-power of $|\xi|$ in this trilinear estimate.
\end{remark}

On the one hand, Lemma \ref{LEM:Extra} shows that 
we can absorb $\frac 14$-derivative in the second generation.
On the other hand, we still need to handle the remaining $\frac 34$-derivative.
We can resolve this  issue
by reformulating the normal form reductions
as follows.
By the construction, 
we have  $\xi^{(2)}  = \xi^{(1)}_k $ for some $k \in \{1, 2, 3\}$.
See \eqref{tree3} and Definitions \ref{DEF:tree2} and~\ref{DEF:tree3}.
Hence, we can rewrite \eqref{m2} as 
\begin{align*}
\N^{(1)}_2(v)(\xi,t)
&=|\xi|^\frac{3}{4} \intt_{\substack{\pmb{\xi}\in\Xi(\TT_1)\\\pmb{\xi}_r=\xi}}\ind_{C_0}  \text{sgn}(\xi)|\xi|^\frac{1}{4}e^{i\mu_1 t}
\prod_{a\in\TT_1^\infty}v_{\xi_a}\notag \\
&=|\xi|^\frac{3}{4} \cdot \partial_t\Bigg[\intt_{\substack{\pmb{\xi}\in\Xi(\TT_1)\\\pmb{\xi}_r=\xi}}
\ind_{C_0} \text{sgn}(\xi)|\xi|^\frac{1}{4}
\frac{e^{i\mu_1 t}}{\mu_1}\prod_{a\in\TT_1^\infty}v_{\xi_a}\Bigg] \notag\\
& \hphantom{X}
- |\xi|^\frac{3}{4} \sum_{\TT_1\in\TF(1)}
\sum_{p^{(1)} \in \TT_1^\infty}
\intt_{\substack{\pmb{\xi}\in\Xi(\TT_1)\\\pmb{\xi}_r=\xi}} 
\ind_{C_0}\text{sgn}(\xi)|\xi|^\frac{1}{4}
|\xi_{p^{(1)}}|^\frac{3}{4} \\
& \hphantom{XXXXXXXXXXXXXX}
\times 
\frac{e^{i\mu_1t}}{\mu_1}
\M(v)(\xi_{p^{(1)}})
\prod_{a\in\TT_1^\infty \setminus \{p^{(1)}\}}v_{\xi_a}\notag \\
& = \partial_t \N_{0}^{(2)}(v)(\xi,t)
+ \N^{(2)}(v)(\xi,t), 
\end{align*}

\noindent
where $\text{sgn}(\xi)=\pm1$  denotes the sign\footnote{When $\xi = 0$, 
we have $\N^{(1)}_2(v)(\xi,t) = 0$ and hence we may assume $\xi \ne 0$.}
of $\xi$ and $\M(v) = \M(v, v, v)$ is defined by 
\begin{align}
\M(v_1, v_2, v_3) (\xi, t) 
:=  - i\intt_{\xi = \xi_1+\xi_2+\xi_3}  \text{sgn}(\xi)|\xi|^\frac{1}{4}e^{i\Psi(\bar{\xi})t}v(\xi_1)v(\xi_2)v(\xi_3)d\xi_1 d\xi_2.
\label{m3a}
\end{align}

In particular, we have {\it shifted 
$\frac 34$-derivative up by one generation}
so that there is only $\frac 14$-derivative in the second generation,
for which Lemma \ref{LEM:Extra} is applicable.
Similarly, with \eqref{tree3} and Remark \ref{REM:order}, 
we can express $\N^{(J+1)}$  appearing in the $J$th step
as
\begin{align}
& \N^{(J+1)}  (v)(\xi) \notag \\
&\hphantom{X}
 = \sum_{\TT_{J+1}\in\TF(J+1)}
\intt_{\substack{C_0\, \cap\, \bigcap_{j=1}^{J-1}C_j^c\\
 \pmb{\xi}\in\Xi(\TT_{J+1})\\ \pmb{\xi}_r=\xi}}
\bigg(\prod_{j = 1}^{J+1} \xi^{(j)}\bigg)
\frac{e^{ i\wt{\mu}_{J+1}t}}{\prod_{j=1}^{J}\wt{\mu}_j}\prod_{a\in\TT_{J+1}^\infty}v_{\xi_a}\notag \\
&\hphantom{X}
=  |\xi|^\frac{3}{4} 
\sum_{\TT_{J}\in\TF(J)}
\sum_{p^{(J)} \in\TT_{J}^\infty}
\intt_{\substack{C_0\,\cap\,\bigcap_{j=1}^{J-1}C_j^c\\
\pmb{\xi}\in\Xi(\TT_{J})\\ \pmb{\xi}_r=\xi}} 
\text{sgn}(\xi)|\xi|^\frac{1}{4} \notag\\
&\hphantom{XXXX} \times \bigg(\prod_{j \in \#P(r^{(1)}, p^{(J)})
 \setminus \{1\}} |\xi_{p^{(j-1)}}|^\frac{3}{4}\cdot \text{sgn}(\xi_{p^{(j-1)}})|\xi_{p^{(j-1)}}|^\frac{1}{4}\bigg)
\notag\\
&\hphantom{XXXX}
\times
\bigg( \prod_{j \notin \#P(r^{(1)}, p^{(J)})} \xi_{r^{(j)}}\bigg) |\xi_{p^{(J)}}|^\frac{3}{4}
\frac{e^{ i\wt{\mu}_Jt}}{\prod_{j=1}^{J}\wt{\mu}_j}
\M(v)(\xi_{p^{(J)}})
\prod_{a\in\TT_{J}^\infty\setminus \{p^{(J)}\} }v_{\xi_a}.
\label{m4}
\end{align}

\noi
Here, $P(r^{(1)}, p^{(J)})$ is the shortest path
from $r^{(1)}$ to $p^{(J)}$ 
defined in Remark \ref{REM:order}
and $\#P(r^{(1)}, p^{(J)})$ is defined by 
\begin{align}
\#P(r^{(1)}, p^{(J)}) = \big\{ j \in \{1, \dots, J\}: r^{(j)} \in P(r^{(1)}, p^{(J)})\big\}.
\label{m5}
\end{align}

\noi
Note that $1 \in \#P(r^{(1)}, p^{(J)})$.

In view of \eqref{jg}, 
we see that 
$\N^{(J+1)}_1$
and $\N^{(J+1)}_2$
are given by 
\eqref{m4} 
after modifying the frequency restriction
to 
$C_0\cap\bigcap_{j=1}^{J-1}C_j^c\cap C_{J}$
and 
$C_0\cap\bigcap_{j=1}^{J}C_j^c$,
respectively.
Then, we have the following $\xi$-dependent estimates, 
replacing the $\F L^\infty$-estimates in Lemma \ref{LEM:NLS5} 
for the cubic NLS.

\begin{lemma}\label{LEM:mKdV1}
Let $s > \frac 14$.
Then, we have
\begin{align}
|\N^{(J+1)}(v)(\xi)|
& \les |\xi|^\frac{3}{4} N^{-\frac{J}{3}+\frac{J-1}{3}\dl+}\|v\|^{2J}_{H^s} \|v\|_{H^\frac{1}{4}}^3, 
\label{mj1}\\
|\N_2^{(J+1)}(v)(\xi)|
& \les |\xi|^\frac{3}{4} N^{-\frac{J}{3}+\frac{J-1}{3}\dl+}\|v\|^{2J}_{H^s} \|v\|_{H^\frac{1}{4}}^3, 
\label{mj2}
\end{align}

\noindent
for $0< \dl < 1$ and $\xi \in \R$.
\end{lemma}

We present the proof of this lemma in the next subsection.
On the one hand, the estimates in 
Lemma \ref{LEM:mKdV1} 
depend on $\xi \in \R$.
On the other hand, we only need to justify 
the normal form reductions for each {\it fixed} $\xi \in \R$.
Hence, this $\xi$-dependence does not cause any trouble.
In fact, once we have Lemma \ref{LEM:mKdV1},  
we can proceed as in Subsection \ref{SUBSEC:NLS}
and justify
\begin{itemize}
\item[(ii)] switching time derivatives and integrals in spatial frequencies
and 
\item[(iii)]
the remainder term  
 $ \N_{2}^{(J+1)}(v)(\xi)$  tends to $0$ as $ J\to \infty $
 (for each fixed $\xi \in \R$)
\end{itemize}

\noi
in the normal form reductions.

\subsection{Proof of Lemma \ref{LEM:mKdV1}}
\label{SUBSEC:4.3}

We conclude this paper by presenting the proof of 
Lemma \ref{LEM:mKdV1}. 
We first need to introduce new trilinear operators.
For $j \in \{1, 2, 3\}$, $M\geq 1$, and $\al \in \R$, 
define trilinear operators $ \N^\al _{j, \leq M}$ 
and $ \I^\al _{j,  M}$ by 
\begin{align}
 \N^\al _{j, \leq M}(v_1,v_2,v_3)(\xi,t)
&:= \intt_{\substack{\xi = \xi_1+\xi_2+\xi_3\\|\Psi(\bar{\xi})-\al |\leq M}}
|\xi|^\frac{1}{4} |\xi_j|^\frac{3}{4} e^{i\Psi(\bar{\xi}) t}
v_1(\xi_1)v_2(\xi_2)v_3(\xi_3)d\xi_1d\xi_2,  \notag \\
 \I^\al _{j,  M}(v_1,v_2,v_3)(\xi,t)
& := \intt_{\substack{\xi = \xi_1+\xi_2+\xi_3\\|\Psi(\bar{\xi})-\al |\sim M }} 
|\xi|^\frac{1}{4} |\xi_j|^\frac{3}{4} \frac{  e^{i\Psi(\bar{\xi}) t}}{\Psi(\bar{\xi})-\al }v_1(\xi_1)v_2(\xi_2)v_3(\xi_3)d\xi_1d\xi_2.
\label{mk0}
\end{align}

\noi
As in Section \ref{SEC:2}, 
we also define
$ \N^\al _{j,  M}$, 
 $ \N^\al _{j, > M}$, 
and $ \I^\al _{j, > M}$
in an analogous manner.

\begin{lemma}[Localized modulation estimates for the mKdV in the weak norm]
\label{LEM:mk1}
Let $s > \frac 14$.
Then, we have
\begin{align}
\|\N^{\al} _{j, \leq M}(v_1,v_2,v_3)\|_{\F L^\infty}
& \lesssim \max\{|\alpha|,M\}^\frac{1}{12} M^{\frac{1}{2}}
\|v_j\|_{\F L^{\infty}}
  \prod_{\substack{k = 1\\k\ne j}}^3
\|v_k\|_{H^s}
\label{mk1}
\end{align}

\noindent
for any $j \in \{1, 2, 3\}$, $M\geq 1$, and $\al  \in \R$,
where the implicit constant is independent of 
$j \in \{1, 2, 3\}$.
\end{lemma}

We postpone the proof of Lemma \ref{LEM:mk1} to the end of this section.
As an immediate  corollary to Lemma \ref{LEM:mk1}, we have the following
estimates on $\I^\al_{j, M}$ and $\I^\al_{j,> M}$.
\begin{lemma}\label{LEM:mk2}
Let $s > \frac 14$.
Then, we have
\begin{align*}
  \|\I^\al _{j,  M}(v_1, v_2, v_3)\|_{\F L^{ \infty}}& \lesssim \max\{|\alpha|,M\}^\frac{1}{12} M^{-\frac{1}{2}}
\|v_j\|_{\F L^{\infty}}
  \prod_{\substack{k = 1\\k\ne j}}^3
\|v_k\|_{H^s}, \\
 \|\I^\al _{j, > M}(v_1, v_2, v_3)\|_{\F L^{\infty}}& \lesssim \max\{|\alpha|,M\}^\frac{1}{12} M^{-\frac{1}{2}}
\|v_j\|_{\F L^{\infty}}
  \prod_{\substack{k = 1\\k\ne j}}^3
\|v_k\|_{H^s}, 
\end{align*}

\noi
for any $j \in \{1, 2, 3\}$, $M\geq 1$, and $\al  \in \R$, 
where the implicit constant is independent of 
$j \in \{1, 2, 3\}$.
\end{lemma}

Now, we are ready to prove Lemma \ref{LEM:mKdV1}
(and hence Theorem \ref{THM:mKdV}), assuming Lemma \ref{LEM:mk1}.
Given $\TT \in \TF(J)$, we first define $\wt \SF_0(\TT, v)$
by making the following modifications
in  Steps (ii), (iii),  and (iv) 
of the definition of $\SF_0(\TT; v)$ in Definition \ref{DEF:S}:

\medskip

\begin{itemize}
\item[]
\hspace{-8.5mm}
(ii) and (iii):
Let $j = 2, \dots, J$.
Recall the definitions of  $\#P(r^{(1)}, p^{(J)})$ 
and the order $\#p^{(j)}$  of $p^{(j)}$
from  \eqref{m5}
and  Definition \ref{DEF:tree3}.

\end{itemize}

\medskip

\begin{itemize}
\item[$\bullet$]
 If $j \in \#P(r^{(1)}, p^{(J)})$, then
we make the following change:

\medskip
 $\I^{\wt \mu_{j-1}}_{|\mu_j + \wt \mu_{j-1}| > (2j+1)^3 M_{j-1}^{1-\dl}  } $ 
 \quad $\LRA$ \qquad 
 $\I^{\wt \mu_{j-1}}_{\#p^{(k)}, |\mu_j + \wt \mu_{j-1}| > (2j+1)^3 M_{j-1}^{1-\dl}  } $,\\ 

\noi	
where $ \I^\al _{j,  M}$ is as in \eqref{mk0} 
and $p^{(k)}$ is the unique node such that
$p^{(k)} \in \pi_j(\TT)^\infty \cap P(r^{(1)}, p^{(J)})$.

\medskip

\item[$\bullet$] If $j \notin \#P(r^{(1)}, p^{(J)})$, 
we do not make any modification.

\medskip

\item[(iv):]

$\I_{|\mu_1 | > N  }$
 \quad $\LRA$ \qquad 
 $|\xi|^\frac{3}{4} \I_{\#p^{(k)}, |\mu_1 | > N  }$,

\medskip

\noi
where $p^{(k)}$ is the unique node such that
$p^{(k)} \in \pi_1(\TT)^\infty\cap P(r^{(1)}, p^{(J)})$.

\end{itemize}

\medskip

Given $\TT\in\TF(J)$, 
we label its terminal nodes by $a_1, \dots, a_{2J+1}$.
Then, it follows from the definition above for $\wt \SF_0(\TT; \, \cdot\, )$
and \eqref{m4}  with 
\eqref{X1} and  Remark \ref{REM:multi}
 that 
\begin{align}
\N^{(J+1)}(v) 
& = 
\sum_{\TT\in\TF(J)}
\sum_{k = 1}^{2J+1}
\wt \SF_0\bigg( \TT;  {\bf v}_k\bigg), 
\label{md3}
\end{align}

\noi
where 
${\bf v}_k = ( v, \dots,  v, \underbrace{\M(v)}_{k\text{th spot}},  v, \dots,  v)$.
Compare this with \eqref{jg'2}.

Lemma \ref{LEM:Extra} with \eqref{m3a}
yields
\begin{align}
\|\M(v)\|_{\F L^\infty} \les \| v\|_{H^\frac{1}{4}}^3.
\label{md4}
\end{align}

\noi
Then, 
a slight modification\footnote{In particular, 
in \eqref{jg03}, 
we replace
$N_1^{-\frac{1}{2}+}
\prod_{j = 2}^{J}N_j^{-\frac12+} $
by 
\[N_1^{-\frac{1}{2}+\frac{1}{12}+}
\prod_{j = 2}^{J}\big( N_j^{-\frac12+} 
\max (N_{j-1}^{\frac{1}{12}}, N_{j}^{\frac{1}{12}})\big)
\leq \prod_{j = 1}^{J} N_j^{-\frac13+} .\]
}
 of  the proof of Lemma \ref{LEM:jg0}
with
Lemmas \ref{LEM:KdV2} and \ref{LEM:mk2}
and \eqref{md4} yields
the first estimate \eqref{mj1} in 
Lemma \ref{LEM:mKdV1}.
The only difference between 
$\N^{(J+1)}(v)$ and 
$\N^{(J+1)}_2(v)$ 
is the modulation restriction $C_J^c$ in the last generation.
In particular, 
$\N^{(J+1)}_2(v)$ 
can be expressed as \eqref{md3}
with an extra  modulation restriction on $\M(v)$.
Since the proof of Lemma \ref{LEM:Extra} remains true
 with such a modulation restriction, 
the second estimate \eqref{mj2} in Lemma \ref{LEM:mKdV1} follows
in an analogous manner.
This completes the proof of Lemma \ref{LEM:mKdV1}
and hence the proof of Theorem \ref{THM:mKdV},
assuming 
Lemma \ref{LEM:mk1}.

\medskip

It remains to prove Lemma \ref{LEM:mk1}.
The remaining part of this paper is devoted to the proof of Lemma \ref{LEM:mk1}.

\begin{proof}[Proof of Lemma \ref{LEM:mk1}] 
For convenience, let $A_j = \{ 1, 2, 3\} \setminus \{j\}$.
Then, by duality, \eqref{mk1} follows once we prove 
\begin{align}
\Bigg|
\intt_{\xi = \xi_1 + \xi_2 + \xi_3} 
\ind_{|\Psi(\bar{\xi})-\al |\leq M}
& \cdot |\xi|^\frac{1}{4}|\xi_j|^\frac{3}{4}  \prod_{k = 1}^3 v_k(\xi_k) v_4(\xi) 
d\xi_1d\xi_2d\xi\Bigg| 
\notag\\
& \lesssim 
\max\{|\alpha|,M\}^\frac{1}{12} M^\frac{1}{2}, 
\|v_j\|_{\F L^{\infty}}
\bigg(  \prod_{k \in A_j}
\|v_k\|_{H^s}\bigg)
\|v_4\|_{\F L^1}
\label{mk2}
\end{align}

\noindent
for all non-negative functions $v_1, \dots, v_4$
and $j \in \{ 1, 2, 3\}$.
Given $s \geq \frac 14$ and $j \in \{1, 2, 3\}$, define $m_{s,j}(\bar \xi)$ by
\begin{align}
m_{s,j}(\bar \xi)=\frac{|\xi|^\frac{1}{4}|\xi_j|^{\frac{3}{4}}}{\prod_{k\in A_j} \jb{\xi_k}^{s}}.
\label{MK}
\end{align}

\noi
When $s = \frac 14$, we simply denote $m_{\frac{1}{4},j}$ by $m_j$.
By a variant of the Cauchy-Schwarz argument,  we have 
\begin{align}
\text{LHS of }\eqref{mk2}
& \leq \bigg\|
\intt_{\xi = \xi_1 + \xi_2 + \xi_3} 
\ind_{|\Psi(\bar{\xi})-\al |\leq M}\cdot |\xi|^\frac{1}{4}|\xi_j|^\frac{3}{4}
\prod_{k = 1}^3 v_k(\xi_k)\bigg\|_{L^\infty_{\xi}}\|v_4 \|_{L^1_\xi}\notag\\
& \leq 
\sup_{\xi}
\bigg(\intt_{\xi = \xi_1 + \xi_2 + \xi_3} 
\ind_{|\Psi(\bar{\xi})-\al |\leq M} \cdot  m^2_{s,j}(\bar\xi)d\xi_1d\xi_2\bigg)^\frac{1}{2} \notag\\
& \hphantom{XXXXXXXXX}
\times
\|v_j\|_{\F L^{\infty}}
  \prod_{k \in A_j}
\|v_k\|_{H^s} \|v_4\|_{L^1_\xi}.
\label{mk3}
\end{align}

\noindent
Hence, it suffices to show that 
\begin{align}
\sup_{\xi}
\bigg(\intt_{\xi = \xi_1 + \xi_2 + \xi_3} 
\ind_{|\Psi(\bar{\xi})-\al |\leq M}  \cdot m_{s,j}^2(\bar\xi)d\xi_1d\xi_2
 \bigg)^\frac{1}{2}
\les \max\{|\alpha|,M\}^\frac{1}{12} M^\frac{1}{2}, 
\label{mk4}
\end{align}
\noindent
uniformly in $j \in \{1, 2, 3\}$ for $s>\frac{1}{4}$.

While 
we do not explicitly state so, 
it is understood that 
all the estimates and statements in the following 
hold uniformly in $j \in \{1, 2, 3\}$.
Also, we will see that the estimate \eqref{mk1} in fact holds 
at the endpoint regularity $s=\frac{1}{4}$  for many of the following cases.
For those cases, by monotonicity of $\jb{\xi}^s$ in $s$,
it suffices to prove \eqref{mk1} for $s = \frac 14$.

\medskip
\noindent
$\bullet$ {\bf Case 1:}  $|\xi|\les 1$, $|\xi_j|\les 1$.
\\
\indent
In this case,  we prove \eqref{mk4} with $s = \frac 14$.
Without loss of generality, we assume  $j=1$.

\medskip
\noindent
{\bf Subcase 1.a:} $|\xi_{23}|$ and $|\xi_{2-3}|\gtrsim 1$.
\\
\indent
By viewing $\Psi$ as a function of $\xi_2$
for fixed  $\xi$ and $\xi_1$, 
we have
 $|\dd_{\xi_2}\Psi(\bar \xi)|\sim|\xi_{23}||\xi_{2-3}|\gtrsim 1$.
Then, with 
$|m_1(\bar \xi)| \les 1$, we have
\begin{align*}
\text{LHS of }\eqref{mk4}
& \les 
\sup_{\xi}
\bigg(\int_{|\xi_1| \les 1} 
\ind_{|\Psi(\bar{\xi})-\al |\leq M} d\xi_2 d\xi_1 \bigg)^\frac{1}{2}
\les M^{\frac{1}{2}}.
\end{align*}

\noindent
{\bf Subcase 1.b:} $|\xi_{23}|$ or $|\xi_{2-3}|\ll 1$.
\\
\indent
In this case, we have 
$\jb{\xi_2} \sim \jb{\xi_3}$.
When $|\xi_2|, |\xi_3|\les 1$, 
it is easy to  see that that the left-hand side of \eqref{mk4} is $O(1)$
with   $|m_1(\bar \xi)| \les 1$ and  integration in $\xi_2$ and $\xi_3$.
Now, suppose that  $|\xi_2|\sim|\xi_3|\gg 1$.
By viewing $\Psi$ as a function of $\xi_2$
for fixed  $\xi$ and $\xi_3$, 
we have
 $|\dd_{\xi_2}\Psi(\bar \xi)|\sim|\xi_{12}||\xi_{1-2}| \sim |\xi_2|^2\gg  1$
 since $|\xi_1| \les 1 \ll |\xi_2|$.
Then, with 
$|m_1(\bar \xi)| \les \jb{\xi_2}^{-\frac 12}$,  we have
\begin{align*}
\text{LHS of }\eqref{mk4}
& \les 
\sup_{\xi}
\bigg( \sum_{\substack{ N\gg1  \\\text{dyadic}}} \frac 1N\intt_{|\xi_2|\sim|\xi_3|\sim N}
\ind_{|\Psi(\bar{\xi})-\al |\leq M} d\xi_2 d\xi_3 \bigg)^\frac{1}{2}\\
& \les 
\bigg( \sum_{\substack{ N\gg1  \\\text{dyadic}}} \frac{M}{N^3} N \bigg)^\frac{1}{2}
\les M^{\frac{1}{2}}.
\end{align*}

\medskip
\noindent
$\bullet$ {\bf Case 2:}  $|\xi|\gg 1$, $|\xi_j|\les 1$.
\\
\indent
In this case,  we prove \eqref{mk4} with $s = \frac 14$.
We denote $A_j = \{1, 2, 3\} \setminus \{j\} = \{k_1, k_2\}$.

\medskip
\noindent
{\bf Subcase 2.a:} $|\xi_{k_1}|\sim|\xi_{k_2}|\ges|\xi|\gg 1\ges |\xi_j|$, where $k_1$, $k_2\in A_j$.
\\
\indent
By viewing $\Psi$ as a function of $\xi_{k_1}$
for fixed  $\xi$ and $\xi_{k_2}$, 
we have
$|\dd_{\xi_{k_1}}\Psi(\bar \xi)|\sim |\xi_{j k_1}||\xi_{j - k_1}|
\sim |\xi_{k_2}|^2$.
Then, with $|m_j(\bar\xi)|\les |\xi_{k_2}|^{-\frac{1}{4}}$,  we have
	\begin{align*}
\text{LHS of }\eqref{mk4}
& \les 
\sup_{\xi}
\bigg(\int_{|\xi_{k_1}|\sim |\xi_{k_2}|\gg  1} 
\ind_{|\Psi(\bar{\xi})-\al |\leq M} \cdot\frac{1}{|\xi_{k_2}|^\frac{1}{2}}
d\xi_{k_1} d\xi_{k_2} \bigg)^\frac{1}{2}
\les M^\frac{1}{2}.
	\end{align*}

\noindent
{\bf Subcase 2.b:} $|\xi_{k_1}|\sim|\xi|\gg \max( |\xi_{k_2}|, |\xi_j|, 1)$, where $k_1$, $k_2\in A_j$.

In this case, $|\xi_{k_1k_2}|\sim|\xi_{k_1}|$ and $|\xi_{k_1-k_2}|\sim|\xi_{k_1}|$.
By viewing $\Psi$ as a function of $\xi_{k_1}$
for fixed  $\xi$ and $\xi_{j}$, 
we have
$|\dd_{\xi_{k_1}}\Psi(\bar \xi)|\sim |\xi_{k_1k_2}||\xi_{ k_1-k_2}|\sim |\xi|^2$.
Then, with $|m_j(\bar \xi)|\les1$, we have
\begin{align*}
\text{LHS of }\eqref{mk4}
& \les 
\sup_{\xi}
\bigg(\int_{|\xi_j| \les 1} 
\ind_{|\Psi(\bar{\xi})-\al |\leq M} d\xi_{k_1} d\xi_{j} \bigg)^\frac{1}{2}
\les M^\frac{1}{2}.
\end{align*}

\medskip
\noindent
$\bullet$ {\bf Case 3:}  $|\xi|\les 1$, $|\xi_j|\gg 1$.
\\
\indent
We prove \eqref{mk4} with $s = \frac 14$ in Subcases 3.a and 3.b,
while Subcase 3.c requires $s > \frac 14$.
In Subcases 3.a and 3.b, we only need the condition
$|\xi|\ll |\xi_j|$ and their relative sizes with respect to $1$ is not important.

\medskip
\noindent
{\bf Subcase 3.a:} $|\xi_{k_1}|\sim|\xi_{k_2}|\gg|\xi_j|\gg|\xi|$, where $k_1$, $k_2\in A_j$.
\\
\indent
Let  $s = \frac 14$.
We have $|\xi_{jk_2}|\sim|\xi_{j-k_2}|\sim|\xi_{k_1}|$.
By viewing $\Psi$ as a function of $\xi_{k_2}$
for fixed  $\xi$ and $\xi_{k_1}$, 
we have
 $|\dd_{\xi_{k_2}}\Psi(\bar \xi)|\sim |\xi_{jk_2}||\xi_{j-k_2}|\sim|\xi_{k_1}|^2$.
Then, we have
	\begin{align*}
\text{LHS of }\eqref{mk4}
& \les 
\sup_{\xi}
\bigg(\int_{|\xi_j| \les 1} 
\ind_{|\Psi(\bar{\xi})-\al |\leq M} \cdot\frac{|\xi|^\frac{1}{2}|\xi_{j}|^\frac{3}{2}}{|\xi_{k_1}|}d\xi_{k_2} d\xi_{k_1} \bigg)^\frac{1}{2}\\
&\les M^\frac{1}{2}\sup_{\xi}
\bigg(\int_{|\xi_{k_1}|\gg |\xi|} 
\frac{|\xi|^\frac{1}{2}}{|\xi_{k_1}|^\frac{3}{2}} d\xi_{k_1} \bigg)^\frac{1}{2}\lesssim M^\frac{1}{2}.
\end{align*}

\noindent
{\bf Subcase 3.b:} $|\xi_{1}|\sim|\xi_{2}|\sim |\xi_3|\gg|\xi|$.
\\
\indent
Let  $s = \frac 14$.
We have $m_j(\bar \xi)\lesssim |\xi|^\frac{1}{4} |\xi_1|^\frac{1}{4}$ and $|\xi_{12}|\sim|\xi_{23}|\sim|\xi_{31}|\sim|\xi_1|$.
We claim that  $\max\{|\xi_{1-2}|,|\xi_{2-3}|,|\xi_{3-1}|\}\ges |\xi_1|$.
Otherwise, i.e.~if $\max\{|\xi_{1-2}|,|\xi_{2-3}|,|\xi_{3-1}|\}\ll|\xi_1|$, 
then $\xi_1$, $\xi_2$, and $\xi_3$ must have the same sign and thus $|\xi|=|\xi_1+\xi_2+\xi_3|\sim|\xi_1|$, 
leading to a contradiction.
Without loss of generality, we  assume  $|\xi_{2-3}|\sim|\xi_1|$.
By viewing $\Psi$ as a function of $\xi_{3}$
for fixed  $\xi$ and $\xi_{1}$, 
we have
 $|\dd_{\xi_{3}}\Psi(\bar \xi)|\sim |\xi_{23}||\xi_{2-3}|\sim|\xi_{1}|^2$.
Hence, we have
\begin{align*}
\text{LHS of } \eqref{mk4} 
& \les \sup_{\xi}
\bigg(\intt_{\xi = \xi_1 + \xi_2 + \xi_3} 
\ind_{|\Psi(\bar{\xi})-\al |\leq M}\cdot
|\xi|^\frac{1}{2}|\xi_1|^\frac{1}{2}d\xi_3 d\xi_1 \bigg)^\frac{1}{2} 
\\
& \les M^\frac{1}{2}
\sup_{\xi} 
\bigg(\int_{|\xi_1|\gg|\xi|}\frac{|\xi|^\frac{1}{2}}{|\xi_1|^\frac{3}{2}}d\xi_1 \bigg)^\frac{1}{2} 
\les  M^{\frac{1}{2}}.
	\end{align*}

\noindent
{\bf Subcase 3.c:} $|\xi_{k_1}|\sim|\xi_j|\gg |\xi_{k_2}|$, $|\xi|$, where $k_1$, $k_2\in A_j$.
\\
\indent
In this case, we have $|\xi_{k_1k_2}|\sim|\xi_{k_1-k_2}|\sim|\xi_j|$.
By viewing $\Psi$ as a function of $\xi_{k_1}$
for fixed  $\xi$ and $\xi_{j}$, 
we have
 $|\dd_{\xi_{k_1}}\Psi(\bar \xi)|\sim |\xi_{k_1k_2}||\xi_{k_1 - k_2}|\sim|\xi_{j}|^2$.
Hence, with  $|m_{j,s}(\bar \xi)|\lesssim|\xi|^\frac{1}{4}|\xi_j|^{\frac{3}{4}-s}$, we have
	\begin{align*}
\text{LHS of }\eqref{mk4}
& \les 
\sup_{|\xi|\les 1}
\bigg(\int_{|\xi_j| \gg 1} 
\ind_{|\Psi(\bar{\xi})-\al |\leq M} |\xi|^\frac{1}{2}|\xi_j|^{\frac{3}{2}-2s} d\xi_{k_1} d\xi_{j} \bigg)^\frac{1}{2}\\
& \les 
M^\frac{1}{2}
\sup_{|\xi|\les 1}
\bigg(\int_{|\xi_j| \gg 1} 
\frac{1}{|\xi_j|^{\frac{1}{2}+2s}}  d\xi_{j} \bigg)^\frac{1}{2}
\les M^\frac{1}{2}, 
	\end{align*}
	
\noi
provided that  $s>\frac14$.

\medskip
In the remaining part of the proof, we split 
the case $|\xi|$, $|\xi_j|\gg 1$ into three subcases, 
depending on the sizes of 
$|\xi_{12}|, |\xi_{23}|,$ and $|\xi_{31}|$.
Without loss of generality, we assume 
$|\xi_{12}| \geq |\xi_{23}| \geq |\xi_{31}|$.
Then, by the triangle inequality, we have
\begin{align}
 |\xi_{12}|\ges |\xi|.
 \label{Q1}
\end{align}

\medskip
\noindent
$\bullet$ {\bf Case 4:}  $|\xi|, |\xi_j|\gg 1$ and $|\xi_{31}|\leq |\xi_{23}| \leq 1$. 
\\
\indent
 Arguing as in Case 2 of the proof of Lemma \ref{LEM:KdV1}, 
 we have
\begin{align}
|\xi_1|\sim|\xi_2|\sim|\xi_3|\sim|\xi|\gg1
\label{mk8}
\end{align}

\noindent
 in this case.
 In particular, we have
\begin{align}
m_{j}(\bar \xi) \sim |\xi|^\frac{1}{2}
\sim m(\bar \xi)
\label{mk9}
\end{align}

\noindent
for any $j \in \{1, 2, 3\}$,
where $m(\bar \xi)$ is as in \eqref{K3c}.
Let 
$\zeta_1 = \xi_{23}$, $\zeta_2 = \xi_{31}$, and $\zeta_3 = \xi_{12}$ as before.

We prove \eqref{mk2} with $s = \frac 14$ in Subcases 4.a and 4.b,
while Subcase 4.c requires $s > \frac 14$.

\smallskip

\noindent
{\bf Subcase 4.a:}
$|\xi| \les M$.
\\
\indent
Let  $s = \frac 14$.
By H\"older's inequality with $|\zeta_2| \leq |\zeta_1| \leq 1$, 
we have 
\begin{align*}
\text{LHS of }\eqref{mk2}
& \les 
\|v_4\|_{L^1_\xi}
\sum_{\substack{1\ll N \les M \\\text{dyadic}}}
N^\frac{1}{2}
\sup_{|\xi|\sim N}
\intt_{|\zeta_2| \leq |\zeta_1|\leq 1}
 \prod_{k \in A_j}\jb{\xi_k}^\frac{1}{4} \notag \\
&  \hphantom{XXXXXXXX}
\times v_1(\xi - \zeta_1) v_2 (\xi - \zeta_2) v_3 ( \zeta_1 + \zeta_2 - \xi) d \zeta_1 d\zeta_2 
\notag\\
& \les M^{\frac{1}{2}}
\|v_j\|_{\F L^{\infty}}
\bigg(  \prod_{k \in A_j}
\|v_k\|_{H^\frac{1}{4}}\bigg)
\|v_4\|_{\F L^1}.
\end{align*}

In the following, we assume that $|\xi| \gg M$.

\noindent
{\bf Subcase 4.b:}
$|\zeta_2 |\leq |\zeta_1|\les \frac{M^\frac{1}{2}}{|\xi|^\frac{1}{2}}$.
\\
\indent
Let  $s = \frac 14$.
By H\"older's inequality with \eqref{mk8} and \eqref{mk9}, 
we have 
\begin{align*}
\text{LHS of }\eqref{mk2}
& \les 
\|v_4\|_{L^1_\xi}
\sum_{\substack{N\gg M\\\text{dyadic}}}
N^\frac{1}{2}
\sup_{|\xi|\sim N}
\intt_{|\zeta_2| \leq |\zeta_1|\les \frac{M^\frac{1}{2}}{N^\frac{1}{2}}}
 \prod_{k \in A_j}\jb{\xi_k}^\frac{1}{4} \notag \\
&  \hphantom{XXXXXXXX}
 \times v_1(\xi - \zeta_1) v_2 (\xi - \zeta_2) v_3 ( \zeta_1 + \zeta_2 - \xi) d \zeta_1 d\zeta_2 
\notag\\
& \les M^\frac{1}{2}\|v_j \|_{L^\infty_\xi} \|v_4\|_{L^1_\xi}
\sum_{\substack{N\gg M\\\text{dyadic}}}
\prod_{k \in A_j} \|\P_N v_k\|_{H^\frac{1}{4}} \notag\\
& \les M^{\frac{1}{2}}
\|v_j\|_{\F L^{\infty}}
\bigg(  \prod_{k \in A_j}
\|v_k\|_{H^\frac{1}{4}}\bigg)
\|v_4\|_{\F L^1},
\end{align*}

\noindent
where we used Cauchy-Schwarz inequality (in $N$) in the last step.

\smallskip

\noindent
{\bf Subcase 4.c:} $ \frac{M^\frac{1}{2}}{|\xi|^\frac{1}{2}}
\ll |\zeta_1|$. 

In this case,  we use $m_{s,j}(\bar \xi)\sim|\xi|^{1-2s}$.
By viewing $\Psi$ as a function of $\xi_{3}$
for fixed  $\xi$ and $\zeta_1$, 
we have
 $|\dd_{\xi_{3}}\Psi(\bar \xi)|\sim |\xi_{23}||\xi_{2 - 3}|\sim|\zeta_1||\xi|$, 
since  $|\xi_{2-3}| = |2\xi_3 - \zeta_{1}|\sim |\xi|$.
Hence, we have
\begin{align*}
\text{LHS of } \eqref{mk4} 
& \les \sup_{\xi}
\bigg(
\intt_{ \frac{M^\frac{1}{2}}{|\xi|^\frac{1}{2}} \ll |\zeta_1|\leq 1}  
\ind_{|\Psi(\bar{\xi})-\al |\leq M}
\cdot  |\xi|^{2-4s}
d\xi_3 d\zeta_1 \bigg)^\frac{1}{2} 
\\
& \les M^\frac{1}{2}\sup_{\xi} |\xi|^{\frac{1}{2} - 2s}
\bigg(
\intt_{ \frac{M^\frac{1}{2}}{|\xi|^\frac{1}{2}} \ll |\zeta_1|\leq 1}  
 \frac{1}{|\zeta_1|}
 d\zeta_1 \bigg)^\frac{1}{2} \les M^\frac{1}{2},
\end{align*}

\noi
provided that  $s>\frac{1}{4}$.

\begin{remark}\rm
When  $|\zeta_1| \geq |\zeta_2| \gg \frac{M^\frac{1}{2}}{|\xi|^\frac{1}{2}}$, 
it follows from \eqref{Psi} and \eqref{Q1} that 
\begin{align}
|\zeta_1| \les \frac{|\al|+M}{|\xi| |\zeta_2|}
\ll \frac{|\al|+M}{M}\cdot \frac{M^\frac{1}{2}}{|\xi|^\frac{1}{2}}
\label{mk10}
\end{align}

\noindent
under the condition $|\Psi(\bar{\xi})-\al |\leq M$.
In particular, \eqref{mk10} with  $|\zeta_1|  \gg \frac{M^\frac{1}{2}}{|\xi|^\frac{1}{2}}$
implies that $|\al | \gg M$.
Then, in view of 
\eqref{mk9}, 
the desired estimate \eqref{mk4} with $s = \frac 14$
follows from Subcase 2.b in the proof of Lemma \ref{LEM:KdV1}.

\end{remark}

\medskip

\noindent
$\bullet$ {\bf Case 5:}  $|\xi|, |\xi_j|\gg 1$ and $|\xi_{31}| \leq 1 < |\xi_{23}|\leq |\xi_{12}|$. 
\\
\indent
In this case, we have 
 \begin{align}
 \jb{\xi_1} \sim \jb{\xi_3}
 \qquad\text{and}
 \qquad
  \jb{\xi_2}\sim \jb{\xi}.
 \label{Q2}
  \end{align}
  
  \noi
First, we consider the case $j = 1$.
We have
\begin{align}
m_{s, 1}(\bar \xi) \sim \frac{|\xi|^\frac{1}{4}|\xi_1|^\frac{3}{4}}{\jb{\xi_2}^{s}\jb{\xi_3}^{s}}
\les |\xi|^{\frac{1}{4} - s}|\xi_1|^{\frac{3}{4}-s}.
\label{mk11}
\end{align}

\noi
In the following, we take $s > \frac 14$.

\smallskip

\noindent
{\bf Subcase 5.a:}
$|\xi|\ges |\xi_1| \gg 1$.
\\
\indent
Arguing as in Subcase 3.a in the proof of Lemma \ref{LEM:KdV1}, 
we see that 
$\zeta_2$ belongs to an interval $I = I(\zeta_1, \xi)$ of length
\begin{align} 
| I(\zeta_1, \xi)|\les \frac{ M}{|\zeta_1||\xi|}
\label{mk12}
\end{align}

\noindent
for each fixed $\xi$ and $\zeta_1$
and hence for each fixed $\xi$ and $\xi_1 = \xi - \zeta_1$.
Then, 
using a variant of~\eqref{mk3}
with \eqref{mk11}, \eqref{mk12}, 
and  $|\zeta_1| = |\xi - \xi_1| \les |\xi|$, 
we have
\begin{align}
\text{LHS of }\eqref{mk2}
& \leq  \bigg\|
\int_{1\leq |\zeta_1| \les |\xi| }
\int_{|\zeta_2|\leq 1}
\ind_{|\Psi(\bar{\xi})-\al |\leq M}\cdot m_{s, 1}(\bar \xi) 
 \prod_{k \in A_1} \jb{\xi_k}^s
\notag\\
& \hphantom{XXXXXXX}
\times v_1(\xi - \zeta_1) v_2(\xi - \zeta_2) v_3 ( -\xi + \zeta_1  + \zeta_2)
d\zeta_2 d\zeta_1
\bigg\|_{L^\infty_{\xi}}\|v_4 \|_{L^1_\xi}\notag\\
& \les  \sup_{\xi}
\bigg(
\int_{1\leq |\zeta_1| \les |\xi| }
\int_{\zeta_2 \in I(\zeta_1, \xi)}
\ind_{|\Psi(\bar{\xi})-\al |\leq M} \cdot 
|\xi|^{\frac{1}{2} - 2s}|\xi_1|^{\frac{3}{2}-2s}
 d\zeta_2 d\zeta_1 \bigg)^\frac{1}{2}
\notag\\
& \hphantom{XXXXXXX}
\times 
\|v_1\|_{L^\infty_\xi }\bigg(\prod_{k \in A_1} \| v_k\|_{H^s}\bigg)\|v_4\|_{L^1_\xi }\notag\\
& \les M^\frac{1}{2} \sup_{\xi}
\bigg(
\int_{1\leq |\zeta_1| \les |\xi| }
\frac{|\xi|^{1 - 4s}}{|\zeta_1| }
 d\zeta_1 \bigg)^\frac{1}{2}
\|v_1\|_{L^\infty_\xi }\bigg(\prod_{k \in A_1}\| v_k\|_{H^s}\bigg)\|v_4\|_{L^1_\xi }\notag\\
& \les M^{\frac{1}{2}}
\|v_1\|_{\F L^{\infty}}
\bigg(  \prod_{k \in A_1}
\|v_k\|_{H^s}\bigg)
\|v_4\|_{\F L^1},
\label{mk13}
\end{align}

\noi
provided that $\frac 14 < s \leq \frac 34$.
For $s > \frac 34$, 
we simply use
$|\xi_1|^{\frac{3}{2}-2s} \les 1$ and repeat the computation in \eqref{mk13}.

\medskip
\noindent
{\bf Subcase 5.b:} 
$|\xi_1| \gg |\xi|$. 

By viewing $\Psi$ as a function of $\xi_{2}$
for fixed  $\xi$ and $\xi_1$, 
we have
 $|\dd_{\xi_2}\Psi(\bar \xi)|\sim |\xi_{23}||\xi_{2 - 3}|\sim|\xi_1|^2$
 thanks to \eqref{Q2}.
 Hence, with \eqref{mk11}, we have
 \begin{align*}
\text{LHS of }\eqref{mk4}
& \les 
\sup_{\xi}
\bigg(\intt_{|\xi_1|\gg|\xi|} 
\ind_{|\Psi(\bar{\xi})-\al |\leq M} \cdot|\xi_1|^{\frac{3}{2}-2s} d\xi_{2} d\xi_{1} \bigg)^\frac{1}{2}\\
& \les
M^\frac{1}{2}\sup_{\xi}
\bigg(\intt_{|\xi_1|\gg 1 } 
 |\xi_1|^{-\frac{1}{2}-2s} d\xi_{1} \bigg)^\frac{1}{2}\les M^\frac{1}{2},
	\end{align*}

\noi
provided that  $s>\frac{1}{4}$.

\medskip
Next, we consider the case $j = 2$.
It follows from  \eqref{MK} and \eqref{Q2} that 
\begin{align*}
m_{s,2}(\bar \xi) \sim \frac{|\xi|^\frac{1}{4}|\xi_2|^\frac{3}{4}}{\jb{\xi_1}^{s}\jb{\xi_3}^{s}}
\sim \frac{|\xi|}{\jb{\xi_1}^{2s}}
\end{align*}


\noindent
{\bf Subcase 5.c:}
$|\xi_2| \les |\xi_1|$.
\\
\indent
Proceeding as in Subcase 5.a
with $|\zeta_1| = |\xi - \xi_1| \les |\xi_1|$, 
 we obtain
\begin{align*}
\text{LHS of }\eqref{mk2}
& \leq  \bigg\|
\int_{1\leq |\zeta_1| \les |\xi_1| }
\int_{|\zeta_2|\leq 1}
\ind_{|\Psi(\bar{\xi})-\al |\leq M}\cdot m_{s, 2}(\bar \xi) 
 \prod_{k \in A_2} \jb{\xi_k}^s
\notag\\
& \hphantom{XXXXXXX}
\times v_1(\xi - \zeta_1) v_2(\xi - \zeta_2) v_3 ( -\xi + \zeta_1  + \zeta_2)
d\zeta_2 d\zeta_1
\bigg\|_{L^\infty_{\xi}}\|v_4 \|_{L^1_\xi}\notag\\
& \les  \sup_{\xi}
\bigg(
\int_{1\leq |\zeta_1| \les |\xi_1| }
\int_{\zeta_2 \in I(\zeta_1, \xi)}
\ind_{|\Psi(\bar{\xi})-\al |\leq M} \cdot 
|\xi|^2|\xi_1|^{-4s}
 d\zeta_2 d\zeta_1 \bigg)^\frac{1}{2}
\notag\\
& \hphantom{XXXXXXX}
\times 
\|v_2\|_{L^\infty_\xi }\bigg(\prod_{k \in A_2} \| v_k\|_{H^s}\bigg)\|v_4\|_{L^1_\xi }\notag\\
& \les M^\frac{1}{2} \sup_{\xi}
\bigg(
\int
\frac{1}{\jb{\zeta_1}\jb{\xi - \zeta_1}^{4s-1} }
 d\zeta_1 \bigg)^\frac{1}{2}
\|v_2\|_{L^\infty_\xi }\bigg(\prod_{k \in A_2}\| v_k\|_{H^s}\bigg)\|v_4\|_{L^1_\xi }\notag\\
& \les M^{\frac{1}{2}}
\|v_2\|_{\F L^{\infty}}
\bigg(  \prod_{k \in A_2}
\|v_k\|_{H^s}\bigg)
\|v_4\|_{\F L^1},
\end{align*}

\noi
provided that $s> \frac 14$.

\medskip

\noindent
{\bf Subcase 5.d:}
$|\xi_2| \gg |\xi_1|$.
\\
\indent
By viewing $\Psi$ as a function of $\xi_{1}$
for fixed  $\xi$ and $\xi_3$, 
we have
 $|\dd_{\xi_1}\Psi(\bar \xi)|\sim |\xi_{12}||\xi_{1-2 }|\sim|\xi|^2$.
Hence, we have
\begin{align*}
\text{LHS of }\eqref{mk4}
& \les 
\sup_{\xi}
\bigg(\int_{ |\xi_3|\ll |\xi|} 
\ind_{|\Psi(\bar{\xi})-\al |\leq M} \frac{|\xi|^2}{\jb{\xi_3}^{4s}}d\xi_1 d\xi_3 \bigg)^\frac{1}{2}\\
& \les 
 M^{\frac{1}{2}}
\sup_{\xi}
\bigg(\int
\frac{1}{\jb{\xi_3}^{4s}} d\xi_3 \bigg)^\frac{1}{2}
\les M^{\frac{1}{2}},
\end{align*}

\noi
provided that  $s>\frac{1}{4}$.

\medskip
Lastly, we consider the case $j = 3$.

\noindent
{\bf Subcase 5.e:}
 $|\xi_3| \les |\xi|$.
\\
\indent
If $|\xi_{23}|\ges |\xi|$, 
then
this case follows from Subcase 5.a
by switching $1 \leftrightarrow3$.
Now, 
suppose that 
$|\xi_{23}|\ll |\xi|$.
Then, it follows from \eqref{Q2} that 
 $\jb{\xi_1} \sim \jb{\xi_2}
\sim \jb{\xi_3} \sim \jb{\xi}$.
In particular, we have 
$m_{s, 3}(\bar \xi) \sim |\xi|^{1-2s}$.
Hence, we can apply 
 Subcase 5.a
by replacing  $m_1(\bar \xi)$ with $m_3(\bar \xi)$
(without switching $1\leftrightarrow3$).

\medskip

\noindent
{\bf Subcase 5.f:}
$|\xi_3| \gg |\xi|$.
\\
\indent
This case follows from Subcase 5.b 
by switching $1 \leftrightarrow3$.

\medskip

\noindent
$\bullet$ {\bf Case 6:}  $|\xi|, |\xi_j|\gg 1$ and $ |\xi_{12}|,  |\xi_{23}|,  |\xi_{31}| > 1$. 
\\
\indent
From \eqref{K10a} with $\xi = \xi_1 + \xi_2 + \xi_3$, we have
\begin{align}
|\al| + M \ges \max( |\xi|, |\xi_{1}|,  |\xi_{2}|,  |\xi_{3}|),
\label{Q3}
\end{align}

\noi
which allows us to prove \eqref{mk4} with $s = \frac{1}{4}$
in this case.
In the following, the size relation of 
$ |\xi_{12}|,  |\xi_{23}|,  |\xi_{31}|$ does not play any role.
Hence, without loss of generality, assume that 
$|\xi_1|\ge |\xi_2|\ge |\xi_3|$.
Recall that  $A_j = \{1, 2, 3\} \setminus \{j\} = \{k_1, k_2\}$.

\medskip

\noindent
{\bf Subcase 6.a:}
$|\xi_{1}| \sim |\xi| \gg |\xi_{2}|\ge  |\xi_{3}|$.
\\
\indent
Let  $s = \frac 14$.
By viewing $\Psi$ as a function of $\xi_{2}$
for fixed  $\xi$ and $\xi_3$, 
we have
 $|\dd_{\xi_2}\Psi(\bar \xi)|\sim |\xi_{12}||\xi_{1 -2}| \sim |\xi|^2$.
Hence, with \eqref{Q3}, we have
\begin{align*}
\text{LHS of } \eqref{mk4} 
& \les \sup_{\xi}
\bigg(\intt_{\xi = \xi_1 + \xi_2 + \xi_3} 
\ind_{|\Psi(\bar{\xi})-\al |\leq M}\cdot
\frac{|\xi|^\frac{1}{2}|\xi_{j}|^\frac{3}{2}}{\jb{\xi_{k_1}}^\frac{1}{2}\jb{\xi_{k_2}}^\frac{1}{2}}d\xi_{2} d\xi_{3} \bigg)^\frac{1}{2} 
\\
& \les \sup_{\xi}
\bigg(\intt_{\xi = \xi_1 + \xi_2 + \xi_3} 
\ind_{|\Psi(\bar{\xi})-\al |\leq M}\cdot
\frac{|\xi|^2}{\jb{\xi_{3}}}d\xi_{2} d\xi_{3} \bigg)^\frac{1}{2} 
\\
& \les M^\frac{1}{2}
\bigg(\int_{|\xi_{3}|\ll |\xi|} 
\frac{1}{\jb{\xi_{3}}}d\xi_{3} \bigg)^\frac{1}{2} 
\les M^\frac{1}{2}(\log |\xi|)^\frac{1}{2}
\les \jb{\al }^{0+} M^{\frac{1}{2}+}.
\end{align*}

\medskip

\noindent
{\bf Subcase 6.b:}
$|\xi_{1}| \sim |\xi_{2}|  \gg |\xi| \gg  |\xi_3|$.
\\
\indent
Let  $s = \frac 14$.
We have $|\xi_{12}|=|\xi-\xi_3|\sim |\xi|$
 and $|\xi_{1-2}|\sim|\xi_1|$,
 since $\xi_1$ and $\xi_2$ have opposite signs
 in this case.
By viewing $\Psi$ as a function of $\xi_{1}$
for fixed  $\xi$ and $\xi_3$, 
we have
 $|\dd_{\xi_1}\Psi(\bar \xi)|\sim |\xi_{12}||\xi_{1 -2}| \sim |\xi||\xi_1|$.
Then, noting that
\[ |m_{s,j}(\bar \xi)| \les \frac{|\xi|^\frac{1}{4}|\xi_1|^{\frac{1}{2}}}{\jb{\xi_3}^s},\]

\noi
it follows from \eqref{Q3} that 
\begin{align*}
\text{LHS of } \eqref{mk4} 
& \les \sup_{\xi}
\bigg(\sum_{\substack{|\xi| \ll N \les |\al| + M\\\text{dyadic}}}
\intt_{\substack{|\xi_1|\sim N\\|\xi_3|\ll |\xi|}}
\ind_{|\Psi(\bar{\xi})-\al |\leq M}\cdot
\frac{|\xi|^{\frac{1}{2}}N}{\jb{\xi_3}^{\frac 12}}d\xi_1 d\xi_3 \bigg)^\frac{1}{2} 
\\
& \les M^\frac{1}{2}
\sup_{\xi} 
\bigg(\sum_{\substack{|\xi| \ll N \les |\al| + M\\\text{dyadic}}}\int_{|\xi_3|\ll |\xi|}
\frac{1}{|\xi|^\frac{1}{2}\jb{\xi_3}^{\frac 12}}d\xi_3 \bigg)^\frac{1}{2} 
\les 
\jb{\al }^{0+} M^{\frac{1}{2}+}. 
\end{align*}

\medskip
\noindent
\textbf{Subcase 6.c:} $|\xi_1|\sim|\xi_2|\sim|\xi|\gg|\xi_3|$.

In this case, the desired estimate holds
for $s=\frac{1}{4}$ but 
it requires an extra factor of  $\max\{|\alpha|,M\}^\frac{1}{12}$.
We have 
$|\xi_{12}|=|\xi-\xi_3|\sim|\xi|$, $|\xi_{23}|\sim |\xi|$, and $|\xi_{13}|\sim|\xi|$.
Hence,  the condition $|\Psi(\bar \xi)|\leq |\alpha|+M$ with \eqref{Psi}
implies  $|\xi|^\frac{1}{4}\lesssim\max\{|\alpha|,M\}^\frac{1}{12}$.
In particular, we have 
\[|m_{j}(\bar \xi)| \les \frac{|\xi|^{\frac{3}{4}}}{\jb{\xi_3}^{\frac{1}{4}}}
\les \max\{|\alpha|,M\}^\frac{1}{12}|\xi|^\frac{1}{2}\]

\noi
for any $j \in \{1, 2, 3\}$.
By viewing $\Psi$ as a function of $\xi_{1}$
for fixed  $\xi$ and $\xi_2$, 
we have
 $|\dd_{\xi_1}\Psi(\bar \xi)|\sim |\xi_{13}||\xi_{1 -3}| \sim |\xi|^2$.
 Hence, we have
\noindent
\begin{align*}
\text{LHS of } \eqref{mk4} 
& \les \max\{|\alpha|,M\}^\frac{1}{12}\sup_{\xi}
\bigg(\intt_{\xi = \xi_1 + \xi_2 + \xi_3} 
\ind_{|\Psi(\bar{\xi})-\al |\leq M}\cdot
|\xi|d\xi_1 d\xi_2 \bigg)^\frac{1}{2} 
\\
& \les \max\{|\alpha|,M\}^\frac{1}{12}M^\frac{1}{2}
\sup_{\xi} 
\bigg(\int_{|\xi_2|\sim |\xi|}
\frac{1}{|\xi|}d\xi_2 \bigg)^\frac{1}{2} 
\les \max\{|\alpha|,M\}^\frac{1}{12}M^\frac{1}{2}.
\end{align*}

\medskip

In the following three subcases, we deal with the case $|\xi_1|, |\xi_2| , |\xi_3| \ges |\xi|  $.

\smallskip
\noindent
\textbf{Subcase 6.d:} $|\xi_1|\sim |\xi_2|\gg|\xi_3|\gtrsim |\xi|$.
\\
\indent
Let  $s = \frac 14$.
By viewing $\Psi$ as a function of $\xi_{2}$
for fixed  $\xi$ and $\xi_1$, 
we have
 $|\dd_{\xi_2}\Psi(\bar \xi)|\sim |\xi_{23}||\xi_{2 -3}| \sim |\xi_2|^2$.
Hence, with 
$m_{s,j}(\bar \xi)\les |\xi_1|^{\frac{1}{2}}$
and \eqref{Q3}, 
we have
\begin{align*}
\text{LHS of } \eqref{mk4} 
& \les \sup_{\xi}
\bigg(\sum_{\substack{|\xi|\ll N \les|\al|+M\\\text{dyadic}}}\int_{|\xi_1|\sim|\xi_2|\sim N}
\ind_{|\Psi(\bar{\xi})-\al |\leq M}\cdot
N d\xi_2 d\xi_1 \bigg)^\frac{1}{2} 
\\
& \les M^\frac{1}{2}
\sup_{\xi} 
\bigg(\sum_{\substack{|\xi|\ll N |\al| + M \\\text{dyadic}}}\int_{|\xi_1|\sim N}
\frac{1}{N}d\xi_1 \bigg)^\frac{1}{2} 
\les 
\jb{\al }^{0+} M^{\frac{1}{2}+}. 
\end{align*}

\medskip
\noindent
\textbf{Subcase 6.e:} $|\xi_1|\sim |\xi_2|\sim|\xi_3|\gg |\xi|$.
\\
\indent
This case (with $s = \frac 14$) follows from Subcase 3.b.

\medskip
\noindent
\textbf{Subcase 6.f} $|\xi_1|\sim |\xi_2|\sim|\xi_3|\sim |\xi|$
\\
\indent 
In this case, we have $m_{s, j}(\bar \xi)\sim|\xi|^{1 - 2s}$.
In the following, we do not use the size relation between $|\xi_1|$, $|\xi_2|$, and $|\xi_3|$.

\smallskip

\noi
$\circ$ \underline{Subsubcase 6.f.i:} 
 $\min\{|\zeta_1|,|\zeta_2|,|\zeta_3|\}\ll |\xi|$.
\\
\indent
Let $s = \frac 14$.
Without loss of generality, assume  $|\zeta_3|=|\xi_{12}|\ll|\xi|$. 
Then, we have
$|\xi_{1-2}| = |2\xi_1 - \xi_{12}| \sim|\xi|$.
By viewing $\Psi$ as a function of $\xi_2$
for fixed  $\xi$ and $\xi_3$ (or $\zeta_3=\xi-\xi_3$), 
we have
 $|\dd_{\xi_2}\Psi(\bar \xi)|\sim |\zeta_3||\xi_{1-2}| \sim |\zeta_3||\xi|$.
Hence, with \eqref{K10a}, we have
\begin{align*}
\text{LHS of } \eqref{mk4} 
& \les \sup_{\xi}
\bigg(\intt_{\xi = \xi_1 + \xi_2 + \xi_3} 
\ind_{|\Psi(\bar{\xi})-\al |\leq M}\cdot
|\xi|d\xi_2 d\zeta_3 \bigg)^\frac{1}{2} 
\\
& \les M^\frac{1}{2}
\sup_{\xi} 
\bigg(\int_{1\leq|\zeta_3|\les|\alpha|+M}\frac{1}{|\zeta_3|}d\zeta_3 \bigg)^\frac{1}{2} 
\les \jb{\alpha}^{0+} M^{\frac{1}{2}+}.
\end{align*}

\medskip

In the following, we consider the case:  $\min\{|\zeta_1|,|\zeta_2|,|\zeta_3|\}\sim |\xi|$.
By the triangle inequality, we have
\[
 \max(|\xi - \xi_1|, |\xi - \xi_2|, |\xi - \xi_3|)
\ge |3\xi - \xi_{123}| = 2  |\xi|.\]

\noi
Hence, we have $|\zeta_1|\sim |\zeta_2|\sim |\zeta_3|\sim |\xi|$.

\smallskip
\noi
$\circ$
\underline{Subsubcase 6.f.ii:}
$|\zeta_1|\sim |\zeta_2|\sim |\zeta_3|\sim |\xi|$
 and $\max\{|\xi_{1-2}|,|\xi_{2-3}|,|\xi_{3-1}|\}\sim|\xi|$.
\\
\indent
Let $s = \frac 14$.
Without loss of generality, assume that $|\xi_{1-2}|\sim|\xi|$.
By viewing $\Psi$ as a function of $\xi_2$
for  fixed  $\xi$ and $\xi_3$, 
we have $|\partial_{\xi_2}\Psi(\bar \xi)|\sim|\zeta_3||\xi_{1-2}|\sim|\xi|^2$.
Hence we have
\begin{align*}
\text{LHS of } \eqref{mk4} 
& \les \sup_{\xi}
\bigg(\intt_{\xi = \xi_1 + \xi_2 + \xi_3} 
\ind_{|\Psi(\bar{\xi})-\al |\leq M}\cdot
|\xi|d\xi_2 d\xi_3 \bigg)^\frac{1}{2} 
\\
& \les M^\frac{1}{2}
\sup_{\xi} 
\bigg(\int_{|\xi_3|\sim|\xi|}\frac{1}{|\xi|}d\xi_3 \bigg)^\frac{1}{2} 
\les M^{\frac{1}{2}}.
\end{align*}

\medskip

In the three subsubcases, we assume 
that 
$|\xi_{1-2}|\geq |\xi_{2-3}|\geq |\xi_{3-1}|$
without loss of generality.

\smallskip
\noi
$\circ$
\underline{Subsubcase 6.f.iii:}
$|\zeta_1|\sim |\zeta_2|\sim |\zeta_3|\sim |\xi|$ and 
$1\les|\xi_{1-2}|\ll|\xi|$.

Let  $s = \frac 14$.
For fixed $\xi$ and a  dyadic number $1\les N\ll|\xi|$, 
suppose that  $|\xi_{1-2}|\sim N$.
Then, by writing $\xi=3\xi_3+\xi_{1-2}+2\xi_{2-3}$, 
we see that  $\xi_3$ is contained in an interval $I_3(\xi,N)$ of length $\les N$.
Moreover, for fixed $\xi$ and $\xi_3$, we have
 $|\dd_{\xi_2}\Psi(\bar \xi)|\sim|\zeta_3||\xi_{1-2}| \sim|\xi|N$.
 Hence, we obtain
\begin{align*}
\text{LHS of } \eqref{mk4} 
& \les \sup_{\xi}
\bigg(\sum_{1\les N\ll|\xi|}\intt_{\xi = \xi_1 + \xi_2 + \xi_3} 
\ind_{|\Psi(\bar{\xi})-\al |\leq M}\cdot\ind_{|\xi_{1-2}|\sim N}
|\xi|d\xi_2 d\xi_3 \bigg)^\frac{1}{2} 
\\
& \les  M^{\frac{1}{2}}
\sup_{\xi} 
\bigg(\sum_{1\les N\ll|\xi|}\int_{\xi_3\in I_3(\xi,N)}\frac{1}{ N}d\xi_3 \bigg)^\frac{1}{2} \\
& \les \jb{\al}^{0+}M^{\frac{1}{2}+}
\sup_{\xi} 
\bigg(\sum_{1\les N\ll|\xi|}\frac{1}{|\xi|^{0+}} \bigg)^\frac{1}{2} 
\les \jb{\al}^{0+} M^{\frac{1}{2}}, 
\end{align*}

\noi
where we used \eqref{K10a} in the penultimate step.

%
%

\smallskip

\noi
$\circ$
\underline{Subsubcase 6.f.iv:} 
$|\zeta_1|\sim |\zeta_2|\sim |\zeta_3|\sim |\xi|$, 
$|\xi_{1-2}| \ll 1$, and 
$|\xi_{3-1}|\les \frac{M}{|\xi|}$.

Let $s = \frac 14$.
Arguing as in Subsubcase 6.f.iii, 
we see that for fixed $\xi$, 
 $\xi_3$ is  contained in an interval $I(\xi)$ of length $O(1)$.
Hence we have
\begin{align*}
\text{LHS of } \eqref{mk4} 
& \les \sup_{\xi}
\bigg(\intt_{\substack{\xi_3\in I(\xi)\\ |\xi_1-\xi_3|\lesssim\frac{M}{|\xi|}} }
\ind_{|\Psi(\bar{\xi})-\al |\leq M}\cdot
|\xi|d\xi_1 d\xi_3 \bigg)^\frac{1}{2} 
\\
& \les M^\frac{1}{2}
\sup_{\xi} 
\bigg(\int_{\xi_3\in I(\xi)}1d\xi_3 \bigg)^\frac{1}{2} \les M^\frac{1}{2}.
\end{align*}

\medskip

\noi
$\circ$
\underline{Subsubcase 6.f.v:}
$|\zeta_1|\sim |\zeta_2|\sim |\zeta_3|\sim |\xi|$, 
and 
$ \frac{M}{|\xi|}\ll |\xi_{3-1}|
\leq |\xi_{1-2}| \ll 1$.

Let  $s = \frac 14$.
For fixed $\xi$ and a dyadic number $1\ll N\ll\frac{|\xi|}{M}$, 
suppose that 
$|\xi_{1-2}|\sim N^{-1}$.
Then, 
arguing as in Subsubcase 6.f.iii, 
we see that 
 $\xi_3$ is contained in an interval $\wt{I}_3(\xi,N)$ of  length $\les N^{-1}$.
Moreover, for fixed $\xi$ and $\xi_3$, 
we have $|\dd_{\xi_2}\Psi(\bar \xi)|\sim|\zeta_3||\xi_{1-2}| \sim|\xi|N^{-1}$.
Hence, with \eqref{K10a}, we obtain
\begin{align*}
\text{LHS of } \eqref{mk4} 
& \les \sup_{\xi}
\bigg(\sum_{1\ll N\ll \frac{|\xi|}{M}}\intt_{\xi = \xi_1 + \xi_2 + \xi_3} 
\ind_{|\Psi(\bar{\xi})-\al |\leq M}\cdot\ind_{|\xi_{1-2}|\sim N^{-1}}
|\xi| d\xi_2 d\xi_3 \bigg)^\frac{1}{2} 
\\
& \les \jb{\al}^{0+} M^{\frac{1}{2}+}
\sup_{\xi} 
\bigg(\sum_{1\ll N\ll \frac{|\xi|}{M}}\int_{\xi_3\in \widetilde{I}_3(\xi,N)}
\frac{N}{|\xi|^{0+}} d\xi_3 \bigg)^\frac{1}{2} \\
& \les \jb{\al}^{0+} M^{\frac{1}{2}+}
\sup_{\xi} 
\bigg(\sum_{N\gg 1}
\frac{1}{M^{0+}N^{0+}} \bigg)^\frac{1}{2} 
\les \jb{\al}^{0+} M^{\frac{1}{2}+}. 
\end{align*}

\noi
This completes the proof of Lemma \ref{LEM:mk1}.
\end{proof}

\smallskip

\noi
{\bf Concluding remark:}
In Section \ref{SEC:3}, we presented the full details
of the normal form reductions
since this is the first paper, where we handle multilinear estimates
by successive applications of the trilinear localized modulation estimate.
The essential part for establishing an a priori estimate
in $L^2(\R)$ for the cubic NLS and in $H^\frac{1}{4}(\R)$
for the mKdV  
appears in  Subsection~\ref{SUBSEC:3.3}, 
where we applied
 the localized modulation estimates 
 from Section \ref{SEC:2}.
In Section~\ref{SEC:4}, we also needed to prove another localized modulation estimate 
(in the weak norm: Lemmas \ref{LEM:NLS3} and \ref{LEM:mk1})
for justifying the formal computations in Section \ref{SEC:3}, 
where an extra complication was introduced for the mKdV problem 
due to the derivative nonlinearity.
In essence, 
our  method  allows one to 
 reduce the  entire problem of proving unconditional well-posedness
to simply proving two basic trilinear estimates
(i.e.~localized modulation estimates in the strong norm
and in the weak norm: Lemmas~\ref{LEM:NLS1} and \ref{LEM:NLS3}
for the cubic NLS 
and Lemma~\ref{LEM:KdV1} and \ref{LEM:mk1}
for the mKdV).
This reduction is the main novelty of the paper
and such a reduction provides a significant simplification 
in studying unconditional well-posedness
for various dispersive PDEs on $\R^d$ and $\T^d$.

\begin{ackno}\rm
S.K.~was partially supported by National Research Foundation of Korea (grant 
NRF-2015R1D1A1A01058832).
T.O.~was supported by the European Research Council (grant no.~637995 ``ProbDynDispEq'').
H.Y.~was partially supported by National Research Foundation of Korea (grant  2012-H1A2A1049375).
S.K. would like to express his gratitude to  the School of Mathematics at the University of Edinburgh 
for its hospitality during his visit in the academic year 2015 - 2016.
The authors would like to thank Nobu Kishimoto for a helpful discussion
and Didier Pilod for a helpful comment.
They are also grateful to Razvan Mosincat for careful proofreading.
\end{ackno}

\end{document}